\documentclass[dvips]{amsart}

\usepackage{amsmath,amscd,amssymb}
\usepackage{bbm}
\usepackage[all,knot]{xy}

\usepackage{amsthm}
\numberwithin{equation}{section}

\title[Further results on the structure of (co)ends]
{Further results on the structure of (co)ends in finite tensor categories}
\author[K.~Shimizu]{Kenichi Shimizu}
\email{kshimizu@shibaura-it.ac.jp}
\address{Department of Mathematical Sciences \\
  Shibaura Institute of Technology \\
  307 Fukasaku, Minuma-ku, Saitama-shi, Saitama 337-8570, Japan.}
\date{}

\newtheorem{counter}{}[section]

\theoremstyle{definition}
\newtheorem{definition}         [counter]{Definition}

\theoremstyle{plain}
\newtheorem{lemma}              [counter]{Lemma}

\newtheorem{proposition}        [counter]{Proposition}
\newtheorem{theorem}            [counter]{Theorem}
\newtheorem{corollary}          [counter]{Corollary}

\newtheorem*{theorem*}          {Theorem}
\theoremstyle{remark}
\newtheorem{remark}             [counter]{Remark}
\newtheorem{example}            [counter]{Example}

\newcommand{\id}{\mathrm{id}}
\newcommand{\eval}{\mathrm{ev}}
\newcommand{\coev}{\mathrm{coev}}
\newcommand{\op}{\mathrm{op}}
\newcommand{\rev}{\mathrm{rev}}
\newcommand{\unitobj}{\mathbbm{1}}

\DeclareMathOperator{\Hom}{\mathrm{Hom}}
\DeclareMathOperator{\Ext}{\mathrm{Ext}}
\DeclareMathOperator{\Tor}{\mathrm{Tor}}
\DeclareMathOperator{\Img}{\mathrm{Im}}
\DeclareMathOperator{\Ker}{\mathrm{Ker}}

\DeclareMathOperator{\End}{\mathrm{End}}

\DeclareMathOperator{\REX}{\mathrm{Rex}}
\DeclareMathOperator{\LEX}{\mathrm{Lex}}
\DeclareMathOperator{\Nat}{\mathrm{Nat}}

\DeclareMathOperator{\Irr}{\mathrm{Irr}}
\DeclareMathOperator{\Gr}{\mathrm{Gr}}
\DeclareMathOperator{\CF}{\mathrm{CF}}

\DeclareMathOperator{\radj}{\mathrm{ra}}
\DeclareMathOperator{\ladj}{\mathrm{la}}
\DeclareMathOperator{\Rey}{\mathrm{Rey}}
\DeclareMathOperator{\HH}{\mathrm{HH}}
\DeclareMathOperator{\Trace}{\mathrm{Tr}}
\newcommand{\trace}{\mathrm{tr}}
\newcommand{\ich}{\mathrm{ch}}

\newcommand{\lmod}[1]{{#1}\text{-{\sf mod}}}
\newcommand{\rmod}[1]{\text{{\sf mod}-}{#1}}
\newcommand{\bimod}[2]{{#1}\text{-{\sf mod}-}{#2}}

\DeclareMathOperator{\iHom}{\underline{\mathrm{Hom}}}
\DeclareMathOperator{\icoHom}{\underline{\mathrm{coHom}}}
\DeclareMathOperator{\ieval}{\underline{\mathrm{ev}}}
\DeclareMathOperator{\icoev}{\underline{\mathrm{coev}}}
\DeclareMathOperator{\icomp}{\underline{\mathrm{com\smash{\mathrm{p}}}}}
\DeclareMathOperator{\iHomA}{\mathfrak{a}}
\DeclareMathOperator{\iHomB}{\mathfrak{b}}
\DeclareMathOperator{\iHomC}{\mathfrak{c}}
\DeclareMathOperator{\iHomD}{\mathfrak{d}}

\usepackage{upgreek}
\newcommand{\ActRex}{\uprho}
\newcommand{\ActLex}{\uplambda}
\newcommand{\Act}{\ActRex}
\newcommand{\rad}{\mathrm{rad}}
\newcommand{\socle}{\mathrm{soc}}
\newcommand{\capital}{\mathrm{cap}}
\newcommand{\Lw}{\mathrm{Lw}}

\newcommand{\Sym}{\mathrm{SLF}}
\newcommand{\Ser}{\mathbbmss{S}}
\newcommand{\Nak}{\mathbbmss{N}}

\begin{document}

\begin{abstract}
  Let $\mathcal{C}$ be a finite tensor category, and let $\mathcal{M}$ be an exact left $\mathcal{C}$-module category. The action of $\mathcal{C}$ on $\mathcal{M}$ induces a functor $\rho: \mathcal{C} \to \mathrm{Rex}(\mathcal{M})$, where $\mathrm{Rex}(\mathcal{M})$ is the category of $k$-linear right exact endofunctors on $\mathcal{M}$. Our key observation is that $\rho$ has a right adjoint $\rho^{\mathrm{ra}}$ given by the end
  \begin{equation*}
    \rho^{\mathrm{ra}}(F) = \int_{M \in \mathcal{M}} \underline{\mathrm{Hom}}(M, M)
    \quad (F \in \mathrm{Rex}(\mathcal{M})).
  \end{equation*}
  As an application, we establish the following results: (1) We give a description of the composition of the induction functor $\mathcal{C}_{\mathcal{M}}^* \to \mathcal{Z}(\mathcal{C}_{\mathcal{M}}^*)$ and Schauenburg's equivalence $\mathcal{Z}(\mathcal{C}_{\mathcal{M}}^*) \approx \mathcal{Z}(\mathcal{C})$. (2) We introduce the space $\mathrm{CF}(\mathcal{M})$ of `class functions' of $\mathcal{M}$ and initiate the character theory for pivotal module categories. (3) We introduce a filtration for $\mathrm{CF}(\mathcal{M})$ and discuss its relation with some ring-theoretic notions, such as the Reynolds ideal and its generalizations. (4) We show that $\mathrm{Ext}_{\mathcal{C}}^{\bullet}(1, \rho^{\mathrm{ra}}(\mathrm{id}_{\mathcal{M}}))$ is isomorphic to the Hochschild cohomology of $\mathcal{M}$. As an application, we show that the modular group acts projectively on the Hochschild cohomology of a modular tensor category.
\end{abstract}

\maketitle

\section{Introduction}
\label{sec:introduction}

Let $\mathcal{C}$ be a finite tensor category. In recent study of finite tensor categories and their applications, it is important to consider the end $A = \int_{X \in \mathcal{C}} X \otimes X^*$ and the coend $L = \int^{X \in \mathcal{C}} X^* \otimes X$. The end $A$ is a categorical counterpart of the adjoint representation of a Hopf algebra. By using the end $A$, we have established the character theory and the integral theory for finite tensor categories in \cite{MR3631720} and \cite{2016arXiv170202425S}, respectively. The coend $L$, which is isomorphic to $A^*$, plays a central role in Lyubashenko's work on `non-semisimple' modular tensor categories \cite{MR1324034,MR1352517,MR1354257,MR1862634}. These results are used in recent progress of topological quantum field theory and conformal field theories \cite{2016arXiv160504448G,2017arXiv170300150G,2017arXiv170201086F}.

Since these objects are defined by the universal property, it is difficult to analyze its structure. The aim of this paper is to provide a general framework to deal with such (co)ends. Let $\mathcal{M}$ be an indecomposable exact left $\mathcal{C}$-module category in the sense of \cite{MR2119143}. We denote by $\REX(\mathcal{M})$ the category of $k$-linear right exact endofunctors on $\mathcal{M}$. The action of $\mathcal{C}$ on $\mathcal{M}$ induces a functor $\Act: \mathcal{C} \to \mathrm{Rex}(\mathcal{M})$ given by $\Act(X)(M) = X \otimes M$. Our key observation is that a right adjoint of $\Act$, say $\Act^{\radj}$, is a $k$-linear faithful exact functor such that
\begin{equation}
  \label{eq:intro-act-right-adj}
  \Act^{\radj}(F) = \int_{M \in \mathcal{M}} \iHom(M, F(M))
  \quad (F \in \REX(\mathcal{M})),
\end{equation}
where $\iHom$ is the internal Hom functor (Theorem~\ref{thm:action-adj-by-end}). The end $A$ considered at the beginning of this paper is just the case where $\mathcal{M} = \mathcal{C}$ and $F = \id_{\mathcal{C}}$. This result allows us to discuss interaction between several ends through $\Act^{\radj}$. As applications, we obtain several results on finite tensor categories and their module categories as summarized below:
\begin{enumerate}
\item Let $\mathcal{C}_{\mathcal{M}}^*$ be the dual of $\mathcal{C}$ with respect to $\mathcal{M}$. We give an explicit description of the composition of the induction functor $\mathcal{C}_{\mathcal{M}}^* \to \mathcal{Z}(\mathcal{C}_{\mathcal{M}}^*)$ and Schauenburg's equivalence $\mathcal{Z}(\mathcal{C}_{\mathcal{M}}^*) \approx \mathcal{Z}(\mathcal{C})$. We note that this kind of method has been utilized to compute higher Frobenius-Schur indicators \cite{MR3441221}.
\item Generalizing \cite{MR3631720}, we introduce the space $\mathrm{CF}(\mathcal{M})$ of class functions of $\mathcal{M}$. We also introduce the notion of pivotal module category and develop the character theory for such a module category. Especially, we show that the characters of simple objects are linearly independent.
\item We introduce a filtration $\CF_1(\mathcal{M}) \subset \CF_2(\mathcal{M}) \subset \dotsb \subset \CF(\mathcal{M})$ for the space of class functions. If $\mathcal{M}$ is pivotal, then the first term $\CF_1(\mathcal{M})$ is spanned by the characters of simple objects of $\mathcal{M}$ and the second term has the following expression:
  \begin{equation*}
    \CF_2(\mathcal{M}) \cong \CF_1(\mathcal{M}) \oplus \bigoplus_{L \in \Irr(\mathcal{M})} \Ext_{\mathcal{M}}^1(L, L).
  \end{equation*}
\item We show that $\mathrm{Ext}_{\mathcal{C}}^*(\unitobj, A_{\mathcal{M}})$ is isomorphic to the Hochschild cohomology of $\mathcal{M}$, where $\unitobj$ is the unit object of $\mathcal{C}$ and $A_{\mathcal{M}} = \uprho^{\mathrm{ra}}(\mathrm{id}_{\mathcal{M}})$. As an application, we show that the modular group $\mathrm{SL}_2(\mathbb{Z})$ acts projectively on the Hochschild cohomology of a modular tensor category, generalizing \cite{2017arXiv170704032L}.
\end{enumerate}

\subsection*{Organization of this paper}

This paper is organized as follows: Section~\ref{sec:FTC} collects several basic notions and facts on finite abelian categories, finite tensor categories and their module categories from \cite{MR1712872,MR3242743,2013arXiv1312.7188D,2014arXiv1406.4204D,2016arXiv161204561F}.

In Section~\ref{sec:adj-of-act}, we study adjoints of the action functor $\Act: \mathcal{C} \to \REX(\mathcal{M})$ for a finite tensor category $\mathcal{C}$ and a finite left $\mathcal{C}$-module category $\mathcal{M}$. We show that $\Act$ is an exact functor, and thus has a left adjoint and a right adjoint. It turns out that a right adjoint $\Act^{\radj}$ of $\Act$ is expressed as in \eqref{eq:intro-act-right-adj}. Moreover, $\Act^{\radj}$ is $k$-linear faithful exact functor if $\mathcal{M}$ is indecomposable and exact (Theorem~\ref{thm:action-adj-by-end}).

The functor $\Act^{\radj}$ has a natural structure of a monoidal functor and a $\mathcal{C}$-bimodule functor as a right adjoint of $\Act$. The structure morphisms of $\Act^{\radj}$ are expressed in terms of the universal dinatural transformation of $\Act^{\radj}$ as an end (Lemmas~\ref{lem:act-fun-adj-C-bimod} and~\ref{lem:act-fun-adj-monoidal}). By using the structure morphisms of $\Act^{\radj}$, we can `lift' the adjoint pair $(\Act, \Act^{\radj})$ to an adjoint pair between the Drinfeld center $\mathcal{Z}(\mathcal{C})$ and the category $\REX_{\mathcal{C}}(\mathcal{M})$ of $k$-linear right exact $\mathcal{C}$-module endofunctors on $\mathcal{M}$ (Theorem~\ref{thm:act-fun-adj-Z}). As an application, we give an explicit description of the composition
\begin{equation*}
  \mathcal{C}_{\mathcal{M}}^* := \REX_{\mathcal{C}}(\mathcal{M})^{\rev}
  \xrightarrow{\quad \text{induction} \quad}
  \mathcal{Z}(\mathcal{C}_{\mathcal{M}}^*)
  \xrightarrow{\quad \text{Schauenburg's equivalence} \quad}
  \mathcal{Z}(\mathcal{C})
\end{equation*}
in terms of the structure morphisms of $\Act^{\radj}$ (Theorem~\ref{thm:induction-Drinfeld-center}).

In Section~\ref{sec:integral-over-fullsub}, we consider an end of the form $A_{\mathcal{S}} := \int_{X \in \mathcal{S}} \iHom(X, X)$ for some topologizing full subcategory $\mathcal{S}$ of $\mathcal{M}$ in the sense of Rosenberg \cite{MR1347919}. The end $A_{\mathcal{S}}$ has a natural structure of an algebra in $\mathcal{C}$. The main result of this section states that, if $\mathcal{M}$ is an indecomposable exact left $\mathcal{C}$-module category, then $A_{\mathcal{S}}$ is a quotient algebra of $A_{\mathcal{M}}$ and the map
\begin{equation*}
  \left\{
    \begin{array}{c}
      \text{topologizing full} \\
      \text{subcategories of $\mathcal{M}$}
    \end{array}
  \right\}
  \to \left\{
    \begin{array}{c}
      \text{quotient algebras} \\
      \text{of $A_{\mathcal{M}}$ in $\mathcal{C}$}
    \end{array}
  \right\},
  \quad \mathcal{S} \mapsto A_{\mathcal{S}}
\end{equation*}
preserves and reflects the order in a certain sense (Theorem \ref{thm:adjoint-alg-quotients}). Another important result in Section~\ref{sec:integral-over-fullsub} is that, if $\mathcal{S}$ is closed under the action of $\mathcal{C}$, then $A_{\mathcal{S}}$ lifts to a commutative algebra $\mathbf{A}_{\mathcal{S}}$ in $\mathcal{Z}(\mathcal{C})$ (Theorem \ref{thm:comm-alg-from-mod-subcat}).

In Section~\ref{sec:class-function}, we consider the space $\CF(\mathcal{M}) := \Hom_{\mathcal{C}}(A_{\mathcal{M}}, \unitobj)$ of `class functions' of $\mathcal{M}$. As we have seen in \cite{MR3631720}, $\CF(\mathcal{M})$ is an algebra if $\mathcal{M} = \mathcal{C}$. We extend this result by constructing a map $\star: \CF(\mathcal{C}) \times \CF(\mathcal{M}) \to \CF(\mathcal{M})$ making $\CF(\mathcal{M})$ a left $\CF(\mathcal{C})$-module (Lemma \ref{lem:CF-C-module-CF-M}). We also introduce the notion of pivotal structure of an exact module category over a pivotal finite tensor category (Definition~\ref{def:pivotal-mod-cat}) in terms of the relative Serre functor introduced in \cite{2016arXiv161204561F}. Let $\mathcal{C}$ be a pivotal finite tensor category, and let $\mathcal{M}$ be a pivotal exact left $\mathcal{C}$-module category. Then, for each object $M \in \mathcal{M}$, the {\em internal character} $\mathrm{ch}_{\mathcal{M}}(M)$ is defined in an analogous way as \cite{MR3631720} (Definition~\ref{def:internal-character}). Our main result in this section is the following generalization of \cite{MR3631720}: The linear map
\begin{equation*}
  \mathrm{ch}_{\mathcal{M}}: \Gr_k(\mathcal{M}) \to \CF(\mathcal{M}),
  \quad [M] \mapsto \mathrm{ch}_{\mathcal{M}}(M)
\end{equation*}
is a well-defined injective map, where $\Gr_k(-) = k \otimes_{\mathbb{Z}} \Gr(-)$ is the coefficient extension of the Grothendieck group. Moreover, we have
\begin{equation*}
  \mathrm{ch}_{\mathcal{M}}(X \otimes M) = \mathrm{ch}_{\mathcal{C}}(X) \star \mathrm{ch}_{\mathcal{M}}(M)
\end{equation*}
for all objects $X \in \mathcal{C}$ and $M \in \mathcal{M}$.

In Section~\ref{sec:filtr-space-class}, we introduce a filtration to the space of class functions. Let $\mathcal{C}$ be a finite tensor category, and let $\mathcal{M}$ be an exact left $\mathcal{C}$-module category. There is the {\em socle filtration} $\mathcal{M}_1 \subset \mathcal{M}_2 \subset \dotsb$ of $\mathcal{M}$. By the result of Section~\ref{sec:integral-over-fullsub}, we have a series $A_{\mathcal{M}} \twoheadrightarrow \dotsb \twoheadrightarrow A_{\mathcal{M}_2} \twoheadrightarrow A_{\mathcal{M}_1}$ of epimorphisms in $\mathcal{C}$. Thus, by applying the functor $\Hom_{\mathcal{C}}(-, \unitobj)$ to this series, we have a filtration
\begin{equation*}
  \CF_1(\mathcal{M}) \subset \CF_2(\mathcal{M}) \subset \CF_3(\mathcal{M}) \subset \dotsb \subset \CF(\mathcal{M}),
\end{equation*}
where $\CF_n(\mathcal{M}) = \Hom_{\mathcal{C}}(A_{\mathcal{M}_n}, \unitobj)$. We investigate relations between this filtration and some ring-theoretic notions, such as the Jacobson radical, the Reynolds ideal and the space of symmetric linear forms. We see that $\CF_1(\mathcal{M})$ is spanned by the characters of simple objects of $\mathcal{M}$. The second term $\CF_2(\mathcal{M})$ is expressed in terms of $\Ext_{\mathcal{M}}^1(L, L)$ for simple objects $L \in \mathcal{M}$. For $\CF_n(\mathcal{M})$ with $n \ge 3$, we have no general results but study some examples.

In Section~\ref{sec:hochsch-cohom}, we discuss the Hochschild (co)homology of finite tensor categories and their module categories. One can define the Hochschild homology $\HH_{\bullet}(\mathcal{M})$ and the Hochschild cohomology $\HH^{\bullet}(\mathcal{M})$ of a finite abelian category $\mathcal{M}$ in terms of the Ext functor in $\REX(\mathcal{M})$. We then show that, if $\mathcal{M}$ is an exact $\mathcal{C}$-module category, then there is an isomorphism
\begin{equation}
  \label{eq:intro-HH}
  \HH^{\bullet}(\mathcal{M}) \cong \Ext_{\mathcal{C}}^{\bullet}(\unitobj, A_{\mathcal{M}})
\end{equation}
If, in addition, $\mathcal{M}$ is pivotal, then there is also an isomorphism
\begin{equation*}
  \HH_{\bullet}(\mathcal{M}) \cong \Ext_{\mathcal{C}}^{\bullet}(A_{\mathcal{M}}, \unitobj)^*.
\end{equation*}
The isomorphism~\eqref{eq:intro-HH} is a generalization of the known fact that the Hochschild cohomology of a Hopf algebra can be computed by the cohomology of the adjoint representation. We use~\eqref{eq:intro-HH} to extend recent results of \cite{2017arXiv170704032L}.

\subsection*{Acknowledgment}

The author is supported by JSPS KAKENHI Grant Number JP16K17568.

\section{Preliminaries}
\label{sec:FTC}

\subsection{Ends and coends}

For basic theory on categories, we refer the reader to the book of Mac Lane \cite{MR1712872}. Let $\mathcal{C}$ and $\mathcal{D}$ be categories, and let $S$ and $T$ be functors from $\mathcal{C}^{\op} \times \mathcal{C}$ to $\mathcal{D}$. A {\em dinatural transformation} \cite[IX]{MR1712872} from $S$ to $T$ is a family
$\xi = \{ \xi_X: S(X, X) \to T(X, X) \}_{X \in \mathcal{C}}$
of morphisms in $\mathcal{D}$ satisfying
\begin{equation*}
  T(\id_X, f) \circ \xi_X \circ S(f, \id_X)
  = T(f, \id_Y) \circ \xi_Y \circ S(\id_Y, f)
\end{equation*}
for all morphisms $f: X \to Y$ in $\mathcal{C}$. An {\em end} of $S$ is an object $E \in \mathcal{D}$ equipped with a dinatural transformation $\pi: E \to S$ that is universal in a certain sense (here the object $E$ is regarded as a constant functor from $\mathcal{C}^{\op} \times \mathcal{C}$ to $\mathcal{D}$). Dually, a {\em coend} of $T$ is an object $C \in \mathcal{D}$ equipped with a `universal' dinatural transformation from $T$ to $C$. An end of $S$ and a coend of $T$ are denoted by $\int_{X \in \mathcal{C}} S(X, X)$ and $\int^{X \in \mathcal{C}} T(X, X)$, respectively.

A (co)end does not exist in general. We note the following useful criteria for the existence of (co)ends. Suppose that $\mathcal{C}$ is essentially small. Let $\mathcal{C}$, $\mathcal{D}$ and $S$ be as above. Since the category $\mathbf{Set}$ of all sets is complete, the end
\begin{equation*}
  S^{\natural}(W) := \int_{X \in \mathcal{C}} \Hom_{\mathcal{D}}(W, S(X, X))
\end{equation*}
exists for each object $W \in \mathcal{D}$. By the parameter theorem for ends \cite[XI.7]{MR1712872}, we extend the assignment $W \mapsto S^{\natural}(W)$ to the contravariant functor $S^{\natural}: \mathcal{D} \to \mathbf{Set}$. The following lemma is the dual of \cite[Lemma 3.1]{MR3632104}.

\begin{lemma}
  \label{lem:existence-end}
  An end of $S$ exists if and only if $S^{\natural}$ is representable.
\end{lemma}

We also note the following lemma:

\begin{lemma}
  \label{lem:end-adj}
  Let $\mathcal{A}$, $\mathcal{B}$ and $\mathcal{V}$ be categories, and let $L: \mathcal{A} \to \mathcal{B}$, $R: \mathcal{B} \to \mathcal{A}$ and $H: \mathcal{B}^{\op} \times \mathcal{A} \to \mathcal{V}$ be functors. Suppose that $L$ is left adjoint to $R$. Then we have an isomorphism
  \begin{equation}
    \label{eq:end-adj-iso}
    \int_{V \in \mathcal{A}} H(V, L(V))
    \cong \int_{W \in \mathcal{B}} H(R(W), W),
  \end{equation}
  meaning that if either one of these ends exists, then both exist and they are canonically isomorphic.
\end{lemma}

This lemma is the dual of \cite[Lemma 3.9]{MR2869176}. For later use, we recall the construction of the canonical isomorphism~\eqref{eq:end-adj-iso}. Let $\mathcal{E}$ and $\mathcal{E}'$ be the left and the right hand side of \eqref{eq:end-adj-iso}, respectively, and let
\begin{equation*}
  \pi(V): \mathcal{E} \to H(V, L(V))
  \quad \text{and} \quad
  \pi'(W): \mathcal{E}' \to H(R(W), W)
\end{equation*}
be the respective universal dinatural transformations. We assume that $(L, R)$ is an adjoint pair with unit $\eta: \id_{\mathcal{D}} \to R L$ and counit $\varepsilon: L R \to \id_{\mathcal{C}}$. By the universal property of $\mathcal{E}$, there is a unique morphism $\alpha: \mathcal{E}' \to \mathcal{E}$ in $\mathcal{V}$ satisfying
\begin{equation*}
  \pi(V) \circ \alpha = H(\eta_{V}, \id_{L(V)}) \circ \pi'(L(V))
\end{equation*}
for all objects $V \in \mathcal{A}$. Similarly, by the universal property of $\mathcal{E}'$, there is a unique morphism $\beta: \mathcal{E} \to \mathcal{E}'$ satisfying
\begin{equation*}
  \pi'(W) \circ \beta = \pi(R(W)) \circ H(\id_{R(W)}, \varepsilon_W)
\end{equation*}
for all objects $W \in \mathcal{B}$. By the zigzag identities and the dinaturality of $\pi$ and $\pi'$, one can verify that $\alpha$ and $\beta$ are mutually inverse.

\subsection{Monoidal categories}

A {\em monoidal category} \cite[VII]{MR1712872} is a category $\mathcal{C}$ equipped with a functor $\otimes: \mathcal{C} \times \mathcal{C} \to \mathcal{C}$, an object $\unitobj \in \mathcal{C}$ and natural isomorphisms $(X \otimes Y) \otimes Z \cong X \otimes (Y \otimes Z)$ and $\unitobj \otimes X \cong X \cong X \otimes \unitobj$ ($X, Y, Z \in \mathcal{C}$) satisfying the pentagon and the triangle axiom. If these natural isomorphisms are identities, then $\mathcal{C}$ is said to be {\em strict}. By the Mac Lane coherence theorem, we may assume that every monoidal category is strict.

We fix several conventions on monoidal categories and related notions: Let $\mathcal{C}$ and $\mathcal{D}$ be monoidal categories. A {\em monoidal functor} \cite[XI.2]{MR1712872} from $\mathcal{C}$ to $\mathcal{D}$ is a functor $F: \mathcal{C} \to \mathcal{D}$ equipped with a natural transformation
\begin{equation*}
  f^{(2)}_{X,Y}: F(X) \otimes F(Y) \to F(X \otimes Y) \quad (X, Y \in \mathcal{C})
\end{equation*}
and a morphism $f^{(0)}: \unitobj \to F(\unitobj)$ in $\mathcal{D}$ satisfying certain axioms. A monoidal functor $F =(F, f^{(2)}, f^{(0)})$ is said to be {\em strong} if the structure morphisms $f^{(2)}$ and $f^{(0)}$ are invertible.

Let $L$ and $R$ be objects of a monoidal category $\mathcal{C}$, and let $\varepsilon: L \otimes R \to \unitobj$ and $\eta: \unitobj \to R \otimes L$ be morphisms in $\mathcal{C}$. We say that $(L, \varepsilon, \eta)$ is a {\em left dual object} of $R$ and $(R, \varepsilon, \eta)$ is a {\em right dual object} of $L$ if the equations
\begin{equation*}
  (\varepsilon \otimes \id_L) \circ (\id_L \otimes \eta) = \id_L
  \quad \text{and} \quad
  (\id_R \otimes \varepsilon) \circ (\eta \otimes \id_R) = \id_R
\end{equation*}
hold. If this is the case, then the morphisms $\varepsilon$ and $\eta$ are called the {\em evaluation} and the {\em coevaluation}, respectively.

A monoidal category $\mathcal{C}$ is said to be {\em rigid} if every object of $\mathcal{C}$ has a left dual object and a right dual object. If $\mathcal{C}$ is rigid, then we denote by $(X^*, \eval_X, \coev_X)$ the left dual object of $X \in \mathcal{C}$. Let $\mathcal{C}^{\rev}$ denote the category $\mathcal{C}$ equipped with the reversed tensor product $X \otimes^{\rev} Y = Y \otimes X$. The assignment $X \mapsto X^*$ gives rise to a strong monoidal functor $(-)^*: \mathcal{C}^{\op} \to \mathcal{C}^{\rev}$ called the {\em left duality functor} of $\mathcal{C}$. The {\em right duality functor} ${}^*(-)$ of $\mathcal{C}$ is also defined by taking the right dual object. The left and the right duality functor are mutually quasi-inverse to each other.

\subsection{Module categories}

Let $\mathcal{C}$ be a monoidal category. A {\em left $\mathcal{C}$-module category} \cite{MR3242743} is a category $\mathcal{M}$ equipped with a functor $\otimes: \mathcal{C} \times \mathcal{M} \to \mathcal{M}$, called the {\em action} of $\mathcal{C}$, and natural isomorphisms
\begin{equation}
  \label{eq:mod-cat-associa}
  (X \otimes Y) \otimes M \cong X \otimes (Y \otimes M)
  \quad \text{and} \quad
  \unitobj \otimes M \cong M
  \quad (X, Y \in \mathcal{C}, M \in \mathcal{M})
\end{equation}
satisfying certain axioms similar to those for monoidal categories. There is an analogue of the Mac Lane coherence theorem for $\mathcal{C}$-module categories. Thus, without loss of generality, we may assume that the natural isomorphisms~\eqref{eq:mod-cat-associa} are the identity; see \cite[Remark 7.2.4]{MR3242743}.

Let $\mathcal{M}$ and $\mathcal{N}$ be left $\mathcal{C}$-module categories. A {\em lax left $\mathcal{C}$-module functor} from $\mathcal{M}$ to $\mathcal{N}$ is a functor $F: \mathcal{M} \to \mathcal{N}$ equipped with a natural transformation
\begin{equation*}
  s_{X,M}: X \otimes F(M) \to F(X \otimes M) \quad (X \in \mathcal{C}, M \in \mathcal{M})
\end{equation*}
such that the equations
\begin{equation*}
  s_{\unitobj, M} = \id_{M}
  \quad \text{and} \quad
  s_{X \otimes Y, M} = s_{X, Y \otimes M} \circ (\id_X \otimes s_{Y, M})
\end{equation*}
hold for all objects $X, Y \in \mathcal{C}$ and $M \in \mathcal{M}$. We omit the definition of morphisms of lax $\mathcal{C}$-module functors; see \cite{2013arXiv1312.7188D,2014arXiv1406.4204D}.

An {\em oplax left $\mathcal{C}$-module functor} from $\mathcal{M}$ to $\mathcal{N}$ is, in a word, a lax left $\mathcal{C}^{\op}$-module functor from $\mathcal{M}^{\op}$ to $\mathcal{N}^{\op}$; see \cite{2014arXiv1406.4204D}. Now let $L: \mathcal{M} \to \mathcal{N}$ be a functor with right adjoint $R: \mathcal{N} \to \mathcal{M}$, and let $\eta: \id_{\mathcal{M}} \to R L$ and $\varepsilon: L R \to \id_{\mathcal{N}}$ be the unit and the counit of the adjunction $L \dashv R$. If $(L, v)$ is an oplax left $\mathcal{C}$-module functor, then $R$ is a lax left $\mathcal{C}$-module functor by the structure morphism defined by
\begin{equation}
  \label{eq:module-func-adj-right}
  \newcommand{\XARR}[1]{\xrightarrow{\makebox[5em]{$\scriptstyle #1$}}}
  \begin{aligned}
    X \otimes R(N)
    & \XARR{\eta_{X \otimes R(N)}}
    R L (X \otimes R(N)) \\
    & \XARR{R(v_{X,R(N)})}
    R(X \otimes L R(N))
    \XARR{R(\id_X \otimes \varepsilon_N)}
    R(X \otimes N)    
  \end{aligned}
\end{equation}
for $X \in \mathcal{C}$ and $N \in \mathcal{N}$. Conversely, if $(R, w)$ is a lax $\mathcal{C}$-module functor, then $L$ is an oplax $\mathcal{C}$-module functor by
\begin{equation}
  \label{eq:module-func-adj-left}
  \newcommand{\XARR}[1]{\xrightarrow{\makebox[5em]{$\scriptstyle #1$}}}
  \begin{aligned}
    L(X \otimes M)
    & \XARR{L(\id_X \otimes \eta_{M})}
    L(X \otimes R L(M)) \\
    & \XARR{L(w_{X,L(M)})}
    L R(X \otimes L(M))
    \XARR{\varepsilon_{X \otimes L(M)}}
    R(X \otimes N)    
  \end{aligned}
\end{equation}
for $X \in \mathcal{C}$ and $M \in \mathcal{M}$ \cite[Lemma 2.11]{2014arXiv1406.4204D}.

We say that an (op)lax $\mathcal{C}$-module functor $(F, s)$ is {\em strong} if the natural transformation $s$ is invertible. If $\mathcal{C}$ is rigid, then every (op)lax $\mathcal{C}$-module functor is strong \cite[Lemma 2.10]{2014arXiv1406.4204D} and thus we refer to an (op)lax $\mathcal{C}$-module functor simply as a $\mathcal{C}$-module functor.

\subsection{Closed module categories}
\label{subsec:cl-mod-cat}

Let $\mathcal{C}$ be a monoidal category. A left $\mathcal{C}$-module category $\mathcal{M}$ is said to be {\em closed} if, for all objects $M \in \mathcal{M}$, the functor
\begin{equation}
  \label{eq:iHom-def-0}
  \mathcal{C} \to \mathcal{M}, \quad X \mapsto X \otimes M
\end{equation}
has a right adjoint ({\it cf}. the definition of a closed monoidal category). Suppose that $\mathcal{M}$ is closed. For each object $M \in \mathcal{M}$, we fix a right adjoint $\iHom(M, -)$ of the functor \eqref{eq:iHom-def-0}. Thus, by definition, there is a natural isomorphism
\begin{equation}
  \label{eq:iHom-def}
  \phi: \Hom_{\mathcal{M}}(X \otimes M, N) \to \Hom_{\mathcal{C}}(X, \iHom(M, N))
\end{equation}
for $N \in \mathcal{M}$ and $X \in \mathcal{C}$. If we denote by
\begin{equation*}
  \icoev_{X,M}: X \to \iHom(M, X \otimes M)
  \quad \text{and} \quad
  \ieval_{M,N}: \iHom(M, N) \otimes M \to N
\end{equation*}
the unit and the counit of the adjunction $(-) \otimes M \dashv \iHom(M, -)$, respectively, then the isomorphism \eqref{eq:iHom-def-0} is given by
\begin{equation}
  \label{eq:iHom-adj}
  \phi(f) = \iHom(M, f) \circ \icoev_{M,X}
  \quad \text{and} \quad
  \phi^{-1}(g) = \ieval_{M,N} \circ (g \otimes \id_M)
\end{equation}
for morphisms $f: X \otimes M \to N$ in $\mathcal{M}$ and $g: X \to \iHom(M, N)$ in $\mathcal{C}$.

By \cite[IV.7]{MR1712872}, one can extend the assignment $(M, N) \mapsto \iHom(M, N)$ to a functor from $\mathcal{M}^{\op} \times \mathcal{M}$ to $\mathcal{C}$ in such a way that the isomorphism \eqref{eq:iHom-def} is natural also in $M$. We call the functor $\iHom$ the {\em internal Hom functor} of $\mathcal{M}$. This makes $\mathcal{M}$ a $\mathcal{C}$-enriched category: The composition
\begin{equation}
  \label{eq:iHom-comp}
  \icomp_{M_1, M_2, M_3}: \iHom(M_2, M_3) \otimes \iHom(M_1, M_2) \to \iHom(M_1, M_3)
\end{equation}
is defined to be the morphism corresponding to
\begin{equation}
  \label{eq:iHom-eval3}
  \ieval^{(3)}_{M_1, M_2, M_3} := \ieval_{M_2, M_3} \circ (\id_{\iHom(M_2, M_3)} \otimes \ieval_{M_1, M_2})
\end{equation}
via the isomorphism~\eqref{eq:iHom-def} with $X = \iHom(M_2, M_3) \otimes \iHom(M_1, M_2)$, $M =M_1$ and $N = M_3$. The identity on $M \in \mathcal{M}$ is $\icoev_{\unitobj, M}$.

We suppose that $\mathcal{C}$ is rigid. Let $M \in \mathcal{M}$ be an object. Since the functor \eqref{eq:iHom-def-0} is a $\mathcal{C}$-module functor, so is its right adjoint $\iHom(M, -)$. We denote by
\begin{equation}
  \label{eq:iHom-iso-a}
  \iHomA_{X,M,N}: X \otimes \iHom(M, N) \to \iHom(M, X \otimes N)
  \quad (X \in \mathcal{C}, N \in \mathcal{M})
\end{equation}
the left $\mathcal{C}$-module structure of $\iHom(M, -)$. There is also an isomorphism
\begin{equation}
  \label{eq:iHom-iso-b-def}
  \iHomB_{X,M,N}: \iHom(X \otimes M, N) \to \iHom(M, N) \otimes X^*
\end{equation}
induced by the natural isomorphisms
\begin{gather*}
  \Hom_{\mathcal{C}}(W, \iHom(Y \otimes M, N))
  \cong \Hom_{\mathcal{M}}(W \otimes Y \otimes M, N) \\
  \cong \Hom_{\mathcal{C}}(W \otimes Y, \iHom(M, N))
  \cong \Hom_{\mathcal{C}}(W, \iHom(M, N) \otimes Y^*)
\end{gather*}
for $W, Y \in \mathcal{C}$ and $N, M \in \mathcal{M}$. It is convenient to introduce the morphism
\begin{equation}
  \label{eq:iHom-b-natural}
  \iHomB^{\natural}_{X,M,N}: \Hom(X \otimes M, N) \otimes X \to \iHom(M, N)
\end{equation}
defined by $\iHomB^{\natural}_{X,M,N} = (\id_{\iHom(M, N)} \otimes \eval_X) \circ (\iHomB_{X,M,N} \otimes \id_X)$. We note that $\iHomB^{\natural}_{X,M,N}$ is natural in the variables $M$ and $N$ and dinatural in $X$.

\begin{lemma}
  \label{lem:iHom-iso-B}
  For all objects $X, Y \in \mathcal{C}$ and $M, N \in \mathcal{M}$, we have the equation
  \begin{equation}
    \label{eq:lem-iHom-iso-B-1}
    \iHomB_{\unitobj, M, N} = \id_{\iHom(M, N)}
  \end{equation}
  and the following commutative diagrams:
  \begin{gather}
    \label{eq:lem-iHom-iso-B-2}
    \xymatrix@C=84pt@R=16pt{
      \iHom(X \otimes Y \otimes M, N)
      \ar[d]_{\iHomB_{X, Y \otimes M, N}}
      \ar[r]^{\iHomB_{X \otimes Y, M, N}}
      & \iHom(M, N) \otimes (X \otimes Y)^*
      \ar@{=}[d] \\
      \iHom(Y \otimes M, N) \otimes X^*
      \ar[r]_{\iHomB_{Y,M,N} \otimes \id_{X^*}}
      & \iHom(M, N) \otimes Y^* \otimes X^*
    } \\
    \label{eq:lem-iHom-iso-B-3}
    \xymatrix@C=84pt@R=16pt{
      X \otimes \iHom(Y \otimes M, N)
      \ar[r]^{\id_X \otimes \iHomB_{Y, M, N}}
      \ar[d]_{\iHomA_{X, Y \otimes M, N}}
      & X \otimes \iHom(M, N) \otimes Y^*
      \ar[d]^{\iHomA_{X,M,N} \otimes \id_{Y^*}} \\
      \iHom(Y \otimes M, X \otimes N)
      \ar[r]^{\iHomB_{Y, M, X \otimes N}}
      & \iHom(M, X \otimes N) \otimes Y^*
    }
  \end{gather}
\end{lemma}

The category $\mathcal{M}^{\op} \times \mathcal{M}$ is a $\mathcal{C}$-bimodule category by the actions given by
\begin{equation}
  \label{eq:Mop-M-actions}
  X \otimes (M, N) =  (M^{\op}, X \otimes N)
  \quad \text{and} \quad
  (M^{\op}, N) \otimes X = ({}^* \! X \otimes M, N)
\end{equation}
for $X \in \mathcal{C}$ and $M, N \in \mathcal{M}$. The above lemma means that the internal Hom functor of $\mathcal{M}$ is a $\mathcal{C}$-bimodule functor from $\mathcal{M}^{\op} \times \mathcal{M}$ to $\mathcal{C}$ with left $\mathcal{C}$-module structure $\iHomA$ and the right $\mathcal{C}$-module structure given by
\begin{equation*}
  (\iHomB_{{}^* \! Y, M, N})^{-1}: \iHom(M, N) \otimes Y \to \iHom({}^* Y \otimes M, N)
  \quad (Y \in \mathcal{C}, M, N \in \mathcal{M}).
\end{equation*}

Although the above lemma seems to be well-known, we give its proof in Appendix~\ref{apdx:act-adj-structure} for the sake of completeness. We will also give some equations involving the natural isomorphisms $\iHomA$ and $\iHomB$ in Appendix~\ref{apdx:act-adj-structure}.

In view of this lemma, we define the isomorphism
\begin{equation}
  \label{eq:iHom-iso-c}
  \iHomC_{X,M,N,Y}: X \otimes \iHom(M, N) \otimes Y^* \to \iHom(Y \otimes M, X \otimes N)
\end{equation}
for $X, Y \in \mathcal{C}$ and $M, N \in \mathcal{M}$ by
\begin{equation}
  \label{eq:iHom-iso-c-def}
  \iHomC_{X,M,N,Y}
  = \iHomB_{Y, M, X \otimes N}^{-1} \circ (\iHomA_{X,M,N} \otimes \id_{Y^*})
  = \iHomA_{X, Y \otimes M, N} \circ (\id_X \otimes \iHomB_{Y,M,N}^{-1}).
\end{equation}

\subsection{Finite abelian categories}
\label{subsec:fin-ab-cat}

Throughout this paper, we work over an algebraically closed field $k$ of arbitrary characteristic. Given algebras $A$ and $B$ over $k$, we denote by $\lmod{A}$, $\rmod{B}$ and $\bimod{A}{B}$ the category of finite-dimensional left $A$-modules, the category of finite-dimensional right $B$-modules, and the category of finite-dimensional $A$-$B$-bimodules, respectively.

A {\em finite abelian category} \cite[Definition 1.8.5]{MR3242743} is a $k$-linear category that is equivalent to $\lmod{A}$ for some finite-dimensional algebra $A$ over $k$. For finite abelian categories $\mathcal{M}$ and $\mathcal{N}$, we denote by $\REX(\mathcal{M}, \mathcal{N})$ the category of $k$-linear right exact functors from $\mathcal{M}$ to $\mathcal{N}$. If $A$ and $B$ are finite-dimensional algebras over $k$, then the Eilenberg-Watts theorem gives an equivalence
\begin{equation}
  \label{eq:algebraic-EW}
  \bimod{B}{A} \xrightarrow{\quad \approx \quad} \REX(\lmod{A}, \lmod{B}), \quad M \mapsto M \otimes_A (-)
\end{equation}
of $k$-linear categories. Thus $\REX(\mathcal{M}, \mathcal{N})$ is a finite abelian category. The above equivalence also implies that a $k$-linear functor $F: \mathcal{M} \to \mathcal{N}$ is right exact if and only if $F$ has a right adjoint.

A $k$-linear category $\mathcal{M}$ is finite abelian if and only if $\mathcal{M}^{\op}$ is. Thus, by the dual argument, we see that a $k$-linear functor $F: \mathcal{M} \to \mathcal{N}$ is left exact if and only if $F$ has a left adjoint. We denote by $\LEX(\mathcal{M}, \mathcal{N})$ the category of $k$-linear left exact functors from $\mathcal{M}$ to $\mathcal{N}$. For a $k$-linear functor $F$, we denote by $F^{\ladj}$ and $F^{\radj}$ a left and a right adjoint of $F$, respectively, if it exists.

Let $M$ be an object of $\mathcal{M}$. Then the functor $\Hom_{\mathcal{M}}(M, -): \mathcal{M} \to \lmod{k}$ is left exact, and thus has a left adjoint. We denote it by $(-) \otimes_k M$. By definition, there is a natural isomorphism
\begin{equation*}
  \Hom_{\mathcal{M}}(X \otimes_k M, N) \cong \Hom_{k}(X, \Hom_{\mathcal{N}}(M, N))
  \quad (X \in \lmod{k}, N \in \mathcal{M}).
\end{equation*}

For two finite abelian categories $\mathcal{M}$ and $\mathcal{N}$, we denote by $\mathcal{M} \boxtimes \mathcal{N}$ their Deligne tensor product \cite[\S1.11]{MR3242743}. If $\mathcal{M} = \lmod{A}$ and $\mathcal{N} = \lmod{B}$ for some finite-dimensional algebras $A$ and $B$, then $\mathcal{M} \boxtimes \mathcal{N}$ is identified with $\lmod{(A \otimes_k B)}$. In view of the equivalence~\eqref{eq:algebraic-EW}, one has:

\begin{lemma}[{\cite[Lemma 3.3]{MR3569179}}]
  The $k$-linear functor
  \begin{equation}
    \label{eq:categorical-EW}
    \Phi_{\mathcal{M},\mathcal{N}}: \mathcal{M}^{\op} \boxtimes \mathcal{N} \to \REX(\mathcal{M}, \mathcal{N}),
    \quad M^{\op} \boxtimes N \mapsto \Hom_{\mathcal{M}}(-, M)^* \otimes_k N
  \end{equation}
  is an equivalence. Moreover, the functor
  \begin{equation}
    \label{eq:categorical-EW-inv}
    \REX(\mathcal{M}, \mathcal{N}) \to \mathcal{M}^{\op} \boxtimes \mathcal{N},
    \quad F \mapsto \int_{M \in \mathcal{M}} M^{\op} \boxtimes F(M)
  \end{equation}
  is a quasi-inverse of \eqref{eq:categorical-EW}
\end{lemma}

The following lemma is proved by utilizing the equivalences~\eqref{eq:categorical-EW} and~\eqref{eq:categorical-EW-inv}.

\begin{lemma}[{\cite[Lemma 2.5]{2017arXiv170709691S}}]
  \label{lem:proj-obj-in-REX}
  Let $\mathcal{M}$ and $\mathcal{N}$ be finite abelian categories. For a $k$-linear right exact functor $F: \mathcal{M} \to \mathcal{N}$, the following are equivalent:
  \begin{enumerate}
  \item $F$ is a projective object of the abelian category $\REX(\mathcal{M}, \mathcal{N})$.
  \item $F(M)$ is a projective object of $\mathcal{N}$ for all objects $M \in \mathcal{M}$ and $F^{\radj}(N)$ is an injective object of $\mathcal{M}$ for all objects $N \in \mathcal{N}$.
  \end{enumerate}
\end{lemma}

For finite abelian categories $\mathcal{M}$ and $\mathcal{N}$, there is also an equivalence
\begin{equation*}
  \Psi_{\mathcal{M},\mathcal{N}}: \LEX(\mathcal{M}, \mathcal{N}) \to \mathcal{M}^{\op} \boxtimes \mathcal{N},
  \quad F \mapsto \int^{M \in \mathcal{M}} M^{\op} \boxtimes F(M)
\end{equation*}
\cite[Lemmas 3.2 and 3.3]{MR3632104}. Fuchs, Schaumann and Schweigert \cite{2016arXiv161204561F} defined the {\em Nakayama functor} of $\mathcal{M}$ by
\begin{equation*}
  \Nak_{\mathcal{M}} := \Phi_{\mathcal{M}, \mathcal{M}} \Psi_{\mathcal{M}, \mathcal{M}}(\id_{\mathcal{M}}) \in \REX(\mathcal{M}, \mathcal{M}).
\end{equation*}
For later use, we recall from \cite{2016arXiv161204561F} the following results:
\begin{enumerate}
\item We say that $\mathcal{M}$ is {\em Frobenius} if the class of injective objects of $\mathcal{M}$ coincides with the class of projective objects of $\mathcal{M}$ (or, equivalently, $\mathcal{M} \approx \lmod{A}$ for some Frobenius algebra $A$). The Nakayama functor $\Nak_{\mathcal{M}}$ is an equivalence if and only if $\mathcal{M}$ is Frobenius.
\item We say that $\mathcal{M}$ is {\em symmetric Frobenius} if $\mathcal{M} \approx \lmod{A}$ for some symmetric Frobenius algebra $A$. The Nakayama functor $\Nak_{\mathcal{M}}$ is isomorphic to $\id_{\mathcal{M}}$ if and only if $\mathcal{M}$ is symmetric Frobenius.
\item If $F: \mathcal{M} \to \mathcal{N}$ is a $k$-linear exact functor between finite abelian categories $\mathcal{M}$ and $\mathcal{N}$, then there is an isomorphism $\Nak_{\mathcal{M}} \circ F^{\ladj} \cong F^{\radj} \circ \Nak_{\mathcal{N}}$.
\item The Nakayama functor of $\REX(\mathcal{M}, \mathcal{N})$ is given by
  \begin{equation*}
    \Nak_{\REX(\mathcal{M}, \mathcal{N})}(F) = \Nak_{\mathcal{N}} \circ F \circ \Nak_{\mathcal{M}}
    \quad (F \in \REX(\mathcal{M}, \mathcal{N})).
  \end{equation*}
\end{enumerate}

\subsection{Finite tensor categories and their modules}

A {\em finite tensor category} \cite{MR2119143} is a rigid monoidal category $\mathcal{C}$ such that $\mathcal{C}$ is a finite abelian category, the tensor product of $\mathcal{C}$ is $k$-linear in each variable, and the unit object $\unitobj$ of $\mathcal{C}$ is a simple object. A finite tensor category is Frobenius. The tensor product of a finite tensor category is exact in each variable.

Let $\mathcal{C}$ be a finite tensor category. A {\em finite left $\mathcal{C}$-module category} is a left $\mathcal{C}$-module category $\mathcal{M}$ such that $\mathcal{M}$ is a finite abelian category and the action of $\mathcal{C}$ on $\mathcal{M}$ is $k$-linear and right exact in each variable. One can define a finite right $\mathcal{C}$-module category and a finite $\mathcal{C}$-bimodule category in a similar manner.

Given an algebra $A \in \mathcal{C}$ ($=$ a monoid in $\mathcal{C}$ \cite{MR1712872}), we denote by $\mathcal{C}_A$ the category of right $A$-modules in $\mathcal{C}$. The category $\mathcal{C}_A$ is a finite left $\mathcal{C}$-module category in a natural way. Moreover, every finite left $\mathcal{C}$-module category is equivalent to $\mathcal{C}_A$ for some algebra $A \in \mathcal{C}$ as a $\mathcal{C}$-module category. This implies that the action of $\mathcal{C}$ on a finite $\mathcal{C}$-module category is {\em exact} in each variable  \cite[Corollary 2.26]{2014arXiv1406.4204D}, although only the right exactness is assumed in our definition.

An {\em exact left $\mathcal{C}$-module category} \cite{MR2119143} is a finite left $\mathcal{C}$-module category $\mathcal{M}$ such that $P \otimes M$ is a projective object of $\mathcal{M}$ for all projective objects $P \in \mathcal{C}$ and for all objects $M \in \mathcal{M}$. It is known that exact module categories are Frobenius.

\section{Adjoint of the action functor}
\label{sec:adj-of-act}

\subsection{The action functor}

Let $\mathcal{C}$ be a finite tensor category, and let $\mathcal{M}$ and $\mathcal{N}$ be two finite left $\mathcal{C}$-module categories. Then $\REX(\mathcal{M}, \mathcal{N})$ is a $\mathcal{C}$-bimodule category by the left action and the right action given by
\begin{equation}
  \label{eq:Rex-action}
  (X \otimes F)(M) = X \otimes F(M)
  \quad \text{and} \quad
  (F \otimes X)(M) = F(X \otimes M),
\end{equation}
respectively, for $F \in \REX(\mathcal{M}, \mathcal{N})$, $X \in \mathcal{C}$ and $M \in \mathcal{M}$. The category $\mathcal{M}^{\op} \boxtimes \mathcal{N}$ is also a $\mathcal{C}$-bimodule category by the left and the right action determined by
\begin{equation*}
  X \otimes (M^{\op} \boxtimes N) =  M^{\op} \boxtimes (X \otimes N)
  \quad \text{and} \quad
  (M^{\op} \boxtimes N) \otimes X = ({}^* \! X \otimes M)^{\op} \boxtimes N,
\end{equation*}
respectively, for $M \in \mathcal{M}$, $N \in \mathcal{N}$ and $X \in \mathcal{C}$. It is easy to see that the equivalence \eqref{eq:categorical-EW} is in fact an equivalence of $\mathcal{C}$-bimodule categories. Since $\mathcal{M}^{\op} \boxtimes \mathcal{N}$ is a finite $\mathcal{C}$-bimodule category, so is $\REX(\mathcal{M}, \mathcal{N})$.

Now we define the functor $\Act_{\mathcal{M}}: \mathcal{C} \to \REX(\mathcal{M}) := \REX(\mathcal{M}, \mathcal{M})$ by $X \mapsto X \otimes \id_{\mathcal{M}}$ and call $\Act_{\mathcal{M}}$ the {\em action functor} of $\mathcal{M}$. Since the action of $\mathcal{C}$ on a finite $\mathcal{C}$-module category is $k$-linear and exact in each variable, we have:

\begin{lemma}
  \label{lem:act-fun-exact}
  The action functor $\Act_{\mathcal{M}}$ is $k$-linear and exact.
\end{lemma}

Thus the action functor $\Act_{\mathcal{M}}$ has a left adjoint and a right adjoint. The aim of this section is to study properties of adjoints of $\Act_{\mathcal{M}}$. Before doing so, we characterize some properties of $\mathcal{M}$ in terms of $\Act_{\mathcal{M}}$.

\begin{lemma}
  \label{lem:action-exact-mod-cat}
  $\mathcal{M}$ is exact if and only if $\Act_{\mathcal{M}}$ preserves projective objects.
\end{lemma}
\begin{proof}
  Suppose that $\mathcal{M}$ is an exact $\mathcal{C}$-module category. We fix a projective object $P \in \mathcal{C}$ and set $F = \Act_{\mathcal{M}}(P)$. By the definition of an exact module category, the object $F(M) = P \otimes M$ is projective for all $M \in \mathcal{M}$. Since $\mathcal{C}$ and $\mathcal{M}$ are Frobenius, the object $F^{\radj}(M) \cong {}^* \! P \otimes M$ is injective for all $M \in \mathcal{M}$. Thus, by Lemma~\ref{lem:proj-obj-in-REX}, $F$ is a projective object of $\REX(\mathcal{M})$. Hence $\Act_{\mathcal{M}}$ preserves projective objects. The converse is easily proved by Lemma~\ref{lem:proj-obj-in-REX}.
\end{proof}

Let $\mathcal{A}$ and $\mathcal{B}$ be finite abelian categories. A $k$-linear functor $F: \mathcal{A} \to \mathcal{B}$ is said to be {\em dominant} if every object of $\mathcal{B}$ is a subobject of an object of the form $F(X)$, $X \in \mathcal{A}$. Suppose that $F$ is exact and $\mathcal{B}$ is Frobenius. Then, as remarked in \cite[Lemma 2.3]{MR3600085}, the functor $F$ is dominant if and only if every object of $\mathcal{B}$ is a quotient of $F(X)$ for some $X \in \mathcal{A}$.

\begin{lemma}
  \label{lem:act-fun-dominant}
  An exact left $\mathcal{C}$-module category $\mathcal{M}$ is indecomposable if and only if the action functor $\Act_{\mathcal{M}}$ is dominant.
\end{lemma}
\begin{proof}
  Suppose that there are non-zero $\mathcal{C}$-module full subcategories $\mathcal{M}_1$ and $\mathcal{M}_2$ of $\mathcal{M}$ such that $\mathcal{M} = \mathcal{M}_1 \oplus \mathcal{M}_2$. Then we have the decomposition
  \begin{equation*}
    \REX(\mathcal{M})
    = \mathcal{E}_{11} \oplus \mathcal{E}_{12} \oplus \mathcal{E}_{21} \oplus \mathcal{E}_{22},
    \quad \mathcal{E}_{i j} = \REX(\mathcal{M}_i, \mathcal{M}_j),
  \end{equation*}
  into four non-zero full subcategories. Since the image of $\Act_{\mathcal{M}}$ is contained in the diagonal part $\mathcal{E}_{11} \oplus \mathcal{E}_{22}$, the action functor $\Act_{\mathcal{M}}$ cannot be dominant. Thus the `if' part has been proved. The `only if' part is \cite[Proposition 2.6 (ii)]{MR3600085}.
\end{proof}

\subsection{Description of adjoints}

For a while, we fix a finite tensor category $\mathcal{C}$ and a finite left $\mathcal{C}$-module category $\mathcal{M}$. We write $\Act = \Act_{\mathcal{M}}$ for simplicity. By Lemma~\ref{lem:act-fun-exact}, the functor $\Act$ has a right adjoint.

\begin{theorem}
  \label{thm:action-adj-by-end}
  For all $k$-linear right exact functor $F: \mathcal{M} \to \mathcal{M}$, the end of
  \begin{equation*}
    \mathcal{M}^{\op} \times \mathcal{M} \to \mathcal{C},
    \quad (M, M') \mapsto \iHom(M, F(M'))
  \end{equation*}
  exists and a right adjoint of $\Act$ is given by
  \begin{equation*}
    \Act^{\radj}: \REX(\mathcal{M}) \to \mathcal{C},
    \quad F \mapsto \int_{M \in \mathcal{M}} \iHom(M, F(M)).
  \end{equation*}
  We also have:
  \begin{itemize}
  \item [(a)] If $\mathcal{M}$ is exact, then $\Act^{\radj}$ is exact.
  \item [(b)] If $\mathcal{M}$ is exact and indecomposable, then $\Act^{\radj}$ is faithful.
  \end{itemize}
\end{theorem}
\begin{proof}
  Let $\Act^{\radj}$ be a right adjoint of $\Act$. Then we have
  \begin{align*}
    \Hom_{\mathcal{C}}(X, \Act^{\radj}(F))
    & \cong \Nat(\Act(X), F) \\
    & \cong \textstyle \int_{M \in \mathcal{M}} \Hom_{\mathcal{M}}(X \otimes M, F(M)) \\
    & \cong \textstyle \int_{M \in \mathcal{M}} \Hom_{\mathcal{M}}(X, \iHom(M, F(M)))
  \end{align*}
  for all $X \in \mathcal{C}$ and $F \in \REX(\mathcal{M})$. Thus, by Lemma~\ref{lem:existence-end}, we see that the end in question exists and $\Act^{\radj}$ is given as stated.

  (a) We suppose that $\mathcal{M}$ is exact. Let $P$ be a projective generator of $\mathcal{C}$. Then, by Lemma~\ref{lem:action-exact-mod-cat}, the object $\Act(P) \in \REX(\mathcal{M})$ is projective. Thus the functor
  \begin{equation*}
    \Hom_{\mathcal{C}}(P, \Act^{\radj}(-))
    \cong \Hom_{\mathcal{C}}(\Act(P), -)
    : \REX(\mathcal{M}) \to \lmod{k}
  \end{equation*}
  is exact. Since $P$ is a projective generator, we conclude that $\Act^{\radj}$ is exact.

  (b) We suppose that $\mathcal{M}$ is exact and indecomposable. Since the functor $\Act^{\radj}$ is exact by Part (a), it is enough to show that $\Act^{\radj}$ reflects zero objects. Let $F$ be an object of $\REX(\mathcal{M})$ such that $\Act^{\radj}(F) = 0$. By Lemma~\ref{lem:act-fun-dominant}, there is an object $X \in \mathcal{C}$ such that $F$ is an epimorphic image of $\Act(X)$. If $F \ne 0$, then we have
  \begin{equation*}
    0 = \Hom_{\mathcal{C}}(X, \Act^{\radj}(F)) \cong \Nat(\Act(X), F) \ne 0,
  \end{equation*}
  a contradiction. Thus $F = 0$. The proof is done.
\end{proof}

For $M, M' \in \mathcal{M}$, we set $\icoHom(M, M') = {}^* \! \iHom(M', M)$. It is easy to see that there is a natural isomorphism
\begin{equation*}
  \Hom_{\mathcal{M}}(M, X \otimes M') \cong \Hom_{\mathcal{M}}(\icoHom(M, M'), X)
\end{equation*}
for $X \in \mathcal{C}$ and $M, M' \in \mathcal{M}$. A left adjoint of $\ActRex$ is expressed as follows:

\begin{theorem}
  \label{thm:action-left-adj-by-coend}
  For all $k$-linear right exact functor $F: \mathcal{M} \to \mathcal{M}$, the coend of
  \begin{equation*}
    \mathcal{M}^{\op} \times \mathcal{M} \to \mathcal{C},
    \quad
    (M^{\op}, M') \mapsto \icoHom(M, F(M'))
  \end{equation*}
  exists. Moreover, a left adjoint of $\Act$ is given by
  \begin{equation*}
    \ActRex^{\ladj}: \REX(\mathcal{M}) \to \mathcal{C},
    \quad F \mapsto \int^{M \in \mathcal{M}} \icoHom(M, F(M))
  \end{equation*}
  We also have:
  \begin{itemize}
  \item [(a)] If $\mathcal{M}$ is exact, then $\Act^{\ladj}$ is exact.
  \item [(b)] If $\mathcal{M}$ is exact and indecomposable, then $\Act^{\ladj}$ is faithful.
  \end{itemize}
\end{theorem}
\begin{proof}
  Let $\Act^{\ladj}$ be a left adjoint of $\Act$. Then we have
  \begin{align*}
    \Hom_{\mathcal{C}}(\Act^{\ladj}(F), X)
    & \cong \Nat(F, \Act(X)) \\
    & \cong \textstyle \int_{M \in \mathcal{M}} \Hom_{\mathcal{M}}(F(M), X \otimes M) \\
    & \cong \textstyle \int_{M \in \mathcal{M}} \Hom_{\mathcal{C}}(\icoHom(M, F(M)), X)
  \end{align*}
  for all $X \in \mathcal{C}$ and $F \in \REX(\mathcal{M})$. Thus, by the dual of Lemma~\ref{lem:existence-end}, we see that the coend in question exists and $\Act^{\ladj}$ is given as stated.

  Suppose that $\mathcal{M}$ is exact. Then, since $\mathcal{M}$ is Frobenius, the Nakayama functor of $\REX(\mathcal{M}) \approx \mathcal{M}^{\op} \boxtimes \mathcal{M}$ is an equivalence. Parts (a) and (b) of this theorem follow from Theorem~\ref{thm:action-adj-by-end} and $\ActRex^{\ladj} \cong \Nak_{\mathcal{C}}^{-1} \circ \ActRex^{\radj} \circ \Nak_{\REX(\mathcal{M})}$.
\end{proof}

\begin{remark}
  \label{rem:lex-version}
  In summary, for $F \in \REX(\mathcal{M})$, we have
  \begin{equation*}
    \ActRex^{\radj}(F) = \int_{M \in \mathcal{M}} \iHom(M, F(M))
    \quad \text{and} \quad
    \ActRex^{\ladj}(F) = \int^{M \in \mathcal{M}} \icoHom(M, F(M)).
  \end{equation*}
  There is a left exact version of the action functor
  \begin{equation*}
    \ActLex_{\mathcal{M}}: \mathcal{C} \to \LEX(\mathcal{M}) := \LEX(\mathcal{M}, \mathcal{M}),
    \quad \ActLex_{\mathcal{M}}(X)(M) = X \otimes M.
  \end{equation*}
  By the same way as above, one can prove that $\ActLex = \ActLex_{\mathcal{M}}$ is a $k$-linear exact functor and its adjoints are given by
  \begin{equation*}
    \ActLex^{\radj}(F)
    = \int_{M \in \mathcal{C}} \iHom(M, F(M))
    \quad \text{and} \quad
    \ActLex^{\ladj}(F)
    = \int^{M \in \mathcal{C}} \icoHom(M, F(M))
  \end{equation*}
  for $F \in \LEX(\mathcal{M})$. Moreover, there are natural isomorphisms
  \begin{equation}
    \label{eq:action-adj-rex-lex}
    \ActLex^{\radj}(F) \cong {}^* (\ActRex^{\ladj}(F^{\ladj}))
    \quad \text{and} \quad
    \ActLex^{\ladj}(F) \cong {}^* (\ActRex^{\radj}(F^{\ladj}))
  \end{equation}
  for $F \in \LEX(\mathcal{M})$. Indeed, we have natural isomorphisms
  \begin{gather*}
    \Hom_{\mathcal{C}}(X, {}^* (\ActRex^{\ladj}(F^{\ladj})))
    \cong \Hom_{\mathcal{C}}(\ActRex^{\ladj}(F^{\ladj}), X^*)
    \cong \Nat(F^{\ladj}, \ActRex(X^*)) \\
    \cong \Nat(F^{\ladj}, \ActRex(X)^{\ladj})
    \cong \Nat(\ActRex(X), F)
    = \Nat(\ActLex(X), F)
  \end{gather*}
  for $F \in \LEX(\mathcal{M})$ and $X \in \mathcal{C}$. The second isomorphism of \eqref{eq:action-adj-rex-lex} is established in a similar way. Theorems~\ref{thm:action-adj-by-end} and~\ref{thm:action-left-adj-by-coend} imply the following results:
  \begin{itemize}
  \item [(a)] If $\mathcal{M}$ is exact, then $\ActLex^{\ladj}$ and $\ActLex^{\radj}$ are exact.
  \item [(b)] If $\mathcal{M}$ is exact and indecomposable, then $\ActLex^{\ladj}$ and $\ActLex^{\radj}$ are faithful.
  \end{itemize}
\end{remark}

\subsection{The unit and the counit of $(\Act, \Act^{\radj})$}

In what follows, we concentrate to study the structures of the right adjoint of $\Act = \Act_{\mathcal{M}}$. For this purpose, it is useful to describe the unit and the counit of the adjunction $\Act \dashv \Act^{\radj}$. For $F \in \REX(\mathcal{M})$ and $M \in \mathcal{M}$, we denote by
\begin{equation}
  \pi_{F}(M): \Act^{\radj}(F) \to \iHom(M, F(M))
\end{equation}
the universal dinatural transformation and define
\begin{equation}
  \label{eq:act-fun-adj-counit}
  \varepsilon_{F, M} = \ieval_{M, F(M)} \circ (\pi_{F}(M) \otimes \id_M).
\end{equation}
By the proof of Theorem~\ref{thm:action-adj-by-end}, the adjunction isomorphism
\begin{equation}
  \label{eq:act-fun-adj-iso}
  \Hom_{\mathcal{C}}(X, \Act^{\radj}(F)) \xrightarrow{\quad \cong \quad}
  \Hom_{\REX(\mathcal{M})}(\Act(X), F) = \Nat(\Act(X), F)
\end{equation}
sends $a \in \Hom_{\mathcal{C}}(X, \Act^{\radj}(F))$ to the natural transformation $\widetilde{a}$ given by
\begin{equation}
  \label{eq:act-fun-adj-mate-1}
  \widetilde{a}_{M} = \varepsilon_{F, M} \circ (a \otimes \id_M)
  \quad (M \in \mathcal{M})
\end{equation}
This implies that $\varepsilon = \{ \varepsilon_{F, M} \}_{F, M}$ is the counit of \eqref{eq:act-fun-adj-iso}. We also observe that the morphism $a$ is characterized by the property that the equation
\begin{equation}
  \label{eq:act-fun-adj-mate-2}
  \iHom(M, \widetilde{a}_M) \circ \icoev_{X,M} = \pi_F(M) \circ a
\end{equation}
holds for all objects $M \in \mathcal{M}$. Let $\eta: \id_{\mathcal{C}} \to \Act^{\radj} \circ \Act$ be the unit of the adjunction isomorphism \eqref{eq:act-fun-adj-iso}. By substituting $a = \eta_X$ and $F = \Act(X)$ into \eqref{eq:act-fun-adj-mate-2}, we see that $\eta$ is characterized by the property that the equation
\begin{equation}
  \label{eq:act-fun-adj-unit}
  \pi_{\Act(X)}(M) \circ \eta_X = \icoev_{X, M}
\end{equation}
holds for all objects $X \in \mathcal{C}$ and $M \in \mathcal{M}$. We also have
\begin{equation}
  \label{eq:act-fun-adj-uni-dinat}
  \pi_{F}(M) = \iHom(M, \varepsilon_{F,M}) \circ \icoev_{\Act^{\radj}(F),M}
\end{equation}
by substituting $X = \Act^{\radj}(F)$ and $a = \id$ into \eqref{eq:act-fun-adj-mate-2}.

\subsection{Bimodule structure of $\Act^{\radj}$}

Since $\Act: \mathcal{C} \to \REX(\mathcal{M})$ is a $\mathcal{C}$-bimodule functor, its right adjoint $\Act^{\radj}$ is also a $\mathcal{C}$-bimodule functor such that the unit and the counit are $\mathcal{C}$-bimodule transformations. We denote by
\begin{equation*}
  \xi^{(\ell)}_{X,F}: X \otimes \Act^{\radj}(F) \to \Act^{\radj}(X \otimes F)
  \quad \text{and} \quad
  \xi^{(r)}_{F,X}: \Act^{\radj}(F) \otimes X \to \Act^{\radj}(F \otimes X)
\end{equation*}
the left and the right $\mathcal{C}$-module structure of $\Act^{\radj}$. These morphisms are expressed in terms of the universal dinatural transformation $\pi$ as follows:

\begin{lemma}
  \label{lem:act-fun-adj-C-bimod}
  For all objects $F \in \REX(\mathcal{M})$, $X \in \mathcal{C}$ and $M \in \mathcal{M}$, we have
  \begin{align}
    \label{eq:act-fun-adj-C-bimod-l-def-2}
    \pi_{X \otimes F}(M) \circ \xi_{X,F}^{(\ell)}
    & = \iHomA_{X,M,F(M)} \circ (\id_X \otimes \pi_F(M)), \\
    \label{eq:act-fun-adj-C-bimod-r-def-2}
    \pi_{F \otimes X}(M) \circ \xi_{X,F}^{(r)}
    & = \iHomB_{X,M,F(X \otimes M)}^{\natural} \circ (\pi_F(X \otimes M) \otimes \id_X).
  \end{align}
\end{lemma}

See Subsection~\ref{subsec:cl-mod-cat} for definitions of $\iHomA$ and $\iHomB^{\natural}$. By the universal property of $\Act^{\radj}(F)$ as an end, the isomorphisms $\xi^{(\ell)}_{X,F}$ and $\xi^{(r)}_{F,X}$ are characterized by equations \eqref{eq:act-fun-adj-C-bimod-l-def-2} and~\eqref{eq:act-fun-adj-C-bimod-r-def-2}, respecively. We postpone the proof of this lemma to Appendix~\ref{apdx:act-adj-structure} since it is straightforward but lengthy.

\subsection{Monoidal structure of $\Act^{\radj}$}

Since the action functor $\Act: \mathcal{C} \to \REX(\mathcal{M})$ is a strong monoidal functor, its right adjoint $\Act^{\radj}$ has a canonical structure of a (lax) monoidal functor. We denote the structure morphisms of $\Act^{\radj}$ by
\begin{equation*}
  \mu^{(2)}_{F,G}: \Act^{\radj}(F) \otimes \Act^{\radj}(G) \to \Act^{\radj}(F \circ G)
  \quad \text{and} \quad
  \mu^{(0)}: \unitobj \to \Act^{\radj}(\id_{\mathcal{M}})
\end{equation*}
for $F, G \in \REX(\mathcal{M})$. They are expressed in terms of the universal dinatural transformation $\pi$ as follows:

\begin{lemma}
  \label{lem:act-fun-adj-monoidal}
  For all objects $F, G \in \REX(\mathcal{M})$ and $M \in \mathcal{M}$, we have
  \begin{gather}
    \label{eq:act-fun-adj-monoidal-2}
    \pi_{F G}(M) \circ \mu^{(2)}_{F,G} = \icomp_{M, G(M), F G(M)} \circ (\pi_{F}(G(M)) \otimes \pi_{G}(M)), \\
    \label{eq:act-fun-adj-monoidal-0}
    \pi_{\id_{\mathcal{M}}}(M) \circ \mu^{(0)} = \icoev_{\unitobj, M}.
  \end{gather}
\end{lemma}

By the universal property, $\mu^{(2)}$ and $\mu^{(0)}$ are characterized by equations \eqref{eq:act-fun-adj-monoidal-2} and~\eqref{eq:act-fun-adj-monoidal-0}, respectively. The proof is postponed to Appendix~\ref{apdx:act-adj-structure}.

\subsection{Lifting the adjunction to the Drinfeld center}

Given two finite left $\mathcal{C}$-module categories $\mathcal{M}$ and $\mathcal{N}$, we denote by $\REX_{\mathcal{C}}(\mathcal{M}, \mathcal{N})$ the category of $k$-linear right exact $\mathcal{C}$-module functors from $\mathcal{M}$ to $\mathcal{N}$. The aim of this subsection is to show that the adjoint pair $(\Act, \Act^{\radj})$ can be `lifted' to an adjoint pair between the Drinfeld center of $\mathcal{C}$ and $\REX_{\mathcal{C}}(\mathcal{M}, \mathcal{M})$.

We first introduce the following generalization of the Drinfeld center construction: For a $\mathcal{C}$-bimodule category $\mathcal{M}$, we define the category $\mathcal{Z}(\mathcal{M})$ as follows: An object of this category is a pair $(M, \sigma)$ consisting of an object $M \in \mathcal{M}$ and a natural isomorphism $\sigma_X: M \otimes X \to X \otimes M$ ($X \in \mathcal{C}$) satisfying the equations
\begin{equation*}
  \sigma_{\unitobj} = \id_M
  \quad \text{and} \quad
  \sigma_{X \otimes Y} = (\id_X \otimes \sigma_Y) \circ (\sigma_X \otimes \id_Y)
\end{equation*}
for all objects $X, Y \in \mathcal{C}$. If $\mathbf{M} = (M, \sigma_{M})$ and $\mathbf{N} = (N, \sigma_N)$ are objects of $\mathcal{Z}(\mathcal{M})$, then a morphism $f: \mathbf{M} \to \mathbf{N}$ is a morphism $f: M \to N$ satisfying
\begin{equation*}
  (\id_X \otimes f) \circ \sigma_{M; X} = \sigma_{N; X} \circ (f \otimes \id_X)
\end{equation*}
for all objects $X \in \mathcal{C}$. The composition of morphisms in $\mathcal{Z}(\mathcal{M})$ is defined by the composition of morphisms in $\mathcal{M}$.

\begin{example}
  The category $\mathcal{C}$ is a finite $\mathcal{C}$-bimodule category by the tensor product of $\mathcal{C}$. The category $\mathcal{Z}(\mathcal{C})$ is the {\em Drinfeld center} of $\mathcal{C}$. If this is the case, then $\mathcal{Z}(\mathcal{C})$ is not only a category but a braided finite tensor category \cite{MR2119143}.
\end{example}

\begin{example}
  If $\mathcal{M}$ and $\mathcal{N}$ are finite left $\mathcal{C}$-module categories, then $\mathcal{F} := \REX(\mathcal{M}, \mathcal{N})$ is a finite $\mathcal{C}$-bimodule category by the actions given by \eqref{eq:Rex-action}. The category $\mathcal{Z}(\mathcal{F})$ can be identified with $\REX_{\mathcal{C}}(\mathcal{M}, \mathcal{N})$.
\end{example}

Now let $\mathcal{C}\mbox{-\underline{bimod}}$ be the 2-category of finite $\mathcal{C}$-bimodule categories, $k$-linear right exact $\mathcal{C}$-bimodule functors and $\mathcal{C}$-bimodule natural transformations. Given a 1-cell $F: \mathcal{M} \to \mathcal{N}$ in $\mathcal{C}\mbox{-\underline{bimod}}$ with structure morphisms
\begin{equation*}
  \ell_{X,M}: X \otimes F(M) \to F(X \otimes M)
  \quad \text{and} \quad
  r_{M, X}: F(M) \otimes X \to F(M \otimes X),
\end{equation*}
we define the $k$-linear functor $\mathcal{Z}(F): \mathcal{Z}(\mathcal{M}) \to \mathcal{Z}(\mathcal{N})$ by
\begin{equation*}
  \mathcal{Z}(F)(\mathbf{M}) = (F(M), \ \ell^{-1} \circ F(\sigma) \circ r)
\end{equation*}
for $\mathbf{M} = (M, \sigma)$ in $\mathcal{Z}(\mathcal{M})$. It is routine to check that these constructions extends to a 2-functor $\mathcal{Z}: \mathcal{C}\mbox{-\underline{bimod}} \to k\mbox{-\underline{Cat}}$, where $k\mbox{-\underline{Cat}}$ is the 2-category of essentially small $k$-linear categories, $k$-linear functors and natural transformations.

We apply the 2-functor $\mathcal{Z}$ to the action functor and its adjoint. Let $\mathcal{M}$ be a finite left $\mathcal{C}$-module category. Since $\Act = \Act_{\mathcal{M}}$ is a $\mathcal{C}$-bimodule functor, its right adjoint $\Act^{\radj}$ is a $\mathcal{C}$-bimodule functor in such a way that the unit and the counit of the adjunction are bimodule natural transformations. Namely, there is an adjoint pair $(\Act, \Act^{\radj})$ in the 2-category $\mathcal{C}\mbox{-\underline{bimod}}$. By applying the 2-functor $\mathcal{Z}$, we obtain:

\begin{theorem}
  \label{thm:act-fun-adj-Z}
  There is an adjoint pair
  \begin{equation}
    \label{eq:act-fun-adj-Z}
    \big( \mathcal{Z}(\Act): \mathcal{Z}(\mathcal{C}) \to \REX_{\mathcal{C}}(\mathcal{M}),
    \ \mathcal{Z}(\Act^{\radj}): \REX_{\mathcal{C}}(\mathcal{M}) \to \mathcal{Z}(\mathcal{C}) \big),
  \end{equation}
  where we have identified $\mathcal{Z}(\REX(\mathcal{M}))$ with $\REX_{\mathcal{C}}(\mathcal{M})$. 
\end{theorem}

It is instructive to describe the functors $\mathcal{Z}(\Act)$ and $\mathcal{Z}(\Act^{\radj})$ explicitly. Given an object $\mathbf{X} = (X, \sigma) \in \mathcal{Z}(\mathcal{C})$, we have $\mathcal{Z}(\Act)(\mathbf{X}) = \Act(X)$. The left $\mathcal{C}$-module structure of $\mathcal{X} := \mathcal{Z}(\Act)(\mathbf{X})$ is given by
\begin{align*}
  (\sigma_{W})^{-1} \otimes \id_M: W \otimes \mathcal{X}(M)
  = W \otimes X \otimes M
  \to X \otimes W \otimes M
  = \mathcal{X}(W \otimes M)
\end{align*}
for $W \in \mathcal{C}$ and $M \in \mathcal{M}$. For an object $\mathbf{F} = (F, s) \in \REX_{\mathcal{C}}(\mathcal{M})$, we have
\begin{equation*}
  \mathcal{Z}(\Act^{\radj})(\mathbf{F}) = (\Act^{\radj}(F), \sigma^{\mathbf{F}}),
  \quad \text{where} \quad
  \sigma_{X}^{\mathbf{F}} = (\xi^{(\ell)}_{X,F})^{-1} \circ \Act^{\radj}(s^{-1}) \circ \xi_{F,X}^{(r)}
\end{equation*}
for $X \in \mathcal{C}$. More explicitly:

\begin{lemma}
  \label{lem:act-fun-adj-half-braiding}
  The half-braiding $\sigma^{\mathbf{F}}$ is a unique natural transformation such that the following diagram commutes for all objects $X \in \mathcal{C}$ and $M \in \mathcal{M}$.
  \begin{equation*}
    \xymatrix@C=90pt@R=16pt{
      \Act^{\radj}(F) \otimes X
      \ar[r]^-{\pi_F(X \otimes M) \otimes \id}
      \ar[ddd]_{\sigma^{\mathbf{F}}_X}
      & \iHom(X \otimes M, F(X \otimes M)) \otimes X
      \ar[d]^{\ \iHomB_{X, M, F(X \otimes M)}} \\
      & \iHom(M, F(X \otimes M))
      \ar[d]^{\ \iHom(M, s^{-1})} \\
      & \iHom(M, X \otimes F(M))
      \ar[d]^{\ \iHomA_{X,M,F(M)}^{-1}} \\
      X \otimes \Act^{\radj}(F)
      \ar[r]^-{\id \otimes \pi_F(X)}
      & X \otimes \iHom(M, F(M)).
    }
  \end{equation*}
\end{lemma}
\begin{proof}
  The commutativity of this diagram follows from Lemma~\ref{lem:act-fun-adj-C-bimod}. By the Fubini theorem for ends, we see that $X \otimes \Act^{\radj}(F)$ is an end of the functor
  \begin{equation*}
    \mathcal{M}^{\op} \times \mathcal{M} \to \mathcal{C},
    \quad (M^{\op}, M') \mapsto X \otimes \iHom(M, F(M')).
  \end{equation*}
  The universal property proves the `uniqueness' part of this lemma.
\end{proof}

\subsection{Induction to the Drinfeld center}

The $k$-linear monoidal category
\begin{equation*}
  \mathcal{C}_{\mathcal{M}}^* := \REX_{\mathcal{C}}(\mathcal{M})^{\rev}
\end{equation*}
is called the {\em dual} of $\mathcal{C}$ with respect to $\mathcal{M}$. By Schauenburg's result \cite{MR1822847} (which we recall later), there is an equivalence $\mathcal{Z}(\mathcal{C}) \approx \mathcal{Z}(\mathcal{C}_{\mathcal{M}}^*)$ of $k$-linear braided monoidal categories. In this subsection, we show that $\mathcal{Z}(\Act_{\mathcal{M}}^{\radj})$ is right adjoint to the composition
\begin{equation*}
  \mathcal{Z}(\mathcal{C})
  \xrightarrow{\quad \text{Schauenburg's equivalence} \quad} \mathcal{Z}(\mathcal{C}_{\mathcal{M}}^*)
  \xrightarrow{\quad \text{the forgetful functor} \quad} \mathcal{C}_{\mathcal{M}}^*.
\end{equation*}

\newcommand{\Sch}{\uptheta}

We first recall Schauenburg's result \cite{MR1822847} on the Drinfeld center of the category of bimodules. Let $A$ be an algebra in $\mathcal{C}$ with multiplication $m: A \otimes A \to A$ and unit $u: \unitobj \to A$. Then the category ${}_A \mathcal{C}_A$ of $A$-bimodules in $\mathcal{C}$ is a $k$-linear abelian monoidal category with respect to the tensor product over $A$. There is a $k$-linear braided strong monoidal functor $\Sch_A: \mathcal{Z}(\mathcal{C}) \to \mathcal{Z}({}_A \mathcal{C}_A)$ defined as follows: For an object $\mathbf{V} = (V, \sigma) \in \mathcal{Z}(\mathcal{C})$, we set $\Sch_A(\mathbf{V}) = (A \otimes V, \widetilde{\sigma})$, where the left action of $A$ on $A \otimes V$ is given by $m \otimes \id_V$, the right action is given by the composition
\begin{equation*}
  (A \otimes V) \otimes A
  \xrightarrow{\quad \id_A \otimes \sigma_A \quad}
  A \otimes A \otimes V
  \xrightarrow{\quad m \otimes \id_V \quad}
  A \otimes V,
\end{equation*}
and the half-braiding $\widetilde{\sigma}$ is determined by the commutative diagram
\begin{equation*}
  \xymatrix@C=48pt@R=16pt{
    (A \otimes V) \otimes_A M
    \ar[rr]^{\widetilde{\sigma}_{M}} & & M \otimes_A (A \otimes M) \\
    A \otimes V \otimes M \ar[r]^{\id_A \otimes \sigma_M}
    \ar@{>>}[u]
    & X \otimes A \otimes M \ar[r]^{\id_X \otimes \, \triangleright_M}
    & M \otimes X \ar[u]_{\cong}
  }
\end{equation*}
for an $A$-bimodule $M$ in $\mathcal{C}$ with left action $\triangleright_M: A \otimes M \to M$. For a morphism $f$ in $\mathcal{Z}(\mathcal{C})$, we set $\Sch_A(f) = \id_A \otimes f$. The monoidal structure of $\Sch_A$ is given by the canonical isomorphism
\begin{equation*}
  \Sch_A(\mathbf{V}) \otimes_A \Sch_A(\mathbf{W}) = (A \otimes V) \otimes_A (A \otimes W) \cong A \otimes (V \otimes W) = \Sch_A(\mathbf{V} \otimes \mathbf{W})
\end{equation*}
for $\mathbf{V} = (V, \sigma), \mathbf{W} = (W, \tau) \in \mathcal{Z}(\mathcal{C})$. Schauenburg \cite{MR1822847} showed that the functor $\Sch_A$ is in fact an equivalence of $k$-linear braided monoidal categories.

Now let $\mathcal{M}$ be a finite left $\mathcal{C}$-module category. Then there is an algebra $A$ in $\mathcal{C}$ such that $\mathcal{M} \approx \mathcal{C}_A$ as a left $\mathcal{C}$-module categories. Moreover, the functor
\begin{equation}
  \label{eq:enriched-EW}
  {}_A \mathcal{C}_A \to \REX_{\mathcal{C}}(\mathcal{C}_A, \mathcal{C}_A)^{\rev},
  \quad M \mapsto (-) \otimes_A M
\end{equation}
is an equivalence of $k$-linear monoidal categories. Thus $\mathcal{Z}(\mathcal{C}_{\mathcal{M}}^*)$ and $\mathcal{Z}(\mathcal{C})$ are equivalent as $k$-linear braided monoidal categories.

Since we are interested in the general theory of finite tensor categories and their module categories, it is preferable to describe the equivalence $\mathcal{Z}(\mathcal{C}) \approx \mathcal{Z}(\mathcal{C}_{\mathcal{M}}^*)$ without referencing the algebra $A$ such that $\mathcal{M} \approx \mathcal{C}_A$. Thus, for a finite left $\mathcal{C}$-module category $\mathcal{M}$, we define the functor $\Sch_{\mathcal{M}}: \mathcal{Z}(\mathcal{C}) \to \mathcal{Z}(\mathcal{C}_{\mathcal{M}}^*)$ as follows: For an object $\mathbf{V} = (V, \sigma) \in \mathcal{Z}(\mathcal{C})$, we set $\Sch_{\mathcal{M}}(\mathbf{V}) = \Act(V)$ as an object of $\REX(\mathcal{M})$. We make $\Act(V)$ into a left $\mathcal{C}$-module functor by the structure morphism given by
\begin{equation*}
  (\sigma_{X})^{-1} \otimes \id_M:
  X \otimes \Act(V)(M)
  = X \otimes V \otimes M
  \to V \otimes X \otimes M
  = \Act(V)(X \otimes M)
\end{equation*}
for $X \in \mathcal{C}$ and $M \in \mathcal{M}$. The half-braiding of $\Sch_{\mathcal{M}}(\mathbf{V})$ is given by
\begin{equation*}
  s_{V,M}: (\Sch_{\mathcal{M}}(\mathbf{V}) \circ \mathbf{F})(M)
  = V \otimes F(M) \to F(V \otimes M)
  = (\mathbf{F} \circ \Sch_{\mathcal{M}}(\mathbf{V}))(M)
\end{equation*}
for $\mathbf{F} = (F, s) \in \mathcal{C}_{\mathcal{M}}^*$ and $M \in \mathcal{M}$. The following theorem is obtained by rephrasing Schauenburg's result.

\begin{theorem}
  \label{thm:categorical-Schauenburg}
  The functor $\Sch_{\mathcal{M}}: \mathcal{Z}(\mathcal{C}) \to \mathcal{Z}(\mathcal{C}_{\mathcal{M}}^*)$ is an equivalence of $k$-linear braided monoidal categories.
\end{theorem}
\begin{proof}
  If finite left $\mathcal{C}$-module categories $\mathcal{M}$ and $\mathcal{N}$ are equivalent, then there is an equivalence $F: \mathcal{C}_{\mathcal{M}}^* \to \mathcal{C}_{\mathcal{N}}^*$ of $k$-linear monoidal categories. It is easy to check
  \begin{equation*}
    \widetilde{F} \circ \Sch_{\mathcal{M}} = \Sch_{\mathcal{N}},
  \end{equation*}
  where $\widetilde{F}: \mathcal{Z}(\mathcal{C}_{\mathcal{M}}^*) \to \mathcal{Z}(\mathcal{C}_{\mathcal{N}}^*)$ is the braided monoidal equivalence induced by the monoidal equivalence $F$. Thus, to show that $\Sch_{\mathcal{M}}$ is an equivalence, we may assume that $\mathcal{M} = \mathcal{C}_A$ for some algebra $A$ in $\mathcal{C}$. We consider the equivalence
  \begin{equation*}
    \Sch'_{\mathcal{M}} := \Big( \mathcal{Z}(\mathcal{C})
    \xrightarrow{\quad \Sch_A \quad} \mathcal{Z}({}_A \mathcal{C}_A)
    \xrightarrow{\quad \text{by \eqref{eq:enriched-EW}} \quad} \mathcal{Z}(\mathcal{C}_{\mathcal{M}}^*) \Big)
  \end{equation*}
  of $k$-linear braided monoidal categories. One can check that $\Sch_{\mathcal{M}} \cong \Sch_{\mathcal{M}}'$ as monoidal functors via the isomorphism given by
  \begin{equation*}
    \Sch_{\mathcal{M}}(\mathbf{V})(M)
    = V \otimes M
    \xrightarrow{\quad \sigma_{M} \quad}
    M \otimes V
    \xrightarrow{\quad \cong \quad} M \otimes_A (A \otimes V)
    = \Sch_{\mathcal{M}}'(\mathbf{V})(M)
  \end{equation*}
  for $\mathbf{V} = (V, \sigma) \in \mathcal{Z}(\mathcal{C})$ and $M \in \mathcal{M}$. Thus $\Sch_{\mathcal{M}}$ is also an equivalence of $k$-linear braided monoidal categories.
\end{proof}

Now we prove the result mentioned at the beginning of this subsection:

\begin{theorem}
  \label{thm:induction-Drinfeld-center}
  Let $\mathsf{U}: \mathcal{Z}(\mathcal{C}_{\mathcal{M}}^*) \to \mathcal{C}_{\mathcal{M}}^*$ be the forgetful functor. Then
  \begin{equation*}
    (\mathsf{U} \circ \Sch_{\mathcal{M}}: \mathcal{Z}(\mathcal{C}) \to \mathcal{C}_{\mathcal{M}}^*,
    \ \mathcal{Z}(\Act_{\mathcal{M}}^{\radj}): \mathcal{C}_{\mathcal{M}}^* \to \mathcal{Z}(\mathcal{C}))
  \end{equation*}
  is an adjoint pair.
\end{theorem}
\begin{proof}
  By Theorem \ref{thm:categorical-Schauenburg}, the functor $\mathsf{U} \circ \Sch_{\mathcal{M}}$ is identical to $\mathcal{Z}(\Act_{\mathcal{M}})$ and therefore it is left adjoint to $\mathcal{Z}(\Act_{\mathcal{M}}^{\radj})$.
\end{proof}

\begin{corollary}
  Let $\mathsf{U}_{\mathcal{C}}: \mathcal{Z}(\mathcal{C}) \to \mathcal{C}$ and $\mathsf{U}_{\mathcal{D}}: \mathcal{Z}(\mathcal{D}) \to \mathcal{D}$ be the forgetful functors, where $\mathcal{D} = \mathcal{C}_{\mathcal{M}}^*$. Then $\mathsf{U}_{\mathcal{D}}$ has a right adjoint. The composition
  \begin{equation*}
    \mathcal{D}
    \xrightarrow{\quad \mathsf{U}_{\mathcal{D}}^{\radj} \quad}
    \mathcal{Z}(\mathcal{D})
    \xrightarrow{\quad \Sch_{\mathcal{M}}^{-1} \quad}
    \mathcal{Z}(\mathcal{C})
    \xrightarrow{\quad \mathsf{U}_{\mathcal{C}} \quad} \mathcal{C}
  \end{equation*}
  sends an object $\mathbf{F} = (F, s) \in \mathcal{D}$ to the end $\int_{M \in \mathcal{M}} \iHom(M, F(M))$.
\end{corollary}
\begin{proof}
  Theorem \ref{thm:induction-Drinfeld-center} implies that $\Sch_{\mathcal{M}} \circ \mathcal{Z}(\Act_{\mathcal{M}}^{\radj})$ is right adjoint to $\mathsf{U}_{\mathcal{D}}$. Thus $\mathsf{U}_{\mathcal{D}}^{\radj}$ exists and is isomorphic to $\Sch_{\mathcal{M}} \circ \mathcal{Z}(\Act_{\mathcal{M}}^{\radj})$. Hence the composition in question is isomorphic to $\mathsf{U}_{\mathcal{C}} \circ \mathcal{Z} (\Act_{\mathcal{M}}^{\radj})$. Now the result follows from the explicit description of $\mathcal{Z}(\Act_{\mathcal{M}}^{\radj})$ given in the previous subsection.
\end{proof}

\section{Integral over a topologizing full subcategory}
\label{sec:integral-over-fullsub}

\newcommand{\Top}{\mathfrak{Top}}

\subsection{Topologizing full subcategory}
\label{subsec:topolo-full-sub}

We first introduce the following terminology and notation: A full subcategory of an abelian category is said to be {\em topologizing} \cite{MR1347919} if it is closed under finite direct sums and subquotients. We denote by $\Top(\mathcal{A})$ the class of topologizing full subcategories of an abelian category $\mathcal{A}$.

Let $\mathcal{M}$ be a finite module category over a finite tensor category. In Section~\ref{sec:adj-of-act}, we have considered several `integrals' over the category $\mathcal{M}$. In this section, based on our results on adjoints of the action functor, we extend techniques used in \cite{MR3631720} and provide a framework to deal with `integrals' of the form $\int_{X \in \mathcal{S}} \iHom(X, X)$ for some $\mathcal{S} \in \Top(\mathcal{M})$.

We first summarize basic results on topologizing full subcategories of a finite abelian category. Let $\mathcal{M}$ be a finite abelian category, and let $\mathcal{S}$ be a topologizing full subcategory of $\mathcal{M}$ with inclusion functor $i: \mathcal{S} \to \mathcal{M}$. For $M \in \mathcal{M}$, we set
\begin{equation}
  \label{eq:coreflector-def}
  i^{\sharp}(M) = \text{(the largest subobject of $M$ belonging to $\mathcal{S}$)}.
\end{equation}
By the assumption that $\mathcal{S}$ is a topologizing full subcategory, one can extend the assignment $M \mapsto i^{\sharp}(M)$ to a $k$-linear functor from $\mathcal{M}$ to $\mathcal{S}$. Dually, we set
\begin{equation}
  \label{eq:reflector-def}
  \upkappa_{\mathcal{S}}(M) = \bigcap \, \{ X \subset M \mid M/X \in \mathcal{S} \}
  \quad \text{and} \quad
  i^{\flat}(M) = M / \upkappa_{\mathcal{S}}(M)
\end{equation}
for $M \in \mathcal{M}$. One can also extend the assignment $M \mapsto i^{\flat}(M)$ to a $k$-linear functor from $\mathcal{M}$ to $\mathcal{S}$. It is easy to see that $i^{\sharp}$ and $i^{\flat}$ are a right and a left adjoint of $i$, respectively. We now define
\begin{equation}
  \label{eq:idempo-monad-def}
  \uptau_{\mathcal{S}} := i \circ i^{\flat}
  \quad \text{and} \quad
  \uptau'_{\mathcal{S}} := i \circ i^{\sharp}.
\end{equation}
Since $i^{\flat} \dashv i \dashv i^{\sharp}$, we have natural isomorphisms
\begin{equation}
  \label{eq:reflector-2}
  \Hom_{\mathcal{M}}(\uptau_{\mathcal{S}} (M), N)
  \cong \Hom_{\mathcal{S}}(i^{\flat}(M), i^{\sharp}(N))
  \cong \Hom_{\mathcal{M}}(M, \uptau'_{\mathcal{S}}(N))
\end{equation}
for $M, N \in \mathcal{M}$. Thus $\uptau_{\mathcal{S}} \in \REX(\mathcal{M})$ and $\uptau'_{\mathcal{S}} = \uptau_{\mathcal{S}}^{\radj}$. Moreover, since $i^{\flat} \circ i = \id_{\mathcal{S}}$, the endofunctor $\uptau_{\mathcal{S}}$ is an idempotent monad on $\mathcal{M}$ whose category of modules coincides with $\mathcal{S}$. By this observation, we have the following consequence:

\begin{lemma}
  A topologizing full subcategory of a finite abelian category is a finite abelian category such that the inclusion functor preserves and reflects exact sequences.
\end{lemma}

Now we choose a finite-dimensional algebra $A$ such that $\mathcal{M} \approx \lmod{A}$. If we identify $\REX(\mathcal{M})$ with $\bimod{A}{A}$, then $\id_{\mathcal{M}} \in \REX(\mathcal{M})$ corresponds to the $A$-bimodule $A$. Thus a subobject of $\id_{\mathcal{M}}$ in $\REX(\mathcal{M})$ corresponds to an ideal of $A$. By abuse of terminology, we call a subobject of $\id_{\mathcal{M}}$ in $\REX(\mathcal{M})$ an {\em ideal} of $\mathcal{M}$. Then we have the following correspondence ({\it cf}. Rosenberg \cite[Chapter III]{MR1347919}):

\begin{lemma}
  \label{lem:Rosen}
  For a finite abelian category $\mathcal{M}$, there is a one-to-one correspondence between the class $\Top(\mathcal{M})$ and the set of ideals of $\mathcal{M}$.
\end{lemma}

For $\mathcal{S} \in \Top(\mathcal{M})$, we define $\upkappa_{\mathcal{S}}$ by \eqref{eq:reflector-def}. The correspondence of the above lemma assigns $\upkappa_{\mathcal{S}} \subset \id_{\mathcal{S}}[$ to $\mathcal{S}$. Conversely, given an ideal $I$ of $\mathcal{M}$, we consider the quotient $\tau := \id_{\mathcal{M}}/I$. If we identify $\mathcal{M}$ with $\lmod{A}$ as above, then $I$ can be regarded as an ideal of the algebra $A$ and the functor $\tau$ is identified with $(A/I) \otimes_A (-)$. Thus $\tau$ is a $k$-linear right exact idempotent monad on $\mathcal{M}$. The correspondence of Lemma~\ref{lem:Rosen} assigns the category of $\tau$-modules to the ideal $I$.

\subsection{Integral over a full subcategory}
\label{subsec:integral-over}

Let $\mathcal{C}$ be a finite tensor category, and let $\mathcal{M}$ be a finite left $\mathcal{C}$-module category. Given $\mathcal{S} \in \Top(\mathcal{M})$, we consider the end
\begin{equation*}
  A'_{\mathcal{S}} := \Act^{\radj}_{\mathcal{M}}(\uptau_{\mathcal{S}}) = \int_{M \in \mathcal{M}} \iHom(M, \uptau_{\mathcal{S}}(M)),
\end{equation*}
where $\uptau_{\mathcal{S}}$ is defined by \eqref{eq:idempo-monad-def}. Let $i: \mathcal{S} \to \mathcal{M}$ be the inclusion functor. By applying Lemma~\ref{lem:end-adj} to the adjunction $i^{\flat} \dashv i$, we see that the end of the functor
\begin{equation}
  \label{eq:adjoint-alg-of-S-1}
  \mathcal{S}^{\op} \times \mathcal{S} \to \mathcal{C},
  \quad (X, X') \mapsto \iHom(i(X), i(X'))
\end{equation}
exists and is canonically isomorphic to $A'_{\mathcal{S}}$. We denote the end of \eqref{eq:adjoint-alg-of-S-1} by
\begin{equation*}
  A_{\mathcal{S}} = \int_{X \in \mathcal{S}} \iHom(X, X)
\end{equation*}
with omitting the inclusion functor. Let $\beta_{\mathcal{S}}: A'_{\mathcal{S}} \to A_{\mathcal{S}}$ be the canonical isomorphism given by Lemma~\ref{lem:end-adj}. If we denote by
\begin{equation*}
  \pi_{\mathcal{S}}(X): A_{\mathcal{S}} \to \iHom(X, X)
  \quad \text{and} \quad
  \quad \pi'_{\mathcal{S}}(M): A'_{\mathcal{S}} \to \iHom(M, \uptau_{\mathcal{S}}(M))
\end{equation*}
the respective universal dinatural transformations, then the isomorphism $\beta_{\mathcal{S}}$ is characterized as a unique morphism in $\mathcal{C}$ such that the equation
\begin{equation}
  \label{eq:end-adj-canonical-iso}
  \pi_{\mathcal{S}}(X) \circ \beta_{\mathcal{S}} = \pi'_{\mathcal{S}}(X)
\end{equation}
holds for all $X \in \mathcal{S}$.

We recall that $\uptau_{\mathcal{S}}$ is an idempotent monad on $\mathcal{M}$. Thus $A_{\mathcal{S}}'$ is an algebra in $\mathcal{C}$ as the image of an algebra under the monoidal functor $\uprho_{\mathcal{M}}^{\radj}$. On the other hand, by the universal property of the end $A_{\mathcal{S}}$, we can define
\begin{equation}
  \label{eq:A_S-unit-multiplication}
  m_{\mathcal{S}}: A_{\mathcal{S}} \otimes A_{\mathcal{S}} \to A_{\mathcal{S}}
  \quad \text{and} \quad
  u_{\mathcal{S}}: \unitobj \to A_{\mathcal{S}}
\end{equation}
to be unique morphisms such that the equations
\begin{equation*}
  \pi_{\mathcal{S}}(X) \circ m_{\mathcal{S}} = \icomp^{\mathcal{M}}_{X, X, X} \circ (\pi_{\mathcal{S}}(X) \otimes \pi_{\mathcal{S}}(X))
  \quad \text{and} \quad
  \pi_{\mathcal{S}}(X) \circ u_{\mathcal{S}} = \icoev_{\unitobj, X}
\end{equation*}
hold for all objects $X \in \mathcal{S}$. It is easy to see that $A_{\mathcal{S}}$ is an algebra in $\mathcal{C}$ with multiplication $m_{\mathcal{S}}$ and unit $u_{\mathcal{S}}$.

\begin{lemma}
  \label{lem:end-adj-alg-iso}
  The morphism $\beta_{\mathcal{S}}$ is an isomorphism of algebras in $\mathcal{C}$.
\end{lemma}
\begin{proof}
  Noting $\uptau_{\mathcal{S}}(X) = X$ for all $X \in \mathcal{S}$, we easily verify that the equations
  \begin{gather*}
    \pi_{\mathcal{S}}(X) \circ \beta_{\mathcal{S}} \circ \mu^{(2)}_{\uptau_{\mathcal{S}}, \uptau_{\mathcal{S}}}
    = \icomp_{X,X,X}
    = \pi_{\mathcal{S}}(X) \circ m_{\mathcal{S}} \circ (\beta_{\mathcal{S}} \otimes \beta_{\mathcal{S}}), \\
    \pi_{\mathcal{S}}(X) \circ \beta_{\mathcal{S}} \circ \mu^{(0)}
    = \icoev_{\unitobj, X} = \pi_{\mathcal{S}}(X) \circ u_{\mathcal{S}}
  \end{gather*}
  hold for all objects $X \in \mathcal{S}$. By the universal property of the end $A_{\mathcal{S}}$, we conclude that $\beta_{\mathcal{S}}$ is a morphism of algebras.
\end{proof}

For $\mathcal{S} \in \Top(\mathcal{M})$, we denote by $q_{\mathcal{S}}: \id_{\mathcal{M}} \to \uptau_{\mathcal{S}}$ the quotient morphism. We recall that the kernel of $q_{\mathcal{S}}$ is $\upkappa_{\mathcal{S}}$. For $\mathcal{S}_1, \mathcal{S}_2 \in \Top(\mathcal{M})$ with $\mathcal{S}_1 \supset \mathcal{S}_2$, we have $\upkappa_{\mathcal{S}_1} \subset \upkappa_{\mathcal{S}_2}$ as subobjects of $\id_{\mathcal{M}}$. Thus there is a unique morphism $q_{\mathcal{S}_1|\mathcal{S}_2}: \uptau_{\mathcal{S}_1} \to \uptau_{\mathcal{S}_2}$ such that $q_{\mathcal{S}_1|\mathcal{S}_2} \circ q_{\mathcal{S}_1} = q_{\mathcal{S}_2}$.

For $\mathcal{S}_1, \mathcal{S}_2$ with $\mathcal{S}_1 \supset \mathcal{S}_2$, we also define a morphism $\phi_{\mathcal{S}_1 | \mathcal{S}_2}: A_{\mathcal{S}_1} \to A_{\mathcal{S}_2}$ to be a unique morphism such that the equation
\begin{equation}
  \label{eq:canonical-epi}
  \pi_{\mathcal{S}_2}(X) \circ \phi_{\mathcal{S}_1 | \mathcal{S}_2} = \pi_{\mathcal{S}_1}(X)
\end{equation}
holds for all objects $X \in \mathcal{S}_2$.

\begin{lemma}
  \label{lem:canonical-epi}
  With the above notation, we have
  \begin{equation*}
    \phi_{\mathcal{S}_1|\mathcal{S}_2} \circ \beta_{\mathcal{S}_1}
    = \beta_{\mathcal{S}_2} \circ \uprho_{\mathcal{M}}^{\radj}(q_{\mathcal{S}_1|\mathcal{S}_2}).
  \end{equation*}
\end{lemma}
\begin{proof}
  For all objects $X \in \mathcal{S}_2$, we have
  \begin{gather*}
    \pi_{\mathcal{S}_2}(X) \circ \phi_{\mathcal{S}_1|\mathcal{S}_2} \circ \beta_{\mathcal{S}_1}
    = \pi_{\mathcal{S}_1}(X) \circ \beta_{\mathcal{S}_1}
    = \pi'_{\mathcal{S}_1}(X)
  \end{gather*}
  by~\eqref{eq:end-adj-canonical-iso} and~\eqref{eq:canonical-epi}. Noting $\uptau_{\mathcal{S}_1}(X) = X$ and $(q_{\mathcal{S}_1|\mathcal{S}_2})_X = \id_X$, we also have
  \begin{align*}
    \pi_{\mathcal{S}_2}(X) \circ \beta_{\mathcal{S}_2} \circ \uprho_{\mathcal{M}}^{\radj}(q_{\mathcal{S}_1|\mathcal{S}_2})
    & = \pi'_{\mathcal{S}_2}(X) \circ \uprho_{\mathcal{M}}^{\radj}(q_{\mathcal{S}_1|\mathcal{S}_2}) \\
    & = \iHom(\id_X, (q_{\mathcal{S}_1|\mathcal{S}_2})_X) \circ \pi'_{\mathcal{S}_2}(X)
    = \pi'_{\mathcal{S}_2}(X).
  \end{align*}
  The claim follows from the universal property of $A_{\mathcal{S}_2}$.
\end{proof}

For $\mathcal{S}_1, \mathcal{S}_2, \mathcal{S}_3 \in \Top(\mathcal{M})$ with $\mathcal{S}_1 \supset \mathcal{S}_2 \supset \mathcal{S}_3$, we have
\begin{equation}
  \label{eq:canonical-epi-transitive}
  q_{\mathcal{S}_2|\mathcal{S}_3} \circ q_{\mathcal{S}_1|\mathcal{S}_2} = q_{\mathcal{S}_1|\mathcal{S}_3}
  \quad \text{and} \quad
  \phi_{\mathcal{S}_2 | \mathcal{S}_3} \circ \phi_{\mathcal{S}_1 | \mathcal{S}_2} = \phi_{\mathcal{S}_1 | \mathcal{S}_3}.
\end{equation}
Lemma~\ref{lem:canonical-epi} says that the inverse system $(\{ A_{\mathcal{S}} \}, \{ \phi_{\mathcal{S}_1|\mathcal{S}_2} \} )$ in $\mathcal{C}$ is obtained from the inverse system $( \{ \uptau_{\mathcal{S}} \}, \{ q_{\mathcal{S}_1|\mathcal{S}_2} \})$ in $\REX(\mathcal{M})$ by applying $\Act_{\mathcal{M}}^{\radj}$. We note that an exact functor preserves epimorphisms. By Theorem~\ref{thm:action-adj-by-end} and Lemma~\ref{lem:canonical-epi}, we have:

\begin{lemma}
  \label{lem:canonical-epi-2}
  If $\mathcal{M}$ is an exact $\mathcal{C}$-module category, then $(\{ A_{\mathcal{S}} \}, \{ \phi_{\mathcal{S}_1|\mathcal{S}_2} \})$ is an inverse system of epimorphisms in $\mathcal{C}$.
\end{lemma}

\newcommand{\Quo}{\mathfrak{Quo}}
\newcommand{\Sub}{\mathfrak{Sub}}

We use the above observation to state the main result of this section.
For an object $X$ of an essentially small category $\mathcal{E}$, we denote by $\Quo(X)$ and $\Sub(X)$ the set of quotient objects of $X$ and the set of subobjects of $X$, respectively. We introduce partial orders on these sets as follows: For $Q_1, Q_2 \in \Quo(X)$, we write $Q_1 \ge Q_2$ if there is a morphism $Q_1 \to Q_2$ in $\mathcal{E}$ compatible with the quotient morphisms from $X$. Dually, for $S_1, S_2 \in \Sub(X)$, we write $S_1 \ge S_2$ if there is a morphism $S_2 \to S_1$ in $\mathcal{E}$ compatible with the inclusion morphisms to $X$.

\begin{theorem}
  \label{thm:adjoint-alg-quotients}
  Let $\mathcal{M}$ be an exact $\mathcal{C}$-module category. Then the map
  \begin{equation*}
    \Top(\mathcal{M}) \to \Quo(A_{\mathcal{M}}),
    \quad \mathcal{S} \mapsto A_{\mathcal{S}} = \int_{X \in \mathcal{S}} \iHom(X, X)
  \end{equation*}
  preserves the order. If, moreover, $\mathcal{M}$ is indecomposable, then this map reflects the order.
\end{theorem}
\begin{proof}
  Lemma~\ref{lem:canonical-epi-2} means that the map in question preserves the order. To complete the proof, we suppose that $\mathcal{M}$ is indecomposable. Let $\mathcal{S}_1$ and $\mathcal{S}_2$ be topologizing full subcategory of $\mathcal{M}$ such that $A_{\mathcal{S}_1} \ge A_{\mathcal{S}_2}$ in $\Quo(A_{\mathcal{M}})$. Then we have $\Act_{\mathcal{M}}^{\radj}(\upkappa_{\mathcal{S}_1}) \le \Act_{\mathcal{M}}^{\radj}(\upkappa_{\mathcal{S}_2})$ in $\Sub(A_{\mathcal{M}})$. Since $\Act^{\radj}_{\mathcal{M}}$ is exact, we have
  \begin{equation*}
    \Act^{\radj}_{\mathcal{M}}\left( \frac{\upkappa_{\mathcal{S}_2}}
      {\upkappa_{\mathcal{S}_2} \cap \upkappa_{\mathcal{S}_1}} \right)
    = \frac{\Act^{\radj}_{\mathcal{M}}(\upkappa_{\mathcal{S}_2})}
    {\Act_{\mathcal{M}}^{\radj}(\upkappa_{\mathcal{S}_2}) \cap \Act_{\mathcal{M}}^{\radj}(\upkappa_{\mathcal{S}_1})}
    = \frac{\Act^{\radj}_{\mathcal{M}}(\upkappa_{\mathcal{S}_2})}
    {\Act_{\mathcal{M}}^{\radj}(\upkappa_{\mathcal{S}_2})}
    = 0.
  \end{equation*}
  Since $\Act_{\mathcal{M}}^{\radj}$ is faithful by Theorem~\ref{thm:action-adj-by-end}, we have $\upkappa_{\mathcal{S}_2} / (\upkappa_{\mathcal{S}_2} \cap \upkappa_{\mathcal{S}_1}) = 0$. This implies that $\upkappa_{\mathcal{S}_1} \subset \upkappa_{\mathcal{S}_2}$. Hence $\mathcal{S}_1 \subset \mathcal{S}_2$. The proof is done.
\end{proof}

The dual of Theorem~\ref{thm:adjoint-alg-quotients} is also interesting. Let $\ActLex_{\mathcal{M}}: \mathcal{C} \to \LEX(\mathcal{M})$ be the left exact version of the action functor. By Remark~\ref{rem:lex-version} and the dual of Lemma~\ref{lem:end-adj} (see \cite[Lemma 3.9]{MR2869176}), the coend of the functor
\begin{equation}
  \label{eq:adjoint-coalg-of-S-1}
  \mathcal{S}^{\op} \times \mathcal{S} \to \mathcal{C},
  \quad (X, X') \mapsto \icoHom(X, X')
\end{equation}
exists for all $\mathcal{S} \in \Top(\mathcal{M})$ and is canonically isomorphic to the coend
\begin{equation*}
  \ActLex_{\mathcal{M}}^{\ladj}(\uptau^{\radj}_{\mathcal{S}}) = \int^{M \in \mathcal{M}} \icoHom(M, \uptau^{\radj}_{\mathcal{S}}(M)).
\end{equation*}
We denote the coend of \eqref{eq:adjoint-coalg-of-S-1} by $L_{\mathcal{S}} = \int^{X \in \mathcal{S}} \icoHom(X, X)$. Since the duality functor is an anti-equivalence, the object ${}^* \! A_{\mathcal{S}}$ is also a coend of the functor \eqref{eq:adjoint-coalg-of-S-1} with universal dinatural transformation
\begin{equation*}
  {}^* \pi_{\mathcal{S}}: {}^* \! A_{\mathcal{S}} \to {}^* \iHom(X, X) = \icoHom(X, X)
  \quad (X \in \mathcal{S}).
\end{equation*}
Thus there is an isomorphism ${}^* \! A_{\mathcal{S}} \cong L_{\mathcal{S}}$ respecting the universal dinatural transformations. By the above observation, we now obtain the following theorem:

\begin{theorem}
  \label{thm:adjoint-alg-quotients-dual}
  Let $\mathcal{M}$ be an exact $\mathcal{C}$-module category. Then the map
  \begin{equation*}
    \Top(\mathcal{M}) \to \Sub(L_{\mathcal{M}}),
    \quad \mathcal{S} \mapsto L_{\mathcal{S}}
  \end{equation*}
  preserves the order. If, moreover, $\mathcal{M}$ is indecomposable, then this map reflects the order.
\end{theorem}

\subsection{Integral over a module full subcategory}
\label{subsec:integral-over-mod-subcat}

Let $\mathcal{C}$ be a finite tensor category, and let $\mathcal{M}$ be a finite left $\mathcal{C}$-module category. We introduce the following terminology:

\begin{definition}
  A {\em $\mathcal{C}$-module full subcategory} of $\mathcal{M}$ is a topologizing full subcategory of $\mathcal{M}$ closed under the action of $\mathcal{C}$.
\end{definition}

Let $\mathcal{S}$ be a $\mathcal{C}$-module full subcategory of $\mathcal{M}$, and let $\upkappa_{\mathcal{S}}$ and $\uptau_{\mathcal{S}}$ be the endofunctors on $\mathcal{M}$ defined by~\eqref{eq:reflector-def} and~\eqref{eq:idempo-monad-def}, respectively. Then we have
\begin{equation*}
  (V \otimes M)(V \otimes \upkappa_{\mathcal{S}}(M))
  \cong V \otimes (M / \upkappa_{\mathcal{S}}(M))
  = V \otimes \uptau_{\mathcal{S}}(M) \in \mathcal{S}
\end{equation*}
for all $V \in \mathcal{C}$ and $M \in \mathcal{M}$. Thus we have a natural transformation
\begin{equation*}
  \uptau_{\mathcal{S}}(V \otimes M) \to V \otimes \uptau_{\mathcal{S}}(M)
  \quad (V \in \mathcal{C}, M \in \mathcal{M})
\end{equation*}
making $\uptau_{\mathcal{S}} \in \REX(\mathcal{M})$ an oplax $\mathcal{C}$-module endofunctor on $\mathcal{M}$. Since $\mathcal{C}$ is rigid, we may regard $\uptau_{\mathcal{S}}$ as a strong $\mathcal{C}$-module functor.

By Theorem~\ref{thm:act-fun-adj-Z}, we endow the algebra $A'_{\mathcal{S}} = \Act^{\radj}(\uptau_{\mathcal{S}})$ with a half-braiding $\sigma'_{\mathcal{S}}$ such that $(A'_{\mathcal{S}}, \sigma'_{\mathcal{S}})$ is an algebra in $\mathcal{Z}(\mathcal{C})$. Since $A_{\mathcal{S}}$ is isomorphic to $A'_{\mathcal{S}}$, the algebra $A_{\mathcal{S}}$ also give rise to an algebra in $\mathcal{Z}(\mathcal{C})$. By Lemma~\ref{lem:act-fun-adj-half-braiding}, the half-braiding
\begin{equation*}
  \sigma_{\mathcal{S}}(V): A_{\mathcal{S}} \otimes V \to V \otimes A_{\mathcal{S}} \quad (V \in \mathcal{C})
\end{equation*}
of $A_{\mathcal{S}}$ inherited from $\Act^{\radj}(\uptau_{\mathcal{S}})$ is the unique morphism such that the diagram
\begin{equation}
  \label{eq:act-fun-adj-half-bra}
  \xymatrix@C=32pt@R=16pt{
    A_{\mathcal{S}} \otimes V
    \ar[rr]^-{\pi_{\mathcal{S}}(V \otimes X) \otimes \id_V}
    \ar[d]_{\sigma_{\mathcal{S}}(V)}
    & & \iHom(V \otimes X, V \otimes X) \otimes V
    \ar[d]^{\iHomB_{V, X, V\otimes X}} \\
    V \otimes A_{\mathcal{S}}
    \ar[r]^-{\id_V \otimes \pi_{\mathcal{S}}}
    & V \otimes \iHom(X, X)
    \ar[r]^{\iHomA_{V,X,X}}
    & \iHom(V \otimes X, X)
  }
\end{equation}
commutes for all objects $X \in \mathcal{S}$. We write $\mathbf{A}_{\mathcal{S}} := (A_{\mathcal{S}}, \sigma_{\mathcal{S}}) \in \mathcal{Z}(\mathcal{C})$. The following result is well-known in the case where $\mathcal{M} = \mathcal{S} = \mathcal{C}$.

\begin{theorem}
  \label{thm:comm-alg-from-mod-subcat}
  The algebra $\mathbf{A}_{\mathcal{S}} \in \mathcal{Z}(\mathcal{C})$ is commutative.
\end{theorem}
\begin{proof}
  We postpone the proof of this theorem to Appendix \ref{apdx:act-adj-structure} since it requires some technical results on the natural isomorphisms $\iHomA$ and $\iHomB$.
\end{proof}

Let $\mathcal{S}_1$ and $\mathcal{S}_2$ be $\mathcal{C}$-module full subcategories of $\mathcal{M}$ such that $\mathcal{S}_1 \supset \mathcal{S}_2$. We have introduced the morphism $\phi_{\mathcal{S}_1 | \mathcal{S}_2}: A_{\mathcal{S}_1} \to A_{\mathcal{S}_2}$ in $\mathcal{C}$ in the previous subsection. By the definition of $\phi_{\mathcal{S}_1|\mathcal{S}_2}$ and the above explicit description of the half-braiding, we have the following result:

\begin{theorem}
  $\phi_{\mathcal{S}_1|\mathcal{S}_2}: \mathbf{A}_{\mathcal{S}_1} \to \mathbf{A}_{\mathcal{S}_2}$ is a morphism in $\mathcal{Z}(\mathcal{C})$.
\end{theorem}

Thus, if $\mathcal{M}$ is exact, then we have an order-preserving map
\begin{equation*}
  \{ \text{$\mathcal{C}$-module full subcategories of $\mathcal{M}$} \}
  \to \{ \text{quotient algebras of $\mathbf{A}_{\mathcal{M}}$ in $\mathcal{Z}(\mathcal{C})$} \}
\end{equation*}
defined by $\mathcal{S} \mapsto \mathbf{A}_{\mathcal{S}}$. If, moreover, $\mathcal{M}$ is indecomposable, then this map reflects the order.

\section{Class functions and characters}
\label{sec:class-function}

\subsection{The space of class functions}

Let $\mathcal{C}$ be a finite tensor category. By the result of the last section, we have the algebra $A_{\mathcal{M}} \in \mathcal{C}$ for each finite left $\mathcal{C}$-module category $\mathcal{M}$. The vector space $\Hom_{\mathcal{C}}(A_{\mathcal{M}}, \unitobj)$ with $\mathcal{M} = \mathcal{C}$ is called the space of class functions in \cite{MR3631720} as it generalizes the usual notion of class functions on a finite group. The aim of this section is to explore the structure of the space of class functions and its generalization to module categories. We first introduce the following notation:

\begin{definition}
  For a finite left $\mathcal{C}$-module category $\mathcal{M}$, we define the space of class functions of $\mathcal{M}$ by $\CF(\mathcal{M}) = \Hom_{\mathcal{C}}(A_{\mathcal{M}}, \unitobj)$.
\end{definition}

Let $\mathsf{U}: \mathcal{Z}(\mathcal{C}) \to \mathcal{C}$ be the forgetful functor. To study class functions, we consider the functor $\mathsf{Z} := \mathsf{U}^{\radj} \circ \mathsf{U}$. There is an equivalence $\Act': \mathcal{C} \to \mathcal{C}_{\mathcal{C}}^*$ given by $\Act'(V) = \id_{\mathcal{C}} \otimes V$. By applying Theorem~\ref{thm:induction-Drinfeld-center} to $\mathcal{M} = \mathcal{C}$, we have
\begin{equation*}
  \mathsf{Z}(V) = \Act_{\mathcal{C}}^{\radj} \Act'(V) = \int_{X \in \mathcal{C}} X \otimes V \otimes X^*
  \quad (V \in \mathcal{C}).
\end{equation*}
Now let $\pi^{\mathsf{Z}}_V(X): \mathsf{Z}(V) \to X \otimes V \otimes X^*$ ($V, X \in \mathcal{C}$) be the universal dinatural transformation for the end $\mathsf{Z}(V)$. The assignment $V \mapsto \mathsf{Z}(V)$ extends to an endofunctor on $\mathcal{C}$ in such a way that $\pi^{\mathsf{Z}}_V(X)$ is natural in $V$ and dinatural in $X$. By the universal property, we define natural transformations $\Delta^{\mathsf{Z}}: \mathsf{Z} \to \mathsf{Z}^2$ and $\varepsilon^{\mathsf{Z}}: \mathsf{Z} \to \id_{\mathcal{C}}$ by
\begin{gather*}
  (\id_{X} \otimes \pi^{\mathsf{Z}}_{V}(Y) \otimes \id_{X^*}) \circ \pi^{\mathsf{Z}}_{\mathsf{Z}(V)}(X) \circ \Delta^{\mathsf{Z}}_V
  = \pi^{\mathsf{Z}}_V(X \otimes Y)
  \quad \text{and} \quad
  \varepsilon^{\mathsf{Z}}_{V} = \pi^{\mathsf{Z}}_{X}(\unitobj)
\end{gather*}
for all objects $V, X, Y \in \mathcal{C}$. The functor $\mathsf{Z}$ is a comonad on $\mathcal{C}$ with comultiplication $\Delta^{\mathsf{Z}}$ and counit $\varepsilon^{\mathsf{Z}}$.

Given an object $\mathbf{V} = (V, \sigma) \in \mathcal{Z}(\mathcal{C})$, we define the morphism $\delta: V \to \mathsf{Z}(V)$ in $\mathcal{C}$ to be the unique morphism such that the equation
\begin{equation*} 
  \pi^{\mathsf{Z}}_V(X) \circ \delta = (\sigma_X \otimes \id_{X^*}) \circ (\id_V \otimes \coev_X)
\end{equation*}
holds for all objects $X \in \mathcal{C}$. The assignment $(V, \sigma) \mapsto (V, \delta)$ allows us to identify $\mathcal{Z}(\mathcal{C})$ with the category of $\mathsf{Z}$-comodules. If we identify them, then a right adjoint of $\mathsf{U}$ is given by the free $\mathsf{Z}$-comodule functor
\begin{equation*}
  \mathsf{U}^{\radj}: \mathcal{C} \to (\text{the category of $\mathsf{Z}$-comodules}),
  \quad V \mapsto (\mathsf{Z}(V), \Delta^{\mathsf{Z}}_V).
\end{equation*}
By Theorem~\ref{thm:comm-alg-from-mod-subcat}, for each finite left $\mathcal{C}$-module category $\mathcal{M}$, there is a commutative algebra $\mathbf{A}_{\mathcal{M}} = (A_{\mathcal{M}}, \sigma_{\mathcal{M}})$ in $\mathcal{Z}(\mathcal{C})$ such that $A_{\mathcal{M}} = \mathsf{U}(\mathbf{A}_{\mathcal{M}})$.

\begin{definition}
  Let $\mathcal{M}$ be as above, and let $\delta_{\mathcal{M}}: A_{\mathcal{M}} \to \mathsf{Z}(A_{\mathcal{M}})$ be the coaction of $\mathsf{Z}$ associated to the half-braiding $\sigma_{\mathcal{M}}$. For $f \in \CF(\mathcal{C})$ and $g \in \CF(\mathcal{M})$, we define their product $f \star g \in \CF(\mathcal{M})$ by
  \begin{equation}
    \label{eq:class-ft-action}
    f \star g = f \circ \mathsf{Z}(g) \circ \delta_{\mathcal{M}}.
  \end{equation}
\end{definition}

In particular, we have a binary operation on $\CF(\mathcal{C})$ by considering the case where $\mathcal{M} = \mathcal{C}$ in the above definition. As we have observed in \cite{MR3631720}, $\CF(\mathcal{C})$ is an associative unital algebra with respect $\star$. Moreover, we have:

\begin{lemma}
  \label{lem:CF-C-module-CF-M}
  $\CF(\mathcal{M})$ is a left $\CF(\mathcal{C})$-module by $\star$.
\end{lemma}
\begin{proof}
  We remark $\mathbf{A}_{\mathcal{C}} = \mathsf{U}^{\radj}(\unitobj)$. Thus there is an isomorphism
  \begin{equation*}
    \Phi_{\mathcal{M}}: \CF(\mathcal{M}) \to \Hom_{\mathcal{Z}(\mathcal{C})}(\mathbf{A}_{\mathcal{M}}, \mathbf{A}_{\mathcal{C}}),
    \quad f \mapsto \mathsf{Z}(f) \circ \delta_{\mathcal{M}}.
  \end{equation*}
  Then, noting $\delta_{\mathcal{C}} = \Delta^{\mathsf{Z}}_{\unitobj}$, we compute
  \begin{gather*}
    \Phi_{\mathcal{C}}(f) \circ \Phi_{\mathcal{M}}(g)
    = \mathsf{Z}(f) \circ \Delta^{\mathsf{Z}}_{\unitobj} \circ \mathsf{Z}(g) \circ \delta_{\mathcal{M}}
    = \mathsf{Z}(f) \circ \mathsf{Z}^2(g) \circ \Delta^{\mathsf{Z}}_{A_{\mathcal{M}}} \circ \delta_{\mathcal{M}} \\
    = \mathsf{Z}(f) \circ \mathsf{Z}^2(g) \circ \mathsf{Z}(\delta_{\mathcal{M}}) \circ \delta_{\mathcal{M}}
    = \mathsf{Z}(f \star g) \circ \delta_{\mathcal{M}} = \Phi_{\mathcal{M}}(f \star g)
  \end{gather*}
  for all elements $f \in \CF(\mathcal{C})$ and $g \in \CF(\mathcal{M})$. Since the composition of morphisms is unital and associative, the action $\star: \CF(\mathcal{C}) \times \CF(\mathcal{M}) \to \CF(\mathcal{M})$ is also unital and associative. The proof is done.
\end{proof}

We set $F = \CF(\mathcal{C})$ and $E = \End_{\mathcal{Z}}(\mathbf{A}_{\mathcal{C}})$ for simplicity. The proof of the above lemma implies the following interesting consequence:

\begin{theorem}
  \label{thm:CF-C-module-CF-M}
  There is an isomorphism $E \cong F$ of algebras. Moreover, the left $F$-module $\CF(\mathcal{M})$ corresponds to the left $E$-module  $\Hom_{\mathcal{Z}(\mathcal{C})}(\mathbf{A}_{\mathcal{M}}, \mathbf{A}_{\mathcal{C}})$ under the isomorphism $E \cong F$.
\end{theorem}

\subsection{Pivotal module category}
\label{subsec:pivotal-mod-cat}

We recall that a {\em pivotal monoidal category} is a rigid monoidal category $\mathcal{C}$ equipped with a pivotal structure, that is, an isomorphism $X \to X^{**}$ ($X \in \mathcal{C}$) of monoidal functors. Let $\mathcal{C}$ be a pivotal finite tensor category with pivotal structure $j$. For an object $X \in \mathcal{C}$, we set
\begin{equation*}
  \trace_{\mathcal{C}}(X) = \eval_{X^*} \circ (j_X \otimes \id_{X^*})
\end{equation*}
and define the {\em internal character} \cite{MR3631720} of $X$ by 
\begin{equation*}
  \mathrm{ch}(X) = \trace_{\mathcal{C}}(X) \circ \pi_{\mathcal{C}}(X) \in \CF(\mathcal{C}).
\end{equation*}
Some applications of this notion are given in \cite{MR3631720}. It is interesting to extend results of \cite{MR3631720} to module categories. We first introduce the notion of {\em pivotal module category}. To give its precise definition, we recall the following notion:

\begin{definition}[{\cite{2016arXiv161204561F}}]
  For an exact $\mathcal{C}$-module category $\mathcal{M}$, there is a unique functor $\Ser_{\mathcal{M}}: \mathcal{M} \to \mathcal{M}$ equipped with a natural isomorphism
  \begin{equation}
    \label{eq:def-relative-Serre}
    \iHom(M, N)^* \cong \iHom(N, \Ser_{\mathcal{M}}(M))
  \end{equation}
  for $M, N \in \mathcal{M}$. We call $\Ser_{\mathcal{M}}$ the {\em relative Serre functor} of $\mathcal{M}$.
\end{definition}

Let $\mathcal{M}$ be an exact left $\mathcal{C}$-module category. We make $\mathcal{M}^{\op} \times \mathcal{M}$ a $\mathcal{C}$-bimodule category by~\eqref{eq:Mop-M-actions}. Then $\iHom$ is a $\mathcal{C}$-bimodule functor. Given a strong monoidal functor $T: \mathcal{C} \to \mathcal{C}$ and a left $\mathcal{C}$-module category $\mathcal{N}$, we denote by ${}_T \mathcal{N}$ the left $\mathcal{C}$-module category whose underlying category is $\mathcal{N}$ but the action of $\mathcal{C}$ on $\mathcal{N}$ is twisted by $T$. The functor
\begin{equation*}
  \mathcal{M}^{\op} \times \mathcal{M} \to {}_{(-)^{**}} \mathcal{C},
  \quad (M, N) \mapsto \iHom(N, M)^*
\end{equation*}
is a $\mathcal{C}$-bimodule functor by \eqref{eq:iHom-iso-a} and \eqref{eq:iHom-iso-b}. By \cite[Lemma 4.22]{2016arXiv161204561F}, there is a unique natural isomorphism
\begin{equation}
  \label{eq:relative-Serre-module-functor}
  X^{**} \otimes \Ser_{\mathcal{M}}(M) \to \Ser_{\mathcal{M}}(X \otimes M) \quad (X \in \mathcal{C}, M \in \mathcal{M})
\end{equation}
making $\Ser_{\mathcal{M}}$ a left $\mathcal{C}$-module functor $\Ser_{\mathcal{M}}: \mathcal{M} \to {}_{(-)^{**}} \mathcal{M}$ such that \eqref{eq:def-relative-Serre} is an isomorphism of $\mathcal{C}$-bimodule functors.

\begin{definition}
  \label{def:pivotal-mod-cat}
  Let $\mathcal{C}$ be a pivotal finite tensor category with pivotal structure $j$, and let $\mathcal{M}$ be an exact $\mathcal{C}$-module category. A {\em pivotal structure} of $\mathcal{M}$ is a natural isomorphism $j': \id_{\mathcal{M}} \to \Ser_{\mathcal{M}}$ such that the equation
  \begin{equation}
    \label{eq:mod-cat-piv-str-def}
    j_{X \otimes M}
    = \Big( X \otimes M
    \xrightarrow{\quad j_X \otimes j_M \quad}
    X^{**} \otimes \Ser_{\mathcal{M}}(M)
    \xrightarrow[\cong]{\quad \text{\eqref{eq:relative-Serre-module-functor}} \quad}
    \Ser_{\mathcal{M}}(X \otimes M) \Big)
  \end{equation}
  holds for all $X \in \mathcal{C}$ and $M \in \mathcal{M}$. A {\em pivotal left $\mathcal{C}$-module category} is an exact left $\mathcal{C}$-module category equipped with a pivotal structure. Let $\mathcal{M}$ be such a category, and let $j'$ be the pivotal structure of $\mathcal{M}$. Then we define the {\em trace}
  \begin{equation*}
    \trace_{\mathcal{M}}(M): \iHom(M, M) \to \unitobj \quad (M \in \mathcal{M})
  \end{equation*}
  to be the morphism corresponding to $j'_M: M \to \Ser_{\mathcal{M}}(M)$ via
  \begin{gather*}
    \Hom_{\mathcal{M}}(M, \Ser_{\mathcal{M}}(M))
    \cong \Hom_{\mathcal{M}}(\unitobj, \iHom(M, \Ser_{\mathcal{M}}(M)) \\
    \cong \Hom_{\mathcal{M}}(\unitobj, \iHom(M, M)^*)
    \cong \Hom_{\mathcal{M}}(\iHom(M, M), \unitobj).
  \end{gather*}
\end{definition}

\begin{remark}
  Let $\mathcal{C}$ and $\mathcal{M}$ be as above. Then, for a morphism $f: M \to M$ in $\mathcal{M}$, the {\em pivotal trace} $\mathrm{ptr}(f) \in k$ is defined by
  \begin{equation}
    \label{eq:def-ptrace}
    \mathrm{ptr}(f) \cdot \id_{\unitobj}
    = \trace_{\mathcal{M}}(M) \circ \iHom(\id_{M}, f) \circ \icoev_{\unitobj, M}.
  \end{equation}
  As in the ordinary trace, the pivotal trace is cyclic, multiplicative with respect to $\otimes$ and additive with respect to exact sequences; see Propositions~\ref{prop:apdx-piv-trace} and \ref{apdx:prop-additivity-piv-trace} in Appendix~\ref{sec:additivity-trace}.
\end{remark}

\subsection{Internal characters for module categories}

Let $\mathcal{C}$ be a pivotal finite tensor category, and let $\mathcal{M}$ be an pivotal exact left $\mathcal{C}$-module category. We now define:

\begin{definition}
  \label{def:internal-character}
  The {\em internal character} of $M \in \mathcal{M}$ is defined by
  \begin{equation}
    \label{eq:def-internal-char}
    \ich_{\mathcal{M}}(M) = \trace_{\mathcal{M}}(M) \circ \pi_{\mathcal{M}}(M) \in \CF(\mathcal{M}).
  \end{equation}
\end{definition}

We give basic properties of internal characters:

\begin{lemma}
  \label{lem:character-multiplicative}
  For all $X \in \mathcal{C}$ and $M \in \mathcal{M}$, we have
  \begin{equation*}
    \ich_{\mathcal{C}}(X) \star \ich_{\mathcal{M}}(M) = \ich_{\mathcal{M}}(X \otimes M).
  \end{equation*}
\end{lemma}
\begin{proof}
  Straightforward. See Appendix \ref{sec:additivity-trace} for the detail.
\end{proof}

\begin{lemma}
  \label{lem:character-additive}
  The internal character is additive in exact sequences: For any exact sequence $0 \to M_1 \to M_2 \to M_3 \to 0$ in $\mathcal{M}$, we have
  \begin{equation*}
    \ich_{\mathcal{M}}(M_2) = \ich_{\mathcal{M}}(M_1) + \ich_{\mathcal{M}}(M_2).
  \end{equation*}
\end{lemma}
\begin{proof}
  It is well-known that the pivotal trace is additive in exact sequences. One can find a detailed proof of this fact in \cite[Lemma 2.5.1]{MR2803849}. The proof of this lemma goes along the same line; see Appendix \ref{sec:additivity-trace} for the detail.
\end{proof}

For a finite abelian category $\mathcal{A}$, we denote by $\Gr(\mathcal{A})$ the Grothendieck group of $\mathcal{A}$, that is, the quotient of the additive group generated by the isomorphism classes of objects of $\mathcal{A}$ by the relation $[M_2] = [M_1] + [M_3]$ for all exact sequences $0 \to M_1 \to M_2 \to M_3 \to 0$ in $\mathcal{A}$. We set $\Gr_k(\mathcal{A}) = k \otimes_{\mathbb{Z}} \Gr(\mathcal{A})$.

Now let $\{ L_1, \dotsc, L_n \}$ be a complete set of representatives of isomorphism classes of simple objects of $\mathcal{M}$. As a generalization of the main result of \cite{MR3631720}, we prove the following theorem:

\begin{theorem}
  \label{thm:irr-char-lin-indep}
  The set $\{ \ich(L_i) \}_{i = 1}^n \subset \CF(\mathcal{M})$ is linearly independent.
\end{theorem}
\begin{proof}
  The proof goes along the same way as \cite{MR3631720}. Let $\mathcal{S}$ be the full subcategory of $\mathcal{M}$ consisting of all semisimple objects of $\mathcal{M}$. Then, since $\mathcal{S}$ is semisimple, we may assume $A_{\mathcal{S}} = \bigoplus_{i = 1}^m \iHom(L_i, L_i)$ and $\pi_{\mathcal{S}}(L_i)$ is the projection to $\iHom(L_i, L_i)$ for $i = 1, \dotsc, n$. Let $\phi_{\mathcal{M}|\mathcal{S}}: A_{\mathcal{M}} \to A_{\mathcal{S}}$ be the morphism defined by~\eqref{eq:canonical-epi}. Since $\phi_{\mathcal{M}|\mathcal{S}}$ is an epimorphism, the map
  \begin{equation*}
    \bigoplus_{i = 1}^n \Hom_{\mathcal{C}}(\iHom(L_i, L_i), \unitobj)
    = \CF(\mathcal{S})
    \xrightarrow{\quad \Hom_{\mathcal{C}}(\phi_{\mathcal{M}|\mathcal{S}}, \unitobj) \quad}
    \CF(\mathcal{M})
  \end{equation*}
  is injective. Since the morphism $\ich_{\mathcal{M}}(L_i)$ is the image of the morphism $\trace_{\mathcal{M}}(L_i)$ under this map, the set $\{ \ich(L_i) \}_{i = 1}^n$ is linearly independent in $\CF(\mathcal{M})$.
\end{proof}

Let $\mathcal{C}$ and $\mathcal{M}$ be as above. By Lemma \ref{lem:character-additive} and Theorem~\ref{thm:irr-char-lin-indep}, the linear map
\begin{equation*}
  \ich_{\mathcal{M}}: \Gr_k(\mathcal{M}) \to \CF(\mathcal{M}),
  \quad [M] \mapsto \ich_{\mathcal{M}}(M)
  \quad (M \in \mathcal{M})
\end{equation*}
is well-defined and injective. Lemma~\ref{lem:character-multiplicative} implies that $\ich_{\mathcal{C}}: \Gr_k(\mathcal{C}) \to \CF(\mathcal{C})$ is an algebra map and $\ich_{\mathcal{M}}: \Gr_k(\mathcal{M}) \to \CF(\mathcal{M})$ is $\Gr_k(\mathcal{C})$-linear if we view $\CF(\mathcal{M})$ as a left $\Gr_k(\mathcal{C})$-module through the algebra map $\ich_{\mathcal{C}}$.

By the proof of the above lemma, we see that the linear map $\ich_{\mathcal{M}}$ is bijective if $\mathcal{M}$ is semisimple. We have proved that, under the assumption that $\mathcal{C}$ is unimodular in the sense of \cite{MR2119143}, the map $\ich_{\mathcal{C}}: \Gr_k(\mathcal{C}) \to \CF(\mathcal{C})$ is bijective if and only if $\mathcal{C}$ is semisimple \cite{MR3631720}. It would be interesting to establish an analogous result for module categories. The unimodularity of module categories, introduced in \cite{2016arXiv161204561F}, may be useful to formulate such a result.

\subsection{Class functions of the dual tensor category}

Let $\mathcal{C}$ be a finite tensor category, and let $\mathcal{M}$ be an indecomposable exact left $\mathcal{C}$-module category. As an application of our results, we give the following description of the algebra of class functions of the dual tensor category:

\begin{theorem}
  \label{thm:CF-D-D}
  $\CF(\mathcal{C}_{\mathcal{M}}^*) \cong \End_{\mathcal{Z}(\mathcal{C})}(\mathbf{A}_{\mathcal{M}})$ as algebras.
\end{theorem}
\begin{proof}
  Set $\mathcal{D} = \mathcal{C}_{\mathcal{M}}^*$. Let $\mathsf{U}: \mathcal{Z}(\mathcal{D}) \to \mathcal{D}$ be the forgetful functor. By Theorems~\ref{thm:categorical-Schauenburg}, \ref{thm:induction-Drinfeld-center} and~\ref{thm:CF-C-module-CF-M}, we have isomorphisms
  \begin{equation*}
    \CF(\mathcal{D}) \cong \End_{\mathcal{Z}(\mathcal{D})}(\mathsf{U}^{\radj}_{\mathcal{D}}(\unitobj_{\mathcal{D}}))
    \cong \End_{\mathcal{Z}(\mathcal{C})}(\Sch_{\mathcal{M}}^{-1} \mathsf{U}_{\mathcal{D}}^{\radj}(\unitobj_{\mathcal{D}}))
    \cong \End_{\mathcal{Z}(\mathcal{C})}(\mathbf{A}_{\mathcal{M}})
  \end{equation*}
  of algebras. The proof is done.
\end{proof}

A semisimple finite tensor category is called a {\em fusion category} \cite{MR2183279}. Our results give some new results on fusion categories. For example:

\begin{corollary}
  Suppose that the base field $k$ is of characteristic zero. Let $\mathcal{C}$ be a fusion category, and let $\mathcal{M}$ be an indecomposable exact left $\mathcal{C}$-module category such that $\mathcal{C}_{\mathcal{M}}^*$ admits a pivotal structure. Then there is an isomorphism
  \begin{equation*}
    \Gr_k(\mathcal{C}_{\mathcal{M}}^*) \cong \End_{\mathcal{Z}(\mathcal{C})}(\mathbf{A}_{\mathcal{M}})
  \end{equation*}
  of algebras.
\end{corollary}
\begin{proof}
  $\mathcal{C}_{\mathcal{M}}^*$ is a pivotal fusion categories by the assumption \cite{MR2183279}. Thus the result follows from the results of the previous subsection.
\end{proof}

The following result generalizes \cite[Example 2.18]{2013arXiv1309.4822O}:

\begin{corollary}
  Under the same assumption on the above corollary, the following two assertions are equivalent:
  \begin{enumerate}
  \item The Grothendieck ring of $\mathcal{C}_{\mathcal{M}}^*$ is commutative.
  \item The object $\mathbf{A}_{\mathcal{M}} \in \mathcal{Z}(\mathcal{C})$ is multiplicity-free.
  \end{enumerate}
\end{corollary}
\begin{proof}
  Since $k$ is of characteristic zero, $\mathcal{Z}(\mathcal{C})$ is a fusion category \cite{MR2183279}. Moreover, the ring $\Gr(\mathcal{D})$ is commutative if and only if the $k$-algebra $\Gr_k(\mathcal{D})$ is. Now the claim follows from the above corollary.
\end{proof}

\section{A filtration on the space of class functions}
\label{sec:filtr-space-class}

\subsection{A filtration on the space of class functions}

Let $\mathcal{M}$ be a finite abelian category. For an object $M \in \mathcal{M}$, we denote by $\socle(M)$ the socle of $M$. Every object $M \in \mathcal{M}$ has a canonical filtration
\begin{equation*}
  0 = M_0 \subset M_1 \subset M_2 \subset M_3 \subset \dotsb \subset M
\end{equation*}
such that $M_{i+1}/M_{i} = \socle(M/M_{i})$. We denote $M_n$ by $\socle_n(M)$. Then the assignment $M \mapsto \socle_n(M)$ extends to a $k$-linear left exact endofunctor on $\mathcal{M}$, which we call the {\em $n$-th socle functor}. The number
\begin{equation*}
  \Lw(M) = \min \{ n = 0, 1, 2, \dotsc \mid \socle_n(M) = M \}
\end{equation*}
is called the {\em Loewy length} of $M$. We define $\mathcal{M}_n$ to be the full subcategory of $\mathcal{M}$ consisting of all objects $M$ with $\Lw(M) \le n$. Since $\mathcal{M}$ is finite, the number
\begin{equation*}
  \Lw(\mathcal{M})
  := \min \{ n = 0, 1, 2, \dotsc \mid \mathcal{M}_n = \mathcal{M} \}
  = \max \{ \Lw(M) \mid M \in \mathcal{M} \}
\end{equation*}
is finite. We call $\Lw(\mathcal{M})$ the {\em Loewy length} of $\mathcal{M}$ and the filtration
\begin{equation}
  \label{eq:socle-filt-cat}
  0 = \mathcal{M}_0 \subset \mathcal{M}_1 \subset \mathcal{M}_2 \subset \dotsb \subset \mathcal{M}_w = \mathcal{M}
  \quad (w = \Lw(\mathcal{M}))
\end{equation}
the {\em socle filtration} of $\mathcal{M}$.

It is easy to see that each $\mathcal{M}_n$ is a topologizing full subcategory of $\mathcal{M}$. Thus, if $\mathcal{C}$ is a finite tensor category and $\mathcal{M}$ is an exact left $\mathcal{C}$-module category with Loewy length $w$, then we have a series
\begin{equation}
  \label{eq:Loewy-filtration-algebra}
  A_{\mathcal{M}} = A_{\mathcal{M}_w}
  \twoheadrightarrow A_{\mathcal{M}_{w-1}}
  \twoheadrightarrow \dotsb
  \twoheadrightarrow A_{\mathcal{M}_2}
  \twoheadrightarrow A_{\mathcal{M}_1}
\end{equation}
of epimorphisms of of algebras in $\mathcal{C}$ by Theorem~\ref{thm:adjoint-alg-quotients}. Applying $\Hom_{\mathcal{C}}(-, \unitobj)$ to this series, we obtain the filtration of the space of class functions
\begin{equation}
  \label{eq:Loewy-filtration-CF}
  \CF_1(\mathcal{M})
  \subset \CF_2(\mathcal{M})
  \subset \dotsb
  \subset \CF_{w-1}(\mathcal{M})
  \subset \CF_{w}(\mathcal{M})
  = \CF(\mathcal{M}),
\end{equation}
where $\CF_n(\mathcal{M}) = \Hom_{\mathcal{C}}(A_{\mathcal{M}_n}, \unitobj)$. In this section, we investigate how this filtration relates to representation-theoretic properties of $\mathcal{M}$.

\subsection{Jacobson radical functor}

For further study of the series \eqref{eq:Loewy-filtration-algebra} and the filtration \eqref{eq:Loewy-filtration-CF}, we introduce the following abstract definition of the Jacobson radical: Let $\mathcal{M}$ be a finite abelian category. For an object $M \in \mathcal{M}$, we define the subobject $\rad(M)$ of $M$ to be the intersection of all maximal subobjects of $M$. It is easy to see that $M \mapsto \rad(M)$ extends to a $k$-linear right exact endofunctor on $\mathcal{M}$. We call $\rad \in \REX(\mathcal{M})$ the {\em Jacobson radical functor} of $\mathcal{M}$.

We rephrase several known results in the representation theory in terms of the Jacobson radical functor. Let $A$ be a finite-dimensional algebra such that $\mathcal{M} \approx \lmod{A}$, and let $J$ be the Jacobson radical of $A$. Then the Jacobson radical functor may be identified with $J \otimes_A(-)$. Thus we have the series
\begin{equation}
  \label{eq:Loewy-series}
  \id_{\mathcal{M}} =: \rad^0
  \supset \rad
  \supset \rad^2
  \supset \dotsb
  \supset \rad^{w-1}
  \supset \rad^{w} = 0
  \quad (w = \Lw(\mathcal{M}))
\end{equation}
of subobjects in $\REX(\mathcal{M})$. We have $\rad_{\mathcal{M}}^i \ne \rad_{\mathcal{M}}^{i+1}$ for all $i = 0, \dotsc, w - 1$ by the Nakayama lemma.

For a positive integer $n$, we define the {\em $n$-th capital functor} $\capital_n \in \REX(\mathcal{M})$ as the quotient object $\id_{\mathcal{M}}/\rad^n$. If we identify $\REX(\mathcal{M})$ with $\bimod{A}{A}$, then this functor corresponds to the bimodule $A/J^n$ and therefore
\begin{equation}
  \label{eq:n-capital}
  \capital_n(M) = (A/J^n) \otimes_A M \cong M/J^n M
\end{equation}
for all $M \in \mathcal{M}$. By Sakurai \cite[Lemma 2.3]{MR3656722}, there is an adjunction
\begin{equation}
  \label{eq:Sakurai-adjunctoin}
  \Hom_{\mathcal{M}}(\capital_n(M), M') \cong \Hom_{\mathcal{M}}(M, \socle_n(M'))
  \quad (M, M' \in \mathcal{M}).
\end{equation}
The $n$-th term $\mathcal{M}_n$ of the socle filtration \eqref{eq:socle-filt-cat} coincides with the full subcategory of $\mathcal{M}$ consisting of all objects $M$ such that $\socle_n(M) = M$. Comparing \eqref{eq:Sakurai-adjunctoin} with \eqref{eq:reflector-2}, we have $\capital_n = \uptau_{\mathcal{M}_n}$ with the notation in Subsection~\ref{subsec:topolo-full-sub}. In other words, $\mathcal{M}_n$ corresponds to $\rad^n$ via the correspondence of Lemma~\ref{lem:Rosen}.

Now we consider the case where $\mathcal{C}$ is a finite tensor category and $\mathcal{M}$ is an exact left $\mathcal{C}$-module category with Loewy length $w$. There is a series
\begin{equation}
  \label{eq:Loewy-filtration-capital}
  \id_{\mathcal{M}} = \capital_w
  \twoheadrightarrow \capital_{w-1}
  \twoheadrightarrow \dotsb
  \twoheadrightarrow \capital_2
  \twoheadrightarrow \capital_1
\end{equation}
of epimorphisms in $\REX(\mathcal{M})$. We have a canonical isomorphism
\begin{equation*}
  \ActRex^{\radj}(\capital_n) \cong \int_{X \in \mathcal{M}_n} \iHom(X, X),
\end{equation*}
and the series \eqref{eq:Loewy-filtration-algebra} is obtained by applying $\Act^{\radj}$ to \eqref{eq:Loewy-filtration-capital}.

\subsection{Reynolds ideal and its generalization}

Let $A$ be a finite-dimensional algebra. For $n \in \mathbb{Z}_{+}$, we define the {\em $n$-th Reynolds ideal} \cite{2017arXiv170103799S} of $A$ by
\begin{equation}
  \label{eq:n-Reynolds-def-1}
  \Rey_n(A) = \socle_n(A) \cap Z(A),
\end{equation}
where $\socle_n(A)$ is the $n$-th socle of the left $A$-module $A$. As $\Rey_n(A)$ is a Morita invariant \cite{2017arXiv170103799S}, it is natural to expect that the $n$-th Reynolds ideal of a finite abelian category is defined in an intrinsic way. For $n = 1$, this was achieved by Gainutdinov and Runkel in \cite{2017arXiv170300150G}. By using the Jacobson radical functor, we propose the following definition, which is different to \cite{2017arXiv170300150G}:

\begin{definition}
  \label{def:Reynolds-n}
  Let $\mathcal{M}$ be a finite abelian category. For a non-negative positive integer $n$, we define the $n$-th {\em Reynolds ideal} of $\mathcal{M}$ by
  \begin{equation*}
    \Rey_n(\mathcal{M}) = \{ \xi \in \End(\id_{\mathcal{M}}) \mid \xi \circ i_n = 0 \},
  \end{equation*}
  where $i_n: \rad^n \to \id_{\mathcal{M}}^{}$ is the inclusion morphism.
\end{definition}

Let $A$ be a finite-dimensional algebra. We explain that $\Rey_n(\mathcal{M})$ can be identified with the $n$-th Reynolds ideal of $A$ when $\mathcal{M} = \lmod{A}$. Let $J$ be the Jacobson radical of $A$. Then the $n$-th socle of $M \in \lmod{A}$ is given by
\begin{equation*}
  \socle_n(M) = \{ m \in M \mid \text{$r m = 0$ for all $r \in J^n$} \},
\end{equation*}
and hence the $n$-th Reynolds ideal of $A$ is expressed as follows:
\begin{equation}
  \label{eq:n-Reynolds-def-2}
  \Rey_n(A) = \{ z \in Z(A) \mid \text{$r z = 0$ for all $r \in J^n$} \}.
\end{equation}
If $\mathcal{M} = \lmod{A}$, then $\REX(\mathcal{M})$ can be identified with $\bimod{A}{A}$. Under this identification, $\id_{\mathcal{M}}$ and $\rad^n$ correspond to the $A$-bimodule $A$ and its subbimodule $J^n$, respectively. By \eqref{eq:n-Reynolds-def-2}, it is easy to check that the isomorphism $Z(A) \cong \End(\id_{\mathcal{M}})$ restricts to an isomorphism $\Rey_n(A) \cong \Rey_n(\lmod{A})$ for each $n$.

We consider the case where $\mathcal{C}$ is a finite tensor category and $\mathcal{M}$ is an indecomposable exact left $\mathcal{C}$-module category with action functor $\Act = \Act_{\mathcal{M}}$. Then we have an adjunction isomorphism
\begin{equation}
  \label{eq:Hom-1-A-and-center}
  \Hom_{\mathcal{C}}(\unitobj, A_{\mathcal{M}})
  = \Hom_{\mathcal{C}}(\unitobj, \Act^{\radj}(\id_{\mathcal{M}}))
  \cong \Nat(\Act(\unitobj), \id_{\mathcal{M}})
  = \End(\id_{\mathcal{C}}).
\end{equation}
Moreover, since $\Act^{\radj}$ is exact by Theorem~\ref{thm:action-adj-by-end}, the object $J_{\mathcal{M}}^n := \Act^{\radj}(\rad^n)$ is a subobject of $A_{\mathcal{M}}$. The following description of $\Rey_n(\mathcal{M})$ may be regarded as a generalization of \eqref{eq:n-Reynolds-def-2}.

\begin{lemma}
  \label{lem:Reynolds-1}
  For an indecomposable exact $\mathcal{C}$-module category $\mathcal{M}$, we define
  \begin{equation*}
    R_n(\mathcal{M}) = \{ a \in \Hom_{\mathcal{C}}(\unitobj, A_{\mathcal{M}}) \mid m \circ (a \otimes i) = 0 \},
  \end{equation*}
  where $m$ is the multiplication of $A_{\mathcal{M}}$ and $i: J_{\mathcal{M}}^n \to A_{\mathcal{M}}$ is the inclusion morphism. Then \eqref{eq:Hom-1-A-and-center} restricts to an isomorphism between $R_n(\mathcal{M})$ and $\Rey_n(\mathcal{M})$.
\end{lemma}
\begin{proof}
  We use the monoidal structure of $\Act^{\radj}$ described in Lemma~\ref{lem:act-fun-adj-monoidal}. Let $a: \unitobj \to A_{\mathcal{M}}$ be a morphism in $\mathcal{C}$, and let $\widetilde{a} \in \End(\id_{\mathcal{C}})$ be the natural transformation corresponding to $a$ via~\eqref{eq:Hom-1-A-and-center}. By the definition of $\mu^{(0)}$, we have $a = \Act^{\radj}(\widetilde{a}) \circ \mu^{(0)}$. Let $i_n: \rad^n \to \id_{\mathcal{M}}$ be the inclusion morphism. Since $i = \Act^{\radj}(i_n)$, we have
  \begin{align*}
    m \circ (a \otimes i)
    & = \mu^{(2)}_{\id_{\mathcal{M}}, \id_{\mathcal{M}}}
    \circ (\Act^{\radj}(\widetilde{a}) \otimes \Act^{\radj}(i_n))
    \circ (\mu^{(0)} \otimes \id_{J^n_{\mathcal{M}}}) \\
    & = \Act^{\radj}(\widetilde{a} \circ i_n)
      \circ \mu^{(2)}_{\id_{\mathcal{M}}, \id_{\mathcal{M}}}
      \circ (\mu^{(0)} \otimes \id_{J^n_{\mathcal{M}}})
      = \Act^{\radj}(\widetilde{a} \circ i_n).
  \end{align*}
  Thus, by the faithfulness of $\Act^{\radj}$ (Theorem~\ref{thm:action-adj-by-end}), the morphism $a$ belongs to $R_n(\mathcal{M})$ if and only if $\widetilde{a} \in \Rey_n(\mathcal{M})$. The proof is done.
\end{proof}

We recall that an algebra $A$ in $\mathcal{C}$ with multiplication $m$ is said to be {\em Frobenius} if there is an isomorphism $\phi: A \to A^*$ of right $A$-modules in $\mathcal{C}$. Given such an isomorphism $\phi$, we define
\begin{equation*}
  e_{\phi} = \eval_A \circ (\phi \otimes \id_A)
  \quad \text{and} \quad
  d_{\phi} = (\id_A \otimes \phi^{-1}) \circ \coev_A.
\end{equation*}
Then the triple $(A, e_{\phi}, d_{\phi})$ is a left dual object of $A$. Thus the map
\begin{equation}
  \label{eq:Frob-alg-Fourier-1}
  \Hom_{\mathcal{C}}(A, \unitobj) \to \Hom_{\mathcal{C}}(\unitobj, A)
  \quad \xi \mapsto (\xi \otimes \id_A) \circ d_{\phi}
\end{equation}
is an isomorphism of vector spaces with inverse
\begin{equation}
  \label{eq:Frob-alg-Fourier-2}
  \Hom_{\mathcal{C}}(\unitobj, A) \to \Hom_{\mathcal{C}}(A, \unitobj),
  \quad a \mapsto  e_{\phi} \circ (a \otimes \id_A).
\end{equation}
The $A$-linearity of $\phi$ imply
\begin{gather}
  \label{eq:Frob-basis-1}
  e_{\phi} \circ (m \otimes \id_A) = e_{\phi} \circ (\id_A \otimes m), \\
  \label{eq:Frob-basis-2}
  (\id_A \otimes m) \circ (d_{\phi} \otimes \id_A)
  = (m \otimes \id_A) \circ (\id_A \otimes d_{\phi}).
\end{gather}
The following lemma may be well-known:

\begin{lemma}
  \label{lem:Reynolds-2}
  Let $J$ be an ideal of $A$ with inclusion morphism $i: J \to A$. Then the isomorphisms \eqref{eq:Frob-alg-Fourier-1} and \eqref{eq:Frob-alg-Fourier-2} restricts to an isomorphism
  \begin{equation*}
    \Hom_{\mathcal{C}}(A/J, \unitobj) \cong \{ a \in \Hom_{\mathcal{C}}(\unitobj, A) \mid m \circ (a \otimes i) = 0 \}.
  \end{equation*}
\end{lemma}
\begin{proof}
  Let $\xi: A \to \unitobj$ be a morphism in $\mathcal{C}$, and let $a: \unitobj \to A$ be the morphism corresponding to $\xi$ by~\eqref{eq:Frob-alg-Fourier-1} and~\eqref{eq:Frob-alg-Fourier-2}. We first suppose that $\xi$ belongs to $\Hom_{\mathcal{C}}(A/J, \unitobj)$, that is, $\xi \circ i = 0$. Then we compute
  \begin{align*}
    m \circ (a \otimes i)
    & = (\xi \otimes \id_A) \circ (\id_A \otimes m) \circ (d_{\phi} \otimes \id_A) \circ i \\
    & = (\xi \otimes \id_A) \circ (m \otimes \id_A) \circ (\id_A \otimes d_{\phi}) \circ i \\
    & = ((\xi \circ m \circ (i \otimes \id_A)) \otimes \id_A) \circ (\id_J \otimes d_{\phi})
  \end{align*}
  by \eqref{eq:Frob-basis-2}. Since $J$ is an ideal of $A$, the image of $m \circ (i \otimes \id_A)$ is contained in $J$. Thus we have $\xi \circ m \circ (i \otimes \id_A) = 0$. Therefore $m \circ (a \otimes i) = 0$. If, conversely, this equation holds, then we have
  \begin{align*}
    \xi \circ i
    = e_{\phi} \circ (i \otimes a)
    & = e_{\phi} \circ (m \otimes \id_A) \circ (u \otimes i \otimes a) \\
    & = e_{\phi} \circ (\id_A \otimes m) \circ (u \otimes i \otimes a) = 0
  \end{align*}
  by \eqref{eq:Frob-basis-1}, where $u: \unitobj \to A$ is the unit of $A$. Thus $\xi \in \Hom_{\mathcal{C}}(A/J, \unitobj)$. The proof is done.
\end{proof}

Now we have the following representation-theoretic description of $\CF_n$.
  
\begin{theorem}
  Let $\mathcal{M}$ be an indecomposable exact $\mathcal{C}$-module category. If $A_{\mathcal{M}}$ is a Frobenius algebra, then the isomorphism
  \begin{equation}
    \label{eq:Fourier-trans}
    \CF(\mathcal{M})
    = \Hom_{\mathcal{C}}(A_{\mathcal{M}}, \unitobj)
    \xrightarrow[\cong]{\quad \eqref{eq:Frob-alg-Fourier-1} \quad}
    \Hom_{\mathcal{C}}(\unitobj, A_{\mathcal{M}})
    \xrightarrow[\cong]{\quad \eqref{eq:Hom-1-A-and-center} \quad} \End(\id_{\mathcal{M}})
  \end{equation}
  restricts to isomorphisms
  \begin{equation*}
    \CF_n(\mathcal{M}) \cong \Rey_n(\mathcal{M}) \quad (n = 1, 2, 3, \dotsc).
  \end{equation*}
\end{theorem}
\begin{proof}
  The subobject $\Act^{\radj}(\rad^n)$ is an ideal of $A_{\mathcal{M}} = \Act^{\radj}(\id_{\mathcal{M}})$. The proof is done by applying the above two lemmas to this ideal.
\end{proof}

A finite tensor category $\mathcal{D}$ is said to be {\em unimodular} \cite{MR2119143} if the projective cover of the unit object $\unitobj \in \mathcal{D}$ is also an injective hull of $\unitobj$. Following \cite{MR3632104}, a finite tensor category $\mathcal{D}$ is unimodular if and only if the algebra $\mathsf{R}(\unitobj) \in \mathcal{Z}(\mathcal{D})$ is Frobenius, where $\mathsf{R}: \mathcal{D} \to \mathcal{Z}(\mathcal{D})$ is a right adjoint of the forgetful functor.

Let $\mathcal{M}$ be an indecomposable exact left $\mathcal{C}$-module category. Then $\mathcal{D} := \mathcal{C}_{\mathcal{M}}^*$ is a finite tensor category. By Theorem~\ref{thm:induction-Drinfeld-center} and the above-mentioned fact, the algebra $\mathbf{A}_{\mathcal{M}} \in \mathcal{Z}(\mathcal{C})$ is Frobenius if and only if $\mathcal{D}$ is unimodular. Thus the algebra $A_{\mathcal{M}} \in \mathcal{C}$ is Frobenius if $\mathcal{D}$ is unimodular. By the above theorem, we have:

\begin{corollary}
  \label{cor:CF-n-and-Reynolds}
  Let $\mathcal{M}$ be an indecomposable exact $\mathcal{C}$-module category. If $\mathcal{C}_{\mathcal{M}}^*$ is unimodular, then we have $\CF_n(\mathcal{M}) \cong \Rey_n(\mathcal{M})$.
\end{corollary}

In particular, if $\mathcal{C}$ is unimodular, then $\CF_n(\mathcal{C}) \cong \Rey_n(\mathcal{C})$.

\subsection{Symmetric linear forms on an algebra}

For a finite-dimensional algebra $A$ with Jacobson radical $J$, we set
\begin{align*}
  \Sym(A) & = \{ f \in A^* \mid \text{$f(a b) = f(b a)$ for all $a, b \in A$} \}, \\
  \Sym_n(A) & = \{ f \in \Sym(A) \mid f(J^n) = 0 \}
              \quad (n \in \mathbb{Z}_{+}).
\end{align*}
If $G$ is a finite group, then $\Sym(k G)$ is the space of class functions on $G$. Thus, for a finite module category $\mathcal{M}$ such that $\mathcal{M} \approx \lmod{A}$, it is natural to ask how $\CF(\mathcal{M})$ relates to $\Sym(A)$. To consider this problem, we first introduce the following categorical definition of the space of symmetric linear forms:

\begin{definition}
  \label{def:SF-and-SFn}
  For a finite abelian category $\mathcal{M}$ and $n \in \mathbb{Z}_{+}$, we set
  \begin{equation*}
    \Sym(\mathcal{M}) := \Nat(\id_{\mathcal{M}}, \Nak_{\mathcal{M}})
    \quad \text{and} \quad
    \Sym_n(\mathcal{M}) = \{ f \in \Sym(\mathcal{M}) \mid f \circ i_n = 0 \},
  \end{equation*}
  where $i_n: \rad^n \to \id_{\mathcal{M}}$ is the inclusion morphism.
\end{definition}

If $\mathcal{M}$ is a finite abelian category such that $\mathcal{M} \approx \lmod{A}$, then $\REX(\mathcal{M})$ can be identified with $\bimod{A}{A}$. Under this identification, $\id_{\mathcal{M}}$ and $\Nak_{\mathcal{M}}$ correspond to the $A$-bimodules $A$ and $A^*$, respectively. Thus we have
\begin{equation}
  \label{eq:SLF-R-mod}
  \Sym(\mathcal{M}) \cong \Hom_{\bimod{A}{A}}(A, A^*)
  \cong \Sym(A),
\end{equation}
where the second isomorphism is given by $f \mapsto f(1)$. If we identify $\Sym(\mathcal{M})$ with $\Sym(A)$ by this isomorphism, then $\Sym_n(\mathcal{M})$ is identified with $\Sym_n(A)$.

\begin{remark}
  \label{rem:SF-n-and-Rey-n}
  Let $\mathcal{M}$ be a finite abelian category. We suppose that $\mathcal{M}$ is symmetric Frobenius and choose an isomorphism $\lambda: \id_{\mathcal{M}} \to \mathbb{N}_{\mathcal{M}}$. Then the map
  \begin{equation*}
    \End(\id_{\mathcal{M}}) \to \Sym(\mathcal{M}), \quad z \mapsto \lambda \circ z
  \end{equation*}
  is an isomorphism. By Definitions~\ref{def:Reynolds-n} and~\ref{def:SF-and-SFn}, we also have isomorphisms
  \begin{equation*}
    \Rey_n(\mathcal{M}) \to \Sym_n(\mathcal{M}), \quad z \mapsto \lambda \circ z
    \quad (n \in \mathbb{Z}_{+}).
  \end{equation*}
  In ring-theoretic terms, this means: Let $A$ be a symmetric Frobenius algebra, and let $\lambda: A \to A^*$ be an isomorphism of $A$-bimodules. For each $n \in \mathbb{Z}_{+}$, the isomorphism $\lambda$ restricts to an isomorphism between $\Rey_n(A)$ and $\Sym_n(A)$.
\end{remark}

Now we consider the case where $\mathcal{M}$ is an exact module category over a finite tensor category $\mathcal{C}$. Although $\CF(\mathcal{M})$ is an analogue of the space of class functions, it does not seem to be isomorphic to $\Sym(\mathcal{M})$ in general. To see when they are isomorphic, we provide the following lemma:

\begin{lemma}
  \label{lem:SLF-lemma}
  There is a natural isomorphism
  \begin{equation*}
    \Hom_{\mathcal{C}}(\Act^{\radj}(\Ser_{\mathcal{M}} \circ F), X^{**})
    \cong \Nat(F, X \otimes \Nak_{\mathcal{M}})
    \quad (F \in \REX(\mathcal{M}), V \in \mathcal{C}).
  \end{equation*}
\end{lemma}
\begin{proof}
  Let $D$ be the distinguished invertible object of $\mathcal{C}$ introduced in \cite{MR2097289}. Then there are natural isomorphisms
  \begin{equation*}
    \Nak_{\mathcal{C}}(X) \cong D^* \otimes X^{**}
    \quad \text{and} \quad
    \Nak_{\mathcal{M}}(M) \cong D^* \otimes \Ser_{\mathcal{M}}(M)
  \end{equation*}
  for $X \in \mathcal{C}$ and $M \in \mathcal{M}$ \cite{2016arXiv161204561F}. Since $\Ser_{\mathcal{M}}: \mathcal{M} \to {}_{(-)^{**}}\mathcal{M}$ is a $\mathcal{C}$-module functor, and since $D$ is an invertible object, we have natural isomorphisms
  \begin{equation*}
    (\Nak^{-1}_{\mathcal{M}} \circ \Ser_{\mathcal{M}})(M)
    \cong \Ser_{\mathcal{M}}^{-1}(D \otimes \Ser_{\mathcal{M}}(M))
    \cong D^{**} \otimes \Ser_{\mathcal{M}}^{-1} \Ser_{\mathcal{M}}(M)
    \cong D \otimes M
  \end{equation*}
  for $M \in \mathcal{M}$. By using these isomorphisms and basic results on the Nakayama functor recalled in Subsection~\ref{subsec:fin-ab-cat}, we have natural isomorphisms
  \begin{align*}
    & \Hom_{\mathcal{C}}(\Act^{\radj}(\Ser_{\mathcal{M}} \circ F), X^{**}) \\
    & \cong \Hom_{\mathcal{C}}(\Act^{\ladj}(\Nak^{-1}_{\mathcal{M}} \circ \Ser_{\mathcal{M}} \circ F \circ \Nak^{-1}_{\mathcal{M}}),
      \ \Nak^{-1}_{\mathcal{C}}(X^{**})) \\
    & \cong \Nat(\Nak^{-1}_{\mathcal{M}} \circ \Ser_{\mathcal{M}} \circ F \circ \Nak^{-1}_{\mathcal{M}},
      \ \Nak^{-1}_{\mathcal{C}}(X^{**}) \otimes \id_{\mathcal{M}}) \\
    & \cong \Nat(\Nak^{-1}_{\mathcal{M}} \circ \Ser_{\mathcal{M}} \circ F,
      \ \Nak_{\mathcal{C}}^{-1}(X^{**}) \otimes \Nak_{\mathcal{M}}) \\
    & \cong \Nat(D \otimes F,
      \ D \otimes X \otimes \Nak_{\mathcal{M}}) \cong \Nat(F, X \otimes \Nak_{\mathcal{M}})
  \end{align*}
  for $F \in \REX(\mathcal{M})$ and $X \in \mathcal{C}$. The proof is done.
\end{proof}

The following theorem is an immediate consequence of the above lemma.

\begin{theorem}
  \label{thm:CF-n-and-SLF-n}
  If $\mathcal{M}$ is an exact $\mathcal{C}$-module category whose relative Serre functor is isomorphic to the identity functor, then there is a natural isomorphism
  \begin{equation*}
    \Hom_{\mathcal{C}}(\Act^{\radj}(F), X^{**})
    \cong \Nat(F, X \otimes \Nak_{\mathcal{M}})
  \end{equation*}
  for $F \in \REX(\mathcal{M})$ and $X \in \mathcal{C}$. In particular, we have an isomorphism
  \begin{equation*}
    \CF(\mathcal{M}) \cong \Sym(\mathcal{M}),
  \end{equation*}
  which restricts to isomorphisms
  \begin{equation*}
    \CF_n(\mathcal{M}) \cong \Sym_n(\mathcal{M})
    \quad (n \in \mathbb{Z}_{+}).
  \end{equation*}
\end{theorem}

\subsection{Dimension of $\CF_1$}

For a finite abelian category $\mathcal{A}$, we denote by $\Irr(\mathcal{A})$ the set of isomorphism classes of simple objects of $\mathcal{M}$. Let $\mathcal{C}$ be a finite tensor category, and let $\mathcal{M}$ be an exact $\mathcal{C}$-module category. Then, by the proof of Theorem~\ref{thm:irr-char-lin-indep}, we have isomorphisms
\begin{equation}
  \label{eq:CF1-iso}
  \CF_1(\mathcal{M})
  \cong \bigoplus_{L \in \Irr(\mathcal{M})} \Hom_{\mathcal{C}}(\iHom(L, L), \unitobj)
  \cong \bigoplus_{L \in \Irr(\mathcal{M})} \Hom_{\mathcal{C}}(L, \Ser_{\mathcal{M}}(L)).
\end{equation}
Thus, by Schur's lemma, we have
\begin{equation*}
  \dim_k \CF_1(\mathcal{M}) = \# \{ L \in \Irr(\mathcal{C}) \mid \Ser_{\mathcal{M}}(L) \cong L \}.
\end{equation*}
We suppose, moreover, that $\mathcal{C}$ is a pivotal finite tensor category and $\mathcal{M}$ is a pivotal $\mathcal{C}$-module category with pivotal structure $j'$. Again by the proof of Theorem~\ref{thm:irr-char-lin-indep}, the internal character of $L \in \Irr(\mathcal{M})$ corresponds to $j'_L$ via~\eqref{eq:CF1-iso}. Thus the set $\{ \ich(L) \mid L \in \Irr(\mathcal{M}) \}$ of `irreducible characters' is a basis of $\CF_1(\mathcal{M})$.

\subsection{Dimension of $\CF_2$}

As we have seen in the above, the dimension of $\CF_1$ is expressed in representation-theoretic terms. It is interesting to give such an expression for the dimension of $\CF_n$ for $n \ge 2$. Here we give the following result:

\begin{theorem}
  \label{thm:dim-CF2}
  Let $\mathcal{C}$ be a finite tensor category. For an exact $\mathcal{C}$-module category $\mathcal{M}$ such that $\Ser_{\mathcal{M}} \cong \id_{\mathcal{M}}$, there is an isomorphism
  \begin{equation*}
    \CF_2(\mathcal{M}) = \CF_1(\mathcal{M}) \oplus \bigoplus_{L \in \Irr(\mathcal{M})} \Ext_{\mathcal{M}}^1(L, L).
  \end{equation*}
\end{theorem}

To prove Theorem~\ref{thm:dim-CF2}, we recall the following expression of $\Ext^1$: Let $A$ be a finite-dimensional algebra. Given $X \in \lmod{A}$, we denote by $g_X: A \to \End_k(X)$ the algebra map induced by the action of $A$ on $X$. For $V, W \in \lmod{A}$, the vector space $\Ext_{A}^1(V, W)$ is identified with the set of equivalence classes of short exact sequences of the form $0 \to W \to X \to V \to 0$ in $\lmod{A}$. If $X \in \lmod{A}$ fits into such an exact sequence, then we may assume that $X = V \oplus W$ as a vector space and the algebra map $g_X$ is given by
\begin{equation*}
  g_X(a) =
  \begin{pmatrix}
    g_V(a) & 0 \\
    \xi(a) & g_W(a)
  \end{pmatrix}
  \in \End_k(X)
  \quad (a \in A)
\end{equation*}
for some $\xi \in \Hom_k(A, \Hom_k(V, W))$. Since $g_X$ is an algebra map, we have
\begin{equation}
  \label{eq:skew-primitive}
  \xi(1) = 0
  \quad \text{and} \quad
  \xi(a b) = \xi(a) \circ g_V(b) + g_W(a) \circ \xi(b)
  \quad (a, b \in A).
\end{equation}
We define $\partial: \Hom_k(V, W) \to \Hom_k(A, \Hom_k(V, W))$ by
\begin{equation*}
  \partial(f)(a) = f \circ g_V(a) - g_W(a) \circ f
  \quad  (f \in \Hom_k(V, W), a \in A).
\end{equation*}
For two linear maps $\xi_i: A \to \Hom_k(V, W)$ ($i = 1, 2$) satisfying~\eqref{eq:skew-primitive}, the corresponding short exact sequences are equivalent if and only if $\xi_1 - \xi_2 \in \Img(\partial)$. Thus the vector space $\Ext_{A}^1(V, W)$ is identified with
\begin{equation}
  \label{eq:matrix-expression-of-Ext1}
  \mathcal{E}_A^1(V, W) := \{ \text{$\xi \in \Hom_k(A, \Hom_k(V, W))$ satisfying \eqref{eq:skew-primitive}} \} / \Img(\partial).
\end{equation}
Theorem~\ref{thm:dim-CF2} is in fact an immediate consequence of Theorem~\ref{thm:CF-n-and-SLF-n} and the following theorem:

\begin{theorem}
  \label{thm:dim-SF2}
  For $M \in \lmod{A}$, we define
  \begin{equation*}
    \Trace^*_{A, M}: \Ext_{A}^1(M, M) = \mathcal{E}^1_A(M, M) \to \Sym(A),
    \quad \xi \mapsto \Trace \circ \xi.
  \end{equation*}
  The map $\Trace^*_{A,L}$ is injective for all $L \in \Irr(A)$. Moreover, we have
  \begin{equation*}
    \Sym_2(A) = \Sym_1(A) \oplus \bigoplus_{L \in \Irr(A)} \Img(\Trace^*_{A,L}).
  \end{equation*}
\end{theorem}

\begin{remark}
  Our construction of the map $\Trace^*_{A,M}$ is inspired from the construction of {\em pseudo-trace functions} introduced by Miyamoto \cite{MR2046807} and further studied in \cite{MR2663653,MR2722373,MR3045342} in relation with conformal field theory and vertex operator algebras.
\end{remark}

\begin{remark}
  Theorem~\ref{thm:dim-SF2} is inspired by the following Okuyama's result: For a symmetric Frobenius algebra $A$, Okuyama \cite{Okuyama81} showed
  \begin{equation*}
    \dim_k \Rey_2(A) = |\Irr(A)| + \sum_{L \in \Irr(A} \dim_k \Ext_{A}^1(L, L)
  \end{equation*}
  (see also Koshitani's review \cite[Section 2]{Koshitani2016}). This formula follows from the above theorem and Remark~\ref{rem:SF-n-and-Rey-n}. We note that Theorem~\ref{thm:dim-SF2} does not require $A$ to be a symmetric Frobenius algebra.
\end{remark}

We give a proof of Theorem~\ref{thm:dim-SF2}. Let $A$ be a finite-dimensional algebra, and write $\Irr(A) = \{ S_1, \dotsc, S_m \}$. For each $i = 1, \dotsc, m$, we fix a primitive idempotent $e_i \in A$ such that $A e_i$ is a projective cover of $S_i$. Set $e = e_1 + \dotsb + e_m$. Then $A^b := e A e$ is a basic algebra and the functor
\begin{equation}
  \label{eq:basic-equiv}
  \lmod{A} \to \lmod{A^b}, \quad X \mapsto e X
\end{equation}
is an equivalence. The following lemma is well-known \cite{MR0009024}, but we give a proof from the viewpoint of Definition~\ref{def:SF-and-SFn}.

\begin{lemma}
  The following map is bijective:
  \begin{equation}
    \label{eq:CF2-lem-SF-iso}
    \Sym(A) \to \Sym(A^b), \quad f \mapsto f|_{A^b}
  \end{equation}  
\end{lemma}
\begin{proof}
  The equivalence \eqref{eq:basic-equiv} induces an equivalence $\bimod{A}{A} \approx \bimod{A^b}{A^b}$ sending $M \in \bimod{A}{A}$ to $e M e$. By~\eqref{eq:SLF-R-mod} and this equivalence, we have
  \begin{equation*}
    \Sym(A)
    \cong \Hom_{\bimod{A}{A}}(A, A^*)
    \cong \Hom_{\bimod{A^b}{A^b}}(A^b, (A^b)^*)
    \cong \Sym(A^b).
  \end{equation*}
  The composition yields the map~\eqref{eq:CF2-lem-SF-iso}.
\end{proof}

The equivalence \eqref{eq:basic-equiv} also induces an isomorphism
\begin{equation}
  \label{eq:CF2-lem-Ext-iso}
  \mathcal{E}_{A}^{1}(V, W)
  = \Ext_A^{1}(V, W)
  \cong \Ext_{A^b}^{1}(e V, e W)
  = \mathcal{E}_{A^b}^{1}(e V, e W)
\end{equation}
for $V, W \in \mathcal{M}$, which sends $\xi \in \mathcal{E}_A^1(V, W)$ to
\begin{equation*}
  \xi^b: A^b \to \Hom_k(e V, e W),
  \quad e a e \mapsto e \xi(e a e)(e v)
  \quad (a \in A, v \in V).
\end{equation*}

\begin{lemma}
  \label{lem:ext-trace-diagram-commute}
  For $V \in \lmod{A}$, the following diagram commutes:
  \begin{equation*}
    \xymatrix@C=64pt@R=16pt{
      \mathcal{E}_{A}^1(V, V)
      \ar[r]^{\Trace^*_{A, V}}
      \ar[d]_{\text{\eqref{eq:CF2-lem-Ext-iso}}}
      & \Sym(A)
      \ar[d]^{\text{\eqref{eq:CF2-lem-SF-iso}}} \\
      \mathcal{E}_{A^b}^1(e V, e V)
      \ar[r]^{\Trace^*_{A^b, e V}}
      & \Sym(A^b)
    }
  \end{equation*}
\end{lemma}
\begin{proof}
  Let $g: A \to \End_k(V)$ be the algebra map defined by the action of $A$. For $\xi \in \mathcal{E}_A^1(V, V)$ and $a \in A$, we have $\xi(e) = \xi(e^2) = \xi(e) g(e) + g(e) \xi(e)$ by \eqref{eq:skew-primitive}. By multiplying $g(e)$ to both sides, we obtain $g(e) \xi(e) g(e) = 0$. Again by \eqref{eq:skew-primitive},
  \begin{align*}
    \Trace(\xi(e a e))
    & = \Trace(\xi(e \cdot e a e \cdot e)) \\
    & = \Trace(\xi(e) g(e a e) g(e)
      + g(e) \xi(e a e) g(e) + g(e) g(e a e) \xi(e)) \\
    & = \Trace(\xi(e) g(e a e) g(e))
      + \Trace(g(e) \xi(e a e) g(e))
      + \Trace(g(e) g(e a e) \xi(e)).
  \end{align*}
  The first term is zero, since $\Trace(\xi(e) g(e a e) g(e)) = \Trace(g(e) \xi(e) g(e) g(e a e))  = \Trace(0) = 0$.
  By a similar computation, the third term is also zero. Thus we have
  \begin{equation*}
    \Trace(\xi(e a e))
    = \Trace(g(e) \xi(e a e) g(e))
    = \Trace(\xi^b(e a e)).
  \end{equation*}
  This means that the diagram in question commutes.
\end{proof}

\begin{proof}[Proof of Theorem~\ref{thm:dim-SF2}]
  Let $A$ be a finite-dimensional algebra, and let $J$ be the Jacobson radical of $A$. Since the $n$-th power of the Jacobson radical of $A^b$ is $e J^n e$, we see that \eqref{eq:CF2-lem-SF-iso} restricts to isomorphisms
  \begin{equation}
    \label{eq:CF2-lem-SF-iso-n}
    \Sym_n(A) \to \Sym_n(A^b), \quad f \mapsto f|_{A^b} \quad (n \in \mathbb{Z}_{+})
  \end{equation}
  Thus, by Lemma~\ref{lem:ext-trace-diagram-commute}, it is sufficient to consider the case where $A$ is basic.

  We assume that $A$ is basic. Then $C := A^*$ is a pointed coalgebra. Let $\Delta$ and $\varepsilon$ denote the comultiplication and the counit of $C$, respectively. We note that the set $\Irr(A)$ is identified with the set
  \begin{equation*}
    G(C) := \{ c \in C \mid \text{$\Delta(c) = c \otimes c$ and $\varepsilon(c) = 1$} \}
  \end{equation*}
  of grouplike elements of $C$. Let $g, h \in G(C)$. By~\eqref{eq:matrix-expression-of-Ext1}, the vector space $\Ext_{A}^1(g, h)$ is identified with the space of $(g, h)$-skew-primitive elements
  \begin{equation*}
    P_{g,h} := \{ x \in C \mid \Delta(x) = x \otimes g + h \otimes x \},
  \end{equation*}
  and the map $\Trace_{A,g}^*$ is just the inclusion map $P_{g,g} \to C$. Thus, to prove this theorem, it is enough to show the following equation:
  \begin{equation}
    \label{eq:SLF2-basic-alg}
    \Sym_2(A) = C_0 \oplus \bigoplus_{g \in G(C)} P_{g,g}.
  \end{equation}
  Let $J$ be the Jacobson radical of $A$. By \cite[Proposition 5.2.9]{MR1243637}, the coradical filtration $\{ C_n \}_{n \ge 0}$ of $C$ is given by $C_n = (A/J^{n+1})^*$. Thus,
  \begin{equation}
    \label{eq:CF2-lem-SF1-SF2}
    \Sym_n(A) = C_{n-1} \cap \Sym(A)
     \quad (n \in \mathbb{Z}_{+}).
  \end{equation}
  For each $g, h \in G(C)$, we choose a subspace $P_{g,h}'$ of $P_{g,h}$ such that $P_{g,h} = P'_{g,h} \oplus k(g - h)$. The Taft-Wilson theorem \cite[Theorem 5.4.1]{MR1243637} states
  \begin{equation}
    \label{eq:Taft-Wilson}
    C_1 = C_0 \oplus \bigoplus_{g, h \in G(C)} P'_{g,h}.
  \end{equation}
  Thus $C_1 \otimes C_1$ is decomposed as follows:
  \begin{equation}
    \label{eq:C1-otimes-C1}
    \begin{aligned}
      C_1 \otimes C_1 & = (C_0 \otimes C_0) \\
      & \oplus \bigoplus_{f, g, h \in G(C)}
      (f \otimes P'_{g,h}) \oplus (P'_{g, h} \otimes f)
      \oplus \bigoplus_{e,f,g,h \in G(C)} P'_{e,f} \otimes P'_{g,h}.
    \end{aligned}
  \end{equation}
  Let $x \in \Sym_2(A)$. By~\eqref{eq:CF2-lem-SF1-SF2} and~\eqref{eq:Taft-Wilson}, we have $x = x_0 + \sum_{g, h \in G(C)} x_{g,h}$ for some $x_0 \in C_0$ and $x_{g,h} \in P'_{g,h}$. Let $\pi_{f,g,h}$ be the projection to $f \otimes P'_{g,h}$ along the direct sum decomposition~\eqref{eq:C1-otimes-C1}. Since $x \in \Sym(A)$, we have
  \begin{equation*}
    \delta_{f,h} \cdot h \otimes x_{g,h}
    = \pi_{f,g,h} \Delta(x)
    = \pi_{f,g,h} \Delta^{\mathrm{cop}}(x)
    = \delta_{f,g} \cdot g \otimes x_{g,h}.
  \end{equation*}
  This implies that $x_{g,h} = 0$ unless $g = h$. Hence,
  \begin{equation*}
    \Sym_2(A) \subset C_0 \oplus \bigoplus_{g \in G} P'_{g,g}
    = C_0 \oplus \bigoplus_{g \in G} P_{g,g}.
  \end{equation*}
  Thus the left-hand side of~\eqref{eq:SLF2-basic-alg} is contained in the right-hand side. It is easy to show the converse inclusion. The proof is done.
\end{proof}

\subsection{Examples}
\label{subsec:cf-examples}

We have no general results on $\CF_n$ for $n \ge 3$. Here we give some computational results on $\CF_n(\mathcal{C})$ for the case where $\mathcal{C}$ is the category of modules over a finite-dimensional Hopf algebra.

Let $H$ be a finite-dimensional Hopf algebra with comultiplication $\Delta$, counit $\varepsilon$ and antipode $S$. We use the Sweedler notation $\Delta(h) = h_{(1)} \otimes h_{(2)}$ to express the comultiplication of $h \in H$. Set $\mathcal{C} = \lmod{H}$. If we identify $\REX(\mathcal{C})$ with $\bimod{H}{H}$, then the action functor $\Act: \mathcal{C} \to \REX(\mathcal{C})$ is given by $\Act(X) = X \otimes_k H$, where the left and the right action of $H$ on $\Act(X)$ are given by
\begin{equation*}
  h \cdot (x \otimes h') = h_{(1)} x \otimes h_{(2)} h'
  \quad \text{and} \quad (x \otimes h') \cdot h = x \otimes h' h,
\end{equation*}
respectively, for $x \in X$ and $h, h' \in H$. A right adjoint of $\Act$ is given as follows: As a vector space, $\Act^{\radj}(M) = M$. The action of $H$ on $\Act^{\radj}(M)$ is given by
\begin{equation*}
  h \cdot m = h_{(1)} m S(h_{(2)}) \quad (h \in H, m \in M).
\end{equation*}
Indeed, one can check that the map
\begin{equation*}
  \Hom_{H}(\Act(X), M) \to \Hom_{\bimod{H}{H}}(X, \Act^{\radj}(M)),
  \quad f \mapsto f(1_H \otimes -)
\end{equation*}
is a natural isomorphism for $X \in \lmod{H}$ and $M \in \bimod{H}{H}$.

In particular, as remarked in \cite{MR3631720}, the algebra $A_{\mathcal{C}} = \Act^{\radj}(\id_{\mathcal{C}}) \in \mathcal{C}$ is the adjoint representation of $H$. Thus the space of class functions is given by
\begin{align*}
  \CF(\lmod{H})
  & = \{ f \in H^* \mid \text{$f(h_{(1)} x S(h_{(2)})) = \varepsilon(h) f(x)$ for all $h, x \in H$} \} \\
  & = \{ f \in H^* \mid \text{$f(a b) = f(b S^2(a))$ for all $a, b \in H$} \}.
\end{align*}
By the above description of $\Act^{\radj}$, we also have
\begin{equation*}
  \CF_n(\lmod{H}) = \{ f \in \CF(\lmod{H}) \mid f(J^n) = 0 \}
\end{equation*}
for all positive integer $n$, where $J$ is the Jacobson radical of $H$. As these expressions show, if $S^2$ is inner, then there are isomorphisms
\begin{equation}
  \label{eq:S2-inner-CF-and-SF}
  \CF(\lmod{H}) \cong \Sym(H)
  \quad \text{and} \quad
  \CF_n(\lmod{H}) \cong \Sym_n(H).
\end{equation}

\newcommand{\Mat}{\mathrm{Mat}}

\begin{example}
  Suppose that the base field $k$ is of characteristic $p > 0$. We consider the cyclic group $G = \langle g \mid g^p = 1 \rangle$ of order $p$. It is easy to see that the Jacobson radical of $k G$ is generated by $x := g - 1$. We note that the set $\{ 1, x, \dotsc, x^{p - 1} \}$ is a basis of $k G$. Since the square of the antipode of $k G$ is the identity, we have
  \begin{equation*}
    \CF_n(\lmod{k G}) = \Sym_n(k G) = \{ f \in (k G)^* \mid \text{$f(x^r) = 0$ for all $r \ge n$} \}.
  \end{equation*}
  Hence the dimension of $\CF_n := \CF_n(\lmod{k G})$ is given by
  \begin{equation*}
    \dim_k \CF_n = n \quad (n = 1, 2, \dotsc, p)
    \quad \text{and} \quad
    \dim_k \CF_n = p \quad (n > p).
  \end{equation*}
  A basis of $\CF := \CF(\lmod{k G})$ can be constructed in the following roundabout but interesting way: There is a matrix representation
  \begin{equation*}
    \renewcommand{\arraystretch}{.9}
    \renewcommand{\arraycolsep}{3pt}
    \rho: k G \to \Mat_p(k), \quad g \mapsto \left(
      \begin{array}{cccc}
        1 & & & 0\\
        1 & 1 \\
          & \ddots & \ddots \\
        0 & & 1 & 1
      \end{array}
    \right).
  \end{equation*}
  Let $\rho_{i j} \in (k G)^*$ be the $(i, j)$-entry of $\rho$. Then the set $\{ \rho_{n 1} \}_{n = 1, \dotsc, p}$ is a basis of $\CF$ such that $\rho_{n 1} \in \CF_n$ and $\rho_{n 1} \not \in \CF_{n-1}$ for all $n = 1, \dotsc, p$ (with convention $\CF_0 = \{ 0 \}$).
\end{example}

\begin{example}
  Suppose that the base field $k$ is of characteristic zero. Let $p \ge 2$ be an integer, and let $q \in k$ be a primitive $2 p$-th root of unity. The algebra $\overline{U}_q := \overline{U}_q(\mathfrak{sl}_2)$ is generated by $E$, $F$ and $K$ subject to the relations
  \begin{equation*}
    E^p = F^p = 0,
    \ K^{2 p} = 1,
    \ K E = q^2 E K,
    \ K F = q^{-2} F K,
    \ [E, F] = \frac{K - K^{-1}}{q - q^{-1}}.
  \end{equation*}
  The algebra $\overline{U}_q$ has the Hopf algebra structure determined by
  \begin{equation*}
    \Delta(E) = E \otimes K + 1 \otimes E,
    \quad \Delta(F) = F \otimes 1 + K^{-1} \otimes F,
    \quad \Delta(K) = K \otimes K.
  \end{equation*}
  The antipode is given by $S(E) = -E K^{-1}$, $S(F) = - K F$ and $S(K) = K^{-1}$ on the generators. Thus $S^2$ is the inner automorphism implemented by $K$. In view of~\eqref{eq:S2-inner-CF-and-SF}, we consider $\Sym_n(\overline{U}_q)$ instead of $\CF_n(\lmod{\overline{U}_q})$.

  An explicit basis of $\Sym(\overline{U}_q)$ is given by Arike \cite{MR2722373}. We recall his construction: For $\alpha \in \{ +, - \}$ and $s \in \{ 1, \dotsc, p \}$, there is an $s$-dimensional simple left $\overline{U}_q$-module $\mathcal{X}_s^{\alpha}$ (see \cite[Subsection 3.4]{MR2722373} for notations). The module $\mathcal{X}_p^{\alpha}$ is projective. For $s < p$, the module $\mathcal{X}_s^{\alpha}$ is not projective. Let $\mathcal{P}_{s}^{\alpha}$ be the projective cover of $\mathcal{X}_s^{\alpha}$. Arike \cite[Subsection 5.1]{MR2722373} showed that $\mathcal{P}_{s}^{\alpha}$ has a matrix presentation of the form
  \begin{equation*}
    \rho^{\alpha}_s: \overline{U}_q \to \Mat_{2p}(k),
    \quad \rho^{\alpha}_s(x) = \left(
      \begin{array}{cccc}
        g_{s}^{\alpha}(x) & 0 & 0 & 0 \\
        a_{s}^{\alpha}(x) & g_{p-s}^{-\alpha}(x) & 0 & 0 \\
        b_{s}^{\alpha}(x) & 0 & g_{p-s}^{-\alpha}(x) & 0 \\
        h_{s}^{\alpha}(x) & a_{p-s}^{-\alpha}(x) & b_{p-s}^{-\alpha}(x) & g_{s}^{\alpha}(x) \\
      \end{array}
    \right),
  \end{equation*}
  where $g_{s}^{\alpha}: \overline{U}_q \to \Mat_s(k)$ is a matrix presentation of $\mathcal{X}_{s}^{\alpha}$ and $a_{s}^{\alpha}$, $b_{s}^{\alpha}$ and $h_{s}^{\alpha}$ are certain matrix-valued linear functions on $\overline{U}_q$ (given by $a_s^{+} = A_{p-s,s}$, $a_s^{-} = C_{s,p-s}$, $b_s^{+} = B_{p-s,s}$, $b_s^{-} = D_{s,p-s}$ $h_{s}^+ = H_s$ and $h_{s}^- = \tilde{H}_s$ with Arike's original notation). Now we define linear forms $\chi_s^{\alpha}$ ($\alpha \in \{ +, - \}$, $s = 1, \dotsc, p$) and $\varphi_{s'}$ ($s' = 1, \dotsc, p - 1$) on $\overline{U}_q$ by
  \begin{equation*}
    \chi_{s}^{\alpha}(x) = \Trace(g_s^{\alpha}(x))
    \quad \text{and} \quad
    \varphi_{s'}(x) = \Trace(h_{s'}^{+}(x)) + \Trace(h_{p-s'}^{-}(x))
    \quad (x \in \overline{U}_q).
  \end{equation*}
  Then the following set is a basis of $\Sym(\overline{U}_q)$ \cite[Theorem 5.5]{MR2722373}:
  \begin{equation*}
    \{ \chi^+_s, \chi^-_s \mid s = 1, \dotsc, p \} \cup \{ \varphi_s \mid s = 1, \dotsc, p - 1 \}.
  \end{equation*}

  Arike's basis respects the filtration of $\Sym(\overline{U}_q)$. More precisely, we have:
  \begin{align}
    \label{eq:Uq-sl2-SLF-1}
    \Sym_1(\overline{U}_q)
    & = \mathrm{span} \{ \chi_s^{+}, \chi_s^{-} \mid 1 \le s \le p \}, \\
    \label{eq:Uq-sl2-SLF-2}
    \Sym_2(\overline{U}_q)
    & = \Sym_1(\overline{U}_q), \\
    \label{eq:Uq-sl2-SLF-3}
    \Sym_3(\overline{U}_q)
    & = \Sym_2(\overline{U}_q) \oplus \mathrm{span} \{ \varphi_s \mid 1 \le s \le p - 1 \}.
  \end{align}
  Indeed, \eqref{eq:Uq-sl2-SLF-1} follows from the fact that $\Sym_1(\overline{U}_q)$ is spanned by the characters of simple modules. Equation \eqref{eq:Uq-sl2-SLF-2} follows from Theorem~\ref{thm:dim-SF2} and the fact that the self-extension vanishes for every simple $\overline{U}_q$-module \cite[p.379]{MR1284788}. To show \eqref{eq:Uq-sl2-SLF-3}, we note $\Lw(\overline{U}_q) = 3$ \cite[p.367]{MR1284788}. Thus we have $\Sym_n(\overline{U}_q) = \Sym(\overline{U}_q)$ for $n \ge 3$. This implies~\eqref{eq:Uq-sl2-SLF-3}.
\end{example}

\section{Hochschild (co)homology}
\label{sec:hochsch-cohom}

\subsection{Hochschild (co)homology of a finite abelian category}

For an algebra $A$, the Hochschild homology and the Hochschild cohomology of $A$ are defined by
\begin{equation*}
  \HH_{\bullet}(A) = \Tor_{\bullet}^{A^e}(A, A)
  \quad \text{and} \quad
  \HH^{\bullet}(A) = \Ext^{\bullet}_{A^e}(A, A),
\end{equation*}
respectively, where $A^e = A \otimes_k A^{\op}$. We note that the 0-th Hochschild cohomology $\HH^0(A) = \Hom_{A^e}(A,A)$ is the center of $A$. It has been known that the modular group $\mathrm{SL}_2(\mathbb{Z})$ acts projectively on the center of a ribbon factorizable Hopf algebra \cite{MR2985696}. Recently, Lentner, Mierach, Schweigert and Sommerh\"auser \cite{2017arXiv170704032L} showed that $\mathrm{SL}_2(\mathbb{Z})$ also acts projectively on the higher Hochschild cohomology of such a Hopf algebra.

A modular tensor category (in the sense of Kerler-Lyubashenko \cite{MR1862634}) is a category-theoretical counterpart of a ribbon factorizable Hopf algebra. The aim of this section is to extend the construction of \cite{2017arXiv170704032L} to modular tensor categories. To accomplish this, we first need to discuss what the Hochschild cohomology of a finite abelian category is. Our proposal is the following definition:

\begin{definition}
  For a finite abelian category $\mathcal{M}$, we define the Hochschild cohomology $\HH^{\bullet}(\mathcal{M})$ of $\mathcal{M}$ by $\HH^{\bullet}(\mathcal{M}) = \Ext^{\bullet}_{\REX(\mathcal{M})}(\id_{\mathcal{M}}, \id_{\mathcal{M}})$.
\end{definition}

If $\mathcal{M} \approx \lmod{A}$ for some finite-dimensional algebra $A$, then $\REX(\mathcal{M})$ is equivalent to $\bimod{A}{A}$ and the identity functor $\id_{\mathcal{M}} \in \REX(\mathcal{M})$ corresponds to $A$ via the equivalence. Since a category equivalence preserves $\Ext^{\bullet}$, we have
\begin{equation*}
  \HH^{\bullet}(\mathcal{M}) 
  = \Ext^{\bullet}_{\REX(\mathcal{M})}(\id_{\mathcal{M}}, \id_{\mathcal{M}})
  \cong \Ext^{\bullet}_{\bimod{A}{A}}(A, A) = \HH^{\bullet}(A),
\end{equation*}
which justifies the definition.

Although it is not directly related to our main purpose of this section, it is also interesting to give a definition of the Hochschild homology of a finite abelian category. Our proposal is:

\begin{definition}
  For a finite abelian category $\mathcal{M}$, we define the Hochschild homology $\HH_{\bullet}(\mathcal{M})$ of $\mathcal{M}$ by $\HH_{\bullet}(\mathcal{M}) = \Ext^{\bullet}_{\REX(\mathcal{M})}(\id_{\mathcal{M}}, \Nak_{\mathcal{M}})^*$, where $\Nak_{\mathcal{M}}$ is the Nakayama functor on $\mathcal{M}$.
\end{definition}

This definition is justified as follows: If $M \leftarrow P_0 \leftarrow P_1 \leftarrow \cdots$ is a projective resolution of $M \in \bimod{A}{A}$, then $\Tor^{A^e}_{\bullet}(A, M)$ is the homology of
\begin{equation}
  \label{eq:chain-cpx-HH_n}
  0 \leftarrow A \otimes_{A^e} P_0 \leftarrow A \otimes_{A^e} P_1 \leftarrow \cdots.
\end{equation}
By the tensor-hom adjunction, the dual of this chain complex is:
\begin{equation*}
  (A \otimes_{A^e} P_{\bullet})^*
  = \Hom_k(A \otimes_{A^e} P_{\bullet}, k)
  \cong \Hom_{A^e}(P_{\bullet}, \Hom_k(A, k))
  = \Hom_{A^e}(P_{\bullet}, A^*).
\end{equation*}
Thus we have $\Tor^{A^e}_{\bullet}(A, M)^* \cong \Ext_{A^e}^{\bullet}(M, A^*)$ by taking the cohomology of the dual of \eqref{eq:chain-cpx-HH_n}. If $\mathcal{M} = \lmod{A}$, then $A$ and $A^*$ corresponds to $\id_{\mathcal{M}}$ and $\Nak_{\mathcal{M}}$, respectively, via the equivalence $\bimod{A}{A} \approx \REX(\mathcal{M})$. Hence we have
\begin{equation*}
  \HH_{\bullet}(\mathcal{M})
  = \Ext_{\REX(\mathcal{M})}^{\bullet}(\id_{\mathcal{M}}, \Nak_{\mathcal{M}})^*
  \cong \Ext_{A^e}^{\bullet}(A, A^*)^*
  \cong \Tor^{A^e}_{\bullet}(A, A) = \HH_{\bullet}(A).
\end{equation*}

\subsection{Formulas of $\HH^{\bullet}$ and $\HH_{\bullet}$ by the adjoint algebra}

Let $H$ be a finite-dimensional Hopf algebra, and let $A$ be the adjoint representation of $H$. It is known that there is an isomorphism
\begin{equation}
  \label{eq:hochschild-adj-alg}
  \Ext_{H}^{\bullet}(k, A) \cong \HH^{\bullet}(H),
\end{equation}
where $k$ is the trivial $H$-module. We now generalize this result to exact module categories. Let $\mathcal{C}$ be a finite tensor category, and let $\mathcal{M}$ be a finite left $\mathcal{C}$-module category over $\mathcal{C}$ with action functor $\Act: \mathcal{C} \to \REX(\mathcal{M})$. 

\begin{theorem}
  \label{thm:Hochschild-adj}
  If $\mathcal{M}$ is exact, then there is a natural isomorphism
  \begin{equation*}
    \Ext_{\mathcal{C}}^{\bullet}(V, \Act^{\radj}(F)) \cong \Ext_{\REX(\mathcal{M})}^{\bullet}(\Act(V), F)
  \end{equation*}
  for $V \in \mathcal{C}$, $F \in \REX(\mathcal{M})$.
\end{theorem}
\begin{proof}
  We set $\mathcal{E} = \REX(\mathcal{M})$ for simplicity. Let $V \leftarrow P^0 \leftarrow P^1 \leftarrow \dotsb$ be a projective resolution of $V$ in $\mathcal{C}$. By applying $\Hom_{\mathcal{E}}(-, \Act^{\radj}(F))$ to this resolution, we have the following commutative diagram:
  \begin{equation*}
    \xymatrix@C=16pt@R=16pt{
      0 \ar[r]
      & \Hom_{\mathcal{C}}(P^0, \Act^{\radj}(F)) \ar[r] \ar[d]_{\cong} 
      & \Hom_{\mathcal{C}}(P^1, \Act^{\radj}(F)) \ar[r] \ar[d]_{\cong} 
      & \Hom_{\mathcal{C}}(P^2, \Act^{\radj}(F)) \ar[r] \ar[d]_{\cong} & \dotsb \\
      0 \ar[r]
      & \Hom_{\mathcal{E}}(\Act(P^0), F) \ar[r]
      & \Hom_{\mathcal{E}}(\Act(P^1), F) \ar[r]
      & \Hom_{\mathcal{E}}(\Act(P^2), F) \ar[r] & \dotsb
    }
  \end{equation*}
  By Lemmas~\ref{lem:act-fun-exact} and~\ref{lem:action-exact-mod-cat}, the sequence $0 \leftarrow \Act(V) \leftarrow \Act(P^0) \leftarrow \Act(P^1) \leftarrow \dotsb$ is a projective resolution of $\Act(V)$. Now the claim is proved by taking the cohomology of the rows of the above commutative diagram.
\end{proof}

\begin{theorem}
  If $\mathcal{M}$ is an exact $\mathcal{C}$-module category and the relative Serre functor of $\mathcal{M}$ is isomorphic to $\id_{\mathcal{M}}$, then there is a natural isomorphism
  \begin{equation*}
    \Ext_{\mathcal{C}}^{\bullet}(\Act^{\radj}(F), V^{**}) \cong \Ext_{\REX(\mathcal{M})}^{\bullet}(F, V \otimes \Nak_{\mathcal{M}})
  \end{equation*}
  for $V \in \mathcal{C}$ and $F \in \REX(\mathcal{M})$
\end{theorem}
\begin{proof}
  We set $\mathcal{E} = \REX(\mathcal{M})$ for simplicity. Let $F \leftarrow P^0 \leftarrow P^1 \leftarrow \dotsb$ be a projective resolution of $F$ in $\mathcal{E}$. By applying $\Hom_{\mathcal{C}}(-, V \otimes \Nak_{\mathcal{M}})$ to this resolution and using Lemma~\ref{lem:SLF-lemma}, we obtain the following commutative diagram:
  \begin{equation*}
    \xymatrix@C=16pt@R=16pt{
      0 \ar[r]
      & \Hom_{\mathcal{E}}(P^0, V \otimes \Nak_{\mathcal{M}}) \ar[r] \ar[d]_{\cong} 
      & \Hom_{\mathcal{E}}(P^1, V \otimes \Nak_{\mathcal{M}}) \ar[r] \ar[d]_{\cong} 
      & \Hom_{\mathcal{E}}(P^2, V \otimes \Nak_{\mathcal{M}}) \ar[r] \ar[d]_{\cong} & \dotsb \\
      0 \ar[r]
      & \Hom_{\mathcal{C}}(\Act^{\radj}(P^0), V^{**}) \ar[r]
      & \Hom_{\mathcal{C}}(\Act^{\radj}(P^1), V^{**}) \ar[r]
      & \Hom_{\mathcal{C}}(\Act^{\radj}(P^2), V^{**}) \ar[r] & \dotsb
    }
  \end{equation*}
  The claim is proved by taking the cohomology of the rows of this commutative diagram.
\end{proof}

Specializing the above theorems, we obtain:

\begin{corollary}
  For an exact left $\mathcal{C}$-module category $\mathcal{M}$, we have
  \begin{equation}
    \label{eq:HH-adj-alg-1}
    \HH^{\bullet}(\mathcal{M}) \cong \Ext_{\mathcal{C}}^{\bullet}(\unitobj, A_{\mathcal{M}}).
  \end{equation}
  If $\Ser_{\mathcal{M}} \cong \id_{\mathcal{M}}$, then we also have an isomorphism
  \begin{equation}
    \label{eq:HH-adj-alg-2}
    \HH_{\bullet}(\mathcal{M}) \cong \Ext_{\mathcal{C}}^{\bullet}(A_{\mathcal{M}}, \unitobj)^*.
  \end{equation}  
\end{corollary}

We consider the case where $\mathcal{C} = \lmod{H}$ for some finite-dimensional Hopf algebra $H$ and $\mathcal{M} = \mathcal{C}$. Let $A$ be the adjoint representation of $H$. Since $A_{\mathcal{M}} \cong A$ in this case, the isomorphism \eqref{eq:HH-adj-alg-1} specializes to \eqref{eq:hochschild-adj-alg} in this case. Since the relative Serre functor of $\mathcal{M}$ is the double dual functor, $\Ser_{\mathcal{M}} \cong \id_{\mathcal{M}}$ if and only if the square of the antipode of $H$ is inner. If this is the case, then we have an isomorphism $\HH_{\bullet}(H) \cong \Ext_{H}^{\bullet}(A, k)^*$ by \eqref{eq:HH-adj-alg-2}.

\subsection{Modular group action on the Hochschild cohomology}

Let $\mathcal{C}$ be a ribbon finite tensor category with braiding $\sigma$ and twist $\theta$. Then the coend $L := \int^{X \in \mathcal{C}} X^* \otimes X$ has a natural structure of a Hopf algebra in $\mathcal{C}$. We note that the algebra $A := (A_{\mathcal{C}}, m_{\mathcal{C}}, u_{\mathcal{C}})$ is dual to the coalgebra $L$, and thus $A$ is also a Hopf algebra (see \eqref{eq:A_S-unit-multiplication} for the definition of $m_{\mathcal{C}}$ and $u_{\mathcal{C}}$). By using the universal property, we define $Q: \unitobj \to A \otimes A$ to be the unique morphism such that the equation
\begin{equation*}
  (\pi_{\mathcal{C}}(X) \otimes \pi_{\mathcal{C}}(Y)) \circ Q
  = (\id_{X} \otimes \sigma_{Y,X^*} \sigma_{X^*, Y} \otimes \id_{Y^*}) \circ (\coev_X \otimes \coev_Y)
\end{equation*}
holds for all $X, Y \in \mathcal{C}$. The morphism $Q$ is dual to the Hopf pairing $\omega: L \otimes L \to \unitobj$ used in \cite{MR1324034,MR1352517,MR1354257} to define the modularity of $\mathcal{C}$. Thus we say that $\mathcal{C}$ is a {\em modular tensor category} if there is a morphism $e: A \otimes A \to \unitobj$ such that $(A, e, Q)$ is a left dual object of $A$ \cite[Definition 5.2.7]{MR1862634}.

Now we suppose that $\mathcal{C}$ is a modular tensor category. Then the Hopf algebra $A$ has a morphism $\lambda: A \to \unitobj$, unique up to sign, such that
\begin{equation}
  \label{eq:ribbon-normalized-condition}
  (\id_A \otimes \lambda) \circ \underline{\Delta}
  = u_{\mathcal{C}} \circ \lambda = (\lambda \otimes \id_A) \circ \underline{\Delta}
  \quad \text{and} \quad (\lambda \otimes \lambda) \circ Q = \id_{\unitobj},
\end{equation}
where $\underline{\Delta}$ is the comultiplication of $A$. We fix such a morphism $\lambda$ and then define two morphisms $\mathfrak{S}, \mathfrak{T}: A \to A$ by
\begin{equation*}
  \mathfrak{S} = (\lambda \otimes \id_A) \circ (m_{\mathcal{C}} \otimes \id_A) \circ (\id_A \otimes Q)
  \quad \text{and} \quad
  \mathfrak{T} = \int_{X \in \mathcal{C}} \theta_X \otimes \id_{X^*}.
\end{equation*}
The morphisms $\mathfrak{S}$ and $\mathfrak{T}$ are the dual of the morphisms $S$ and $T$, respectively, given in \cite[Definition 6.3]{MR1324034}. Thus $\mathfrak{S}$ and $\mathfrak{T}$ are invertible and there is an element $c \in k^{\times}$ such that the following `modular relation' hold:
\begin{equation}
  \label{eq:modular-relation}
  (\mathfrak{S} \mathfrak{T})^3 = c \cdot \mathfrak{S}^2
  \quad \text{and} \quad
  \mathfrak{S}^4 = \theta_{A}^{-1}.
\end{equation}

\begin{theorem}
  \label{thm:modular-group-action}
  With the above notation, we set
  \begin{equation*}
    \widetilde{\mathfrak{S}} = \Ext_{\mathcal{C}}^{\bullet}(\unitobj, \mathfrak{S})
    \quad \text{and} \quad
    \widetilde{\mathfrak{T}} = \Ext_{\mathcal{C}}^{\bullet}(\unitobj, \mathfrak{T}).
  \end{equation*}
  Then we have a well-defined projective representation
  \begin{equation*}
    \mathrm{SL}_2(\mathbb{Z}) \to \mathrm{PGL}\Big( \HH^{\bullet}(\mathcal{C}) \Big),
    \quad 
    \begin{pmatrix}
      0 & -1 \\ 1 & 0
    \end{pmatrix} \mapsto \widetilde{\mathfrak{S}},
    \quad
    \begin{pmatrix}
      1 & 1 \\ 0 & 1
    \end{pmatrix} \mapsto \widetilde{\mathfrak{T}}.
  \end{equation*}
\end{theorem}
\begin{proof}
  It is enough to show that $\widetilde{\mathfrak{S}}$ and $\widetilde{\mathfrak{T}}$ satisfy $(\widetilde{\mathfrak{S}} \widetilde{\mathfrak{T}})^3 = \widetilde{\mathfrak{S}}^2$ and $\widetilde{\mathfrak{S}}^4 = \id$ up to scalar multiple. By the funtorial property of $\Ext$ and \eqref{eq:modular-relation}, we have
  \begin{equation*}
    (\widetilde{\mathfrak{S}} \widetilde{\mathfrak{T}})^3
    = c\cdot \widetilde{\mathfrak{S}}^2
    \quad \text{and} \quad
    \widetilde{\mathfrak{S}}^4
    = \Ext_{\mathcal{C}}^{\bullet}(\unitobj, \theta_{A}^{-1})
    = \Ext_{\mathcal{C}}^{\bullet}(\theta_{\unitobj}^{-1}, A)
    = \id. \qedhere
  \end{equation*}
\end{proof}

Let $H$ be a finite-dimensional ribbon Hopf algebra with universal R-matrix $R$ and ribbon element $v$. We consider the case where $\mathcal{C} = \lmod{H}$. Then $A$ is the adjoint representation of $H$. For $X \in \lmod{H}$, the composition
\begin{equation*}
  A \otimes X
  \xrightarrow{\quad \pi_{\mathcal{C}}(X) \otimes \id_X \quad}
  X \otimes X^* \otimes X
  \xrightarrow{\quad \id_X \otimes \eval_X \quad}
  X
\end{equation*}
sends $a \otimes x \in A \otimes X$ to $a x$ \cite[Subsection 3.7]{MR3631720}. From this, we see that $\pi_{\mathcal{C}}(X)$ is given as follows: Let $\{ x_i \}$ be a basis of $X$, and let $\{ x^i \}$ be the dual basis of $\{ x_i \}$. With the Einstein notation, we have
\begin{equation}
  \label{eq:Hopf-end-dinat-formula}
  \pi_{\mathcal{C}}(X)(a) = a x_i \otimes x^i
  \quad (a \in A).
\end{equation}
The morphism $\lambda$ is in fact a suitably normalized left integral on $H$. The morphism $Q$ can be regarded as an element of $A \otimes_k A$. For simplicity, we express $R$ and $Q$ as $R = R_1 \otimes R_2$ and $Q = Q_1 \otimes Q_2$, respectively. Then the braiding is given by
\begin{equation*}
  \sigma_{X,Y}(x \otimes y) = R_2 y \otimes R_1 x
  \quad (x \in X, y \in Y)
\end{equation*}
for $X, Y \in \lmod{H}$. Let $\{ h_i \}$ and $\{ h^i \}$ be a basis of $H$ and the dual basis of $H^*$, respectively. Then, by \eqref{eq:Hopf-end-dinat-formula} and the definition of $Q$, we have
\begin{align*}
  Q_1 h_i \otimes h^i \otimes Q_2 h_i \otimes h^i
  & = (\pi_{\mathcal{C}}(H) \otimes \pi_{\mathcal{C}}(H)) (Q) \\
  & = h_i \otimes \sigma_{H,H^*} \sigma_{H^*, H}(h^i \otimes h_i) \otimes h^i \\
  & = h_i \otimes R_2' R_1 h^i \otimes R_1' R_2 h_i \otimes h^i \\
  & = S(R_2' R_1) h_i \otimes h^i \otimes R_1' R_2 h_i \otimes h^i,
\end{align*}
where $R' = R'_1 \otimes R'_2$ is a copy of $R$. Thus we have $Q = S(R_2' R_1) \otimes R_1' R_2$. By using the element $Q$, the morphisms $\mathfrak{S}, \mathfrak{T}: A \to A$ are given by
\begin{equation*}
  \mathfrak{S}(a) = \lambda(a Q_1) Q_2
  \quad \text{and} \quad
  \mathfrak{T}(a) = v a
  \quad (a \in A),
\end{equation*}
respectively. Thus our $\mathfrak{S}$ and $\mathfrak{T}$ coincide with those in \cite[Theorem 4.4]{MR1280590}. If we replace $(H, R, v)$ with $(H^{\mathrm{op,cop}}, R, v)$, then the morphisms $\mathfrak{S}$ and $\mathfrak{T}$ coincide with the morphisms considered in \cite{2017arXiv170704032L}.

In \cite{2017arXiv170704032L}, the action $\mathrm{SL}_2(\mathbb{Z}) \to \mathrm{PGL}(\HH^n(H))$ is defined as follows: First, they extend $\mathfrak{S}$ and $\mathfrak{T}$ to cochain maps $\mathfrak{S}^{\bullet}$ and $\mathfrak{T}^{\bullet}$ of a cochain complex $C_1^{\bullet}$ computing the cohomology $\Ext_{H}^{\bullet}(k, A)$. They also established an explicit isomorphism between the complex $C_1^{\bullet}$ and the Hochschild complex $C_2^{\bullet}$ computing the Hochschild cohomology of $H$. The isomorphism $C_1^{\bullet} \cong C_2^{\bullet}$ induces \eqref{eq:hochschild-adj-alg}. The projective action of $\mathrm{SL}_2(\mathbb{Z})$ on $\HH^{\bullet}(H)$ is then given by $\mathfrak{S}^{\bullet}$ and $\mathfrak{T}^{\bullet}$ through \eqref{eq:hochschild-adj-alg}. By the definition of Ext functor, we see that their action is expressed as in Theorem~\ref{thm:modular-group-action}. Thus, in conclusion, we have obtained a generalization of  \cite{2017arXiv170704032L}.

\appendix

\section{Computation of structure morphisms of $\ActRex^{\radj}$}
\label{apdx:act-adj-structure}

\subsection{Bimodule structure of $\iHom$}

Let $\mathcal{C}$ be a rigid monoidal category, and let $\mathcal{M}$ be a closed left $\mathcal{C}$-module category in the sense of Subsection~\ref{subsec:cl-mod-cat}. We establish some results on the natural transformations $\iHomA$, $\iHomB$ and $\iHomB^{\natural}$ introduced in that subsection. For simplicity, we write $\iHom(M, N) = [M, N]$. We recall that $[M, -]$ is defined to be a right adjoint of the functor $\mathcal{C} \to \mathcal{M}$ given by $X \mapsto X \otimes M$. As before, we denote by $\icoev_{(-),M}$ and $\ieval_{M, (-)}$ the unit and the counit of this adjunction, respectively. Then, by~\eqref{eq:module-func-adj-right}, we have
\begin{equation}
  \label{eq:iHom-iso-a-def}
  \iHomA_{X,M,N} = \iHom(\id_M, \id_X \otimes \ieval_{M, N}) \circ \icoev_{X \otimes \iHom(M, N), M}
\end{equation}
for $X \in \mathcal{C}$ and $M, N \in \mathcal{M}$. By~\eqref{eq:iHom-adj}, we have
\begin{equation}
  \label{eq:iHom-iso-b}
  \iHomB_{Y,M,N} = (\iHomB^{\natural}_{Y,M,N} \otimes \id_{Y^*}) \circ (\id_{\iHom(M, N)} \otimes \coev_{Y})
\end{equation}
for $Y \in \mathcal{C}$ and $M, N \in \mathcal{M}$, where
\begin{equation}
  \label{eq:iHom-beta-natural}
  \iHomB^{\natural}_{Y,M,N} = \iHom(M, \ieval_{Y \otimes M, N}) \circ \icoev_{\iHom(Y \otimes M, N) \otimes Y, M}.
\end{equation}
By the zig-zag identities for the adjunction $(-) \otimes M \dashv \iHom(M, -)$, we have
\begin{gather}
  \label{eq:apdx-zig-zag-1}
  \ieval_{M, X \otimes M} \circ (\icoev_{X, M} \otimes \id_M) = \id_{X \otimes M}, \\
  \label{eq:apdx-zig-zag-2}
  [\id_M, \ieval_{M,N}] \circ \icoev_{[M, N], M} = \id_{[M, N]}
\end{gather}
for $X \in \mathcal{C}$ and $N \in \mathcal{M}$. By \eqref{eq:apdx-zig-zag-1}, \eqref{eq:apdx-zig-zag-2} and the naturality of $\ieval_{M, (-)}$, we have
\begin{gather}
  \label{eq:apdx-zig-zag-3}
  \ieval_{M, N} \circ (\iHomA_{X,M,N} \otimes \id_M) = \id_X \otimes \ieval_{M,N}, \\
  \label{eq:apdx-zig-zag-4}
  \ieval_{M, N} \circ (\iHomB^{\natural}_{Y,M,N} \otimes \id_M) = \ieval_{Y \otimes M, N}
\end{gather}
for all objects $X, Y \in \mathcal{C}$ and $M, N \in \mathcal{M}$. We now prove:

\begin{lemma}[$=$ Lemma~\ref{lem:iHom-iso-B}]
  The equations
  \begin{gather}
    \tag{\ref{eq:lem-iHom-iso-B-1}}
    \iHomB_{\unitobj, M, N} = \id_{[M, N]}, \\
    \tag{\ref{eq:lem-iHom-iso-B-2}}
    \iHomB_{X \otimes Y, M, N}
    = (\iHomB_{Y, M, N} \otimes \id_{X^*}) \circ \iHomB_{X, Y \otimes M, N}, \\
    \tag{\ref{eq:lem-iHom-iso-B-3}}
    (\iHomA_{X, M, N} \otimes \id_Y) \circ (\id_X \otimes \iHomB_{Y, M, N})
    = \iHomB_{Y, M, X \otimes N} \circ \iHomA_{X, Y \otimes M,N}
  \end{gather}
  hold for all objects $X, Y \in \mathcal{C}$ and $M, N \in \mathcal{M}$.
\end{lemma}
\begin{proof}
  Equation \eqref{eq:lem-iHom-iso-B-1} is trivial. By the canonical isomorphisms
  \begin{align*}
    & \Hom_{\mathcal{C}}([X \otimes Y \otimes M, N], [M, N] \otimes Y^* \otimes X^*) \\
    & \qquad \qquad \cong \Hom_{\mathcal{C}}([X \otimes Y M, N] \otimes X \otimes Y, [M, N]) \\
    & \qquad \qquad \cong \Hom_{\mathcal{C}}([X \otimes Y M, N] \otimes X \otimes Y \otimes M, N),
  \end{align*}
  we see that \eqref{eq:lem-iHom-iso-B-2} is equivalent to the equation
  \begin{equation*}
    \ieval_{M, N} \circ (\iHomB^{\natural}_{X \otimes Y, M, N} \otimes \id_M)
    = \ieval_{M,N} \circ (\iHomB^{\natural}_{Y, M, N} \otimes \id_M) \circ (\iHomB^{\natural}_{X, Y \otimes M, N} \otimes \id_Y \otimes \id_M).
  \end{equation*}
  By~\eqref{eq:apdx-zig-zag-4}, the both sides are equal to $\ieval_{X \otimes Y \otimes M, N}$. Thus \eqref{eq:lem-iHom-iso-B-2} is verified. In a similar way, we see that \eqref{eq:lem-iHom-iso-B-3} is equivalent to the equation
  \begin{equation*}
    \begin{aligned}
      & \ieval_{M, X \otimes N} \circ (\iHomA_{X,M,N} \otimes \id_M) \circ (\id_X \otimes \iHomB^{\natural}_{Y,M,N} \otimes \id_M) \\
      & \qquad \qquad = \ieval_{M, X \otimes N} \circ (\iHomB^{\natural}_{Y,M,X \otimes N} \otimes \id_M) \circ (\iHomA_{X, Y \otimes M,N} \otimes \id_Y \otimes \id_M).
    \end{aligned}
  \end{equation*}
  By \eqref{eq:apdx-zig-zag-1}--\eqref{eq:apdx-zig-zag-4}, the both sides are equal to $\id_X \otimes \ieval_{Y \otimes M, N}$. The proof is done.
\end{proof}

In view of this lemma, we have defined the natural isomorphism
\begin{equation*}
  \iHomC_{X,M,N,Y}: X \otimes [M, N] \otimes Y^* \to [Y \otimes M, X \otimes N]
  \quad (X, Y \in \mathcal{C}, M, N \in \mathcal{M})
\end{equation*}
by \eqref{eq:iHom-iso-c-def}. The following Lemmas~\ref{lem:apdx-coev-XM} and~\ref{lem:apdx-comp-and-alpha} will be used in later:

\begin{lemma}
  \label{lem:apdx-coev-XM}
  For all $X \in \mathcal{C}$ and $M \in \mathcal{M}$, the following equation holds:
  \begin{gather}
    \label{eq:apdx-coev-XM}
    \icoev_{\unitobj, X \otimes M}
    = \iHomC_{X,M,M,X} \circ (\id_X \otimes \icoev_{\unitobj, M} \otimes \id_{X^*}) \circ \coev_X.
  \end{gather}
\end{lemma}
\begin{proof}
  By the definition of $\iHomC$, the claim is equivalent to that the equation
  \begin{equation*}
    \iHomB^{\natural}_{X, M, X \otimes M} \circ (\icoev_{\unitobj, X \otimes M} \otimes \id_X)
    = \iHomA_{X, M, M} \circ (\id_X \otimes \icoev_{\unitobj, M})
  \end{equation*}
  holds for all $X \in \mathcal{C}$ and $M \in \mathcal{M}$. By \eqref{eq:apdx-zig-zag-1}--\eqref{eq:apdx-zig-zag-4}, the both sides correspond to the identity morphism under the canonical isomorphism
  \begin{equation*}
    \Hom_{\mathcal{C}}(X, \iHom(X \otimes M, M)) \cong \Hom_{\mathcal{M}}(X \otimes M, X \otimes M). \qedhere
  \end{equation*}
\end{proof}

\begin{lemma}
  \label{lem:apdx-comp-and-alpha}
  For all $M_1, M_2, M_3 \in \mathcal{M}$, the following diagram commutes:
  \begin{equation*}
    \xymatrix@C=0pt@R=12pt{
      & [M_2, M_3] \otimes [M_1, M_2] \ar[dd]^{\icomp}
      \ar[ld]_{\iHomA_{[M_2,M_3], M_1, M_2} \qquad \qquad}
      \ar[rd]^{\qquad \qquad \qquad [\ieval_{M_2, M_1}, \id] \otimes \id} \\
      [M_1, [M_2, M_3] \otimes M_2]
      \ar[rd]_{[\id, \ieval_{M_2, M_3}] \qquad \qquad}
      & & [[M_1, M_2] \otimes M_1, M_3] \otimes [M_1, M_2]
      \ar[ld]^{\qquad \qquad \iHomB^{\natural}_{[M_1, M_2], M_1, M_3}} \\
      & [M_1, M_3]
    }
  \end{equation*}
\end{lemma}
\begin{proof}
  By the naturality of $\ieval_{M_1, (-)}$ and \eqref{eq:apdx-zig-zag-3}, we compute
  \begin{align*}
    & \ieval_{M_1, M_3} \circ ([\id_{M_1}, \ieval_{M_2, M_3}] \iHomA_{[M_2, M_3], M_1, M_3} \otimes \id_{M_1}) \\
    & \qquad = \ieval_{M_2, M_3} \circ \ieval_{M_1, [M_2, M_3] \otimes M_2} \circ (\iHomA_{[M_2, M_3], M_1, M_3} \otimes \id_{M_1}) \\
    & \qquad = \ieval_{M_2, M_3} \circ (\id_{[M_2, M_3]} \otimes \ieval_{M_1, M_2}) = \ieval^{(3)}_{M_1, M_2, M_3}.
  \end{align*}
  This shows the commutativity of the left triangle of the diagram. Similarly, by \eqref{eq:apdx-zig-zag-4} and the dinaturality of $\ieval_{(-), M_3}$, we compute
  \begin{align*}
    & \ieval_{M_1, M_3} \circ ((\iHomB^{\natural}_{[M_1, M_2], M_1, M_3} \circ ([\ieval_{M_2, M_1}, \id_{M_3}] \otimes \id_{[M_1, M_2]})) \otimes \id_{M_1}) \\
    & \qquad \qquad = \ieval_{[M_1, M_2] \otimes M_1, M_3} \circ ([\ieval_{M_2, M_1}, \id_{M_3}] \otimes \id_{[M_1, M_2]} \otimes \id_{M_1}) \\
    & \qquad \qquad = \ieval_{M_2, M_3} \circ (\id_{[M_2, M_3]} \otimes \ieval_{M_1, M_2}) = \ieval^{(3)}_{M_1, M_2, M_3}.
  \end{align*}
  Thus the left triangle of the diagram also commutes.
\end{proof}

\subsection{Structure morphisms of $\Act^{\radj}$}

Now we consider the case where $\mathcal{C}$ is a finite tensor category and $\mathcal{M}$ is a finite left $\mathcal{C}$-module category. To save space, we set $\overline{\Act} = \Act^{\radj}_{\mathcal{M}}$. Let $\xi^{(\ell)}$ and $\xi^{(r)}$ be the left and the right $\mathcal{C}$-module structure of $\overline{\Act}$, respectively. By \eqref{eq:module-func-adj-right} and its right module version, we have
\begin{equation}
  \label{eq:apdx-act-adj-C-bimod-def}
  \xi^{(\ell)}_{X, F} = \overline{\Act}(\id_X \otimes \varepsilon_{F}) \circ \eta_{X \otimes \overline{\Act}(F)}
  \quad \text{and} \quad
  \xi^{(r)}_{F, X} = \overline{\Act}(\varepsilon_{F} \otimes \id_X) \circ \eta_{\, \overline{\Act}(F) \otimes X}
\end{equation}
for $F \in \REX(\mathcal{M})$ and $X \in \mathcal{C}$. By using the universal dinatural transformation $\pi_F$ of the end $\overline{\Act}(F)$, these structure morphisms are given as follows:

\begin{lemma}[$=$ Lemma~\ref{lem:act-fun-adj-C-bimod}]
  The equations
  \begin{align}
    \tag{\ref{eq:act-fun-adj-C-bimod-l-def-2}}
    \pi_{X \otimes F}(M) \circ \xi^{(\ell)}_{X,F}
    & = \iHomA_{X, M, F(M)} \circ (\id_X \otimes \pi_{F}(M)), \\
    \tag{\ref{eq:act-fun-adj-C-bimod-r-def-2}}
    \pi_{F \otimes X}(M) \circ \xi^{(r)}_{X,F}
    & = \iHomB^{\natural}_{X, M, F(X \otimes M)} \circ (\pi_{F}(X \otimes M) \otimes \id_X)
  \end{align}
  hold for all $F \in \REX(\mathcal{M})$ and $X \in \mathcal{C}$.
\end{lemma}
\begin{proof}
  \allowdisplaybreaks
  Equation~\eqref{eq:act-fun-adj-C-bimod-l-def-2} is proved as follows:
  \begin{align*}
    & \pi_{X \otimes F}(M) \circ \xi^{(\ell)}_{X,F} \\
    & = [\id_M, \id_X \otimes \varepsilon_{F, M}] \circ \pi_{\Act(X \otimes \overline{\Act}(F))}(M) \circ \eta_{X \otimes \overline{\Act}(F)} \\
    & = [\id_M, (\id_X \otimes \ieval_{M, F(M)}) \circ (\id_X \otimes \pi_{F}(M) \otimes \id_M)]
      \circ \icoev_{X \otimes \overline{\Act}(F), M} \\
    & = [\id_M, \id_X \otimes \ieval_{M, F(M)}] \circ \icoev_{X \otimes [M, F(M)], M} \circ (\id_X \otimes \pi_{F}(M)) \\
    & = \iHomA_{X, M, F(M)} \circ (\id_X \otimes \pi_{F}(M)).
  \end{align*}
  Here, the first equality follows from~\eqref{eq:apdx-act-adj-C-bimod-def} and the naturality of $\pi_{(-)}(M)$, the second from \eqref{eq:act-fun-adj-counit} and~\eqref{eq:act-fun-adj-unit}, and the third from the naturality of $\icoev_{(-), M}$.

  To prove equation~\eqref{eq:act-fun-adj-C-bimod-r-def-2}, we note that the symbol $\varepsilon_F \otimes \id_X$ in~\eqref{eq:apdx-act-adj-C-bimod-def} means the natural transformation whose component is given by $(\varepsilon_F \otimes \id_X)_M = \varepsilon_{F, X \otimes M}$ for $M \in \mathcal{M}$. Thus we compute:
  \begin{align*}
    & \pi_{F \otimes X}(M) \circ \xi^{(r)}_{X,F} \\
    & = [\id_M, \varepsilon_{F, X \otimes M}] \circ \pi_{\Act(\overline{\Act}(F) \otimes X)}(M) \circ \eta_{\, \overline{\Act}(F) \otimes X} \\
    & = [\id_M, \ieval_{X \otimes M, F(X \otimes M)}] \circ [\id_M, \pi_{F}(X \otimes M) \otimes \id_{X \otimes M}]
      \circ \coev_{\overline{\Act}(F) \otimes X, M} \\
    & = [\id_M, \ieval_{X \otimes M, F(X \otimes M)}] \circ \coev_{[X \otimes M, F(X \otimes M)] \otimes X, M}
      \circ (\pi_{F}(X \otimes M) \otimes \id_X) \\
    & = \iHomB^{\natural}_{X, M, F(X \otimes M)} \circ (\pi_{F}(X \otimes M) \otimes \id_X)
  \end{align*}
  in a similar way as above. The proof is done.
\end{proof}

We recall that $\Act = \Act_{\mathcal{M}}$ is a strict monoidal functor. Hence its right adjoint $\overline{\Act}$ has a structure of a monoidal functor. We denote the structure morphisms by
\begin{equation*}
  \mu^{(2)}_{F,G}: \overline{\Act}(F) \otimes \overline{\Act}(G) \to \overline{\Act}(F G)
  \quad \text{and} \quad \mu^{(0)}: \unitobj \to \overline{\Act}(\id_{\mathcal{M}})
\end{equation*}
for $F, G \in \REX(\mathcal{M})$. With the use of $\eta$ and $\varepsilon$, they are expressed by
\begin{equation}
  \label{eq:apdx-mu-def}
  \mu^{(2)}_{F, G} = \overline{\Act}(\varepsilon_F \circ \varepsilon_G) \circ \eta_{\, \overline{\Act}(F) \otimes \overline{\Act}(G)}
  \quad \text{and} \quad
  \mu^{(0)} = \eta_{\unitobj},
\end{equation}
where $\varepsilon_F \circ \varepsilon_G$ means the tensor product of $\varepsilon_F$ and $\varepsilon_G$ in $\REX(\mathcal{M})$, or, equivalently, the horizontal composition of $\varepsilon_F$ and $\varepsilon_G$.

\begin{lemma}[$=$ Lemma~\ref{lem:act-fun-adj-monoidal}]
  The equations
  \begin{gather}
    \tag{\ref{eq:act-fun-adj-monoidal-2}}
    \pi_{F G}(M) \circ \mu^{(2)}_{F,G} = \icomp_{M, G(M), F G(M)} \circ (\pi_{F}(G(M)) \otimes \pi_{G}(M)), \\
    \tag{\ref{eq:act-fun-adj-monoidal-0}}
    \pi_{\id_{\mathcal{M}}}(M) \circ \mu^{(0)} = \icoev_{\unitobj, M}
  \end{gather}
  hold for all $F, G \in \REX(\mathcal{M})$ and $M \in \mathcal{M}$.
\end{lemma}
\begin{proof}
  Equation~\eqref{eq:act-fun-adj-monoidal-0} follows from \eqref{eq:act-fun-adj-unit} and \eqref{eq:apdx-mu-def}. To prove \eqref{eq:act-fun-adj-monoidal-2}, we set
  \begin{equation*}
    w = \eval^{(3)}_{M, G(M), F G(M)} \circ (\pi_{F G}(M) \otimes \pi_{G}(M) \otimes \id_M).
  \end{equation*}
  We note that there is an isomorphism
  \begin{equation}
    \label{eq:apdx-pf-fun-adj-monoidal}
    \Hom_{\mathcal{C}}(\overline{\Act}(F) \otimes \overline{\Act}(G), [M, F G(M)])
    \cong \Hom_{\mathcal{C}}(\overline{\Act}(F) \otimes \overline{\Act}(G) \otimes M, F G(M)).
  \end{equation}
  The right-hand side of \eqref{eq:act-fun-adj-monoidal-2} corresponds to $w$ via \eqref{eq:apdx-pf-fun-adj-monoidal}. On the other hand, the left-hand side of \eqref{eq:act-fun-adj-monoidal-2} corresponds to $(\varepsilon_F \circ \varepsilon_G)_M$ via \eqref{eq:apdx-pf-fun-adj-monoidal}.  By \eqref{eq:act-fun-adj-counit} and the definition of the horizontal composition, we have
  \begin{align*}
    (\varepsilon_F \circ \varepsilon_G)_M
    = \varepsilon_{F, G(M)} \circ (\id_{\overline{\Act}(F)} \otimes \varepsilon_{G, M})
    = w.
  \end{align*}
 Thus \eqref{eq:act-fun-adj-monoidal-2} is verified. The proof is done.
\end{proof}

\subsection{Commutativity of $A_{\mathcal{S}}$}

For a $\mathcal{C}$-module full subcategory $\mathcal{S} \subset \mathcal{M}$, we have proved that the end $A_{\mathcal{S}} := \int_{S \in \mathcal{S}} \iHom(S, S)$ has the half-braiding $\sigma_{\mathcal{S}}$ given by the commutative diagram \eqref{eq:act-fun-adj-half-bra}. Namely, the equation
\begin{equation}
  \label{eq:apdx-act-fun-adj-half-bra}
  \iHomA_{X, W, W} \circ (\id_X \otimes \pi_{\mathcal{S}}(W)) \circ \sigma_{\mathcal{S}}(X)
  = \iHomB^{\natural}_{X, W, X \otimes W} \circ (\pi_{\mathcal{S}}(X \otimes W) \otimes \id_X)
\end{equation}
holds for all $X \in \mathcal{C}$ and $W \in \mathcal{S}$, where $\pi_{\mathcal{S}}(X): A_{\mathcal{S}} \to \iHom(X, X)$ is the universal dinatural transformation. Let $m_{\mathcal{S}}: A_{\mathcal{S}} \otimes A_{\mathcal{S}} \to A_{\mathcal{S}}$ be the multiplication of $A_{\mathcal{S}}$, and set $m^{\op}_{\mathcal{S}} = m_{\mathcal{S}} \circ \sigma_{\mathcal{S}}(A_{\mathcal{S}})$.

\begin{proof}[Proof of Theorem~\ref{thm:comm-alg-from-mod-subcat}]
  The claim of this theorem is that $\mathbf{A}_{\mathcal{S}} = (A_{\mathcal{S}}, \sigma_{\mathcal{S}})$ is a commutative algebra in $\mathcal{Z}(\mathcal{C})$. Thus it is sufficient to show $m_{\mathcal{S}}^{\op} = m_{\mathcal{S}}$. We fix $W \in \mathcal{S}$ and set $E = [W, W]$ for simplicity. Then we compute
  \begin{align*}
    \pi_{\mathcal{S}}(W) & \, \circ m_{\mathcal{S}}^{\op} = \icomp_{W,W,W} \circ (\pi_{\mathcal{S}}(W) \otimes \pi_{\mathcal{S}}(W)) \circ \sigma_{\mathcal{S}}(A_{\mathcal{S}}) \\
    & = [\id_W, \ieval_{W,W}] \circ \iHomA_{E, W, W}
      \circ (\pi_{\mathcal{S}}(W) \otimes \id_{X}) \circ \sigma_{\mathcal{S}}(E)
      \circ (\id_{A_{\mathcal{S}}} \otimes \pi_{\mathcal{S}}(W)) \\
    & = [\id_W, \ieval_{W,W}] \circ \iHomB^{\natural}_{E, W, E \otimes W}
      \circ (\pi_{\mathcal{S}}(E \otimes W) \otimes \pi_{\mathcal{S}}(W)) \\
    & = \iHomB^{\natural}_{E, W, W} \circ ([\id_{E \otimes W}, \ieval_{W,W}] \otimes \id_E)
      \circ (\pi_{\mathcal{S}}(E \otimes W) \otimes \pi_{\mathcal{S}}(W)) \\
    & = \iHomB^{\natural}_{E, W, W} \circ ([\ieval_{W,W}, \id_W] \otimes \id_E)
      \circ (\pi_{\mathcal{S}}(W) \otimes \pi_{\mathcal{S}}(W)) \\
    & = \icomp_{W,W,W} \circ (\pi_{\mathcal{S}}(W) \otimes \pi_{\mathcal{S}}(W))
      = (\pi_{\mathcal{S}}(W) \otimes \pi_{\mathcal{S}}(W)) \circ m_{\mathcal{S}}.
  \end{align*}
  Here, the first and the last equalities follow from the definition of $m_{\mathcal{S}}$, the second and the sixth from Lemma~\ref{lem:apdx-comp-and-alpha}, the third from~\eqref{eq:apdx-act-fun-adj-half-bra}, the fourth from the naturality of $\iHomB^{\natural}_{E, W, (-)}$, and the fifth from the dinaturality of $\pi_{\mathcal{S}}$.  We now obtain $m_{\mathcal{S}}^{\op} = m_{\mathcal{S}}$ by the universal property of $A_{\mathcal{S}}$. The proof is done.
\end{proof}

\section{On the properties of the pivotal trace}
\label{sec:additivity-trace}

We complement properties of the trace in a pivotal module category. Let $\mathcal{C}$ be a finite tensor category, and let $\mathcal{M}$ be an exact left $\mathcal{C}$-module category. For simplicity, we write $[M, N] = \iHom(M, N)$. By the definition of the relative Serre functor, there is a natural isomorphism
\begin{equation}
  \tag{\ref{eq:def-relative-Serre}}
  \iHomD_{M,N}: [M, N]^* \to [N, \Ser(M)]
  \quad (M, N \in \mathcal{M}),
\end{equation}
where $\Ser = \Ser_{\mathcal{M}}$. There is also a natural isomorphism
\begin{equation}
  \tag{\ref{eq:relative-Serre-module-functor}}
  \zeta_{X,M}: X^{**} \otimes \Ser(M) \to \Ser(X \otimes M)
  \quad (X \in \mathcal{C}, M \in \mathcal{M})
\end{equation}
such that $\iHomD$ is an isomorphism of $\mathcal{C}$-bimodule functors from $\mathcal{M}^{\op} \times \mathcal{M}$ to ${}_{(-)^{**}}\mathcal{C}$, that is, the equations $\zeta_{\unitobj, M} = \id_{\Ser(M)}$ and
\begin{equation}
  \label{eq:iHom-appendix-eq-1}
  \begin{aligned}[]
    [\id_{X \otimes N}, \zeta_{Y,M}] \circ \, & \iHomD_{Y \otimes M, X \otimes N} \circ (\iHomC_{X,M,N,Y}^*)^{-1} \\
    & = \iHomC_{Y^{**}, N, \Ser(M), X}^{} \circ (\id_{Y^{**}} \otimes \iHomD_{M,N} \otimes \id_{X^*})
  \end{aligned}
\end{equation}
hold for all objects $X, Y \in \mathcal{C}$ and $M, N \in \mathcal{M}$.

Now we suppose that $\mathcal{C}$ and $\mathcal{M}$ are pivotal with pivotal structures $j$ and $j'$, respectively. By Definition~\ref{def:pivotal-mod-cat}, we have
\begin{equation}
  \label{eq:apdx-pivotal-def}
  \zeta_{X,M} \circ j'_{X \otimes M} = j_X \otimes j'_M
  \quad (X \in \mathcal{C}, M \in \mathcal{M}).
\end{equation}
The trace $\trace_{\mathcal{M}}$, defined in Subsection~\ref{subsec:pivotal-mod-cat}, is characterized by
\begin{equation}
  \label{eq:apdx-trace-M}
  \iHomD_{M, M} \circ \trace_{\mathcal{M}}(M)^* = [\id_M, j'_M] \circ \icoev_{\unitobj, M}
  \quad (M \in \mathcal{M}).
\end{equation}
We recall that $\trace_{\mathcal{C}}$ is defined by $\trace_{\mathcal{C}}(X) = \eval_{X^*} \circ (j_X \otimes \id_X)$ for $X \in \mathcal{C}$. Thus,
\begin{align}
  \label{eq:apdx-trace-C}
  \trace_{\mathcal{C}}(X)^*
  = (\id_{X^{**}} \otimes j_{X}^*) \circ \coev_{X^{**}}
  = (j_X \otimes \id_{X^*}) \circ \coev_{X}.
\end{align}

\begin{lemma}
  \label{lem:pivotal-trace-mod-cat-dinat}
  The trace $\trace_{\mathcal{M}}$ is dinatural, and the equation
  \begin{equation}
    \label{eq:trace-mod-cat-lem}
    \trace_{\mathcal{M}}(X \otimes M) \circ \iHomC_{X,M,M,X}
    = \trace_{\mathcal{C}}(X) \circ (\id_X \otimes \trace_{\mathcal{M}}(M) \otimes \id_{X^*})
  \end{equation}
  holds for all objects $M \in \mathcal{M}$ and $X \in \mathcal{C}$.   
\end{lemma}
\begin{proof}
  \allowdisplaybreaks
  The dinaturality of $\trace_{\mathcal{M}}$ follows from the naturality of $j'$ and the dinaturality of $\icoev_{\unitobj, (-)}$. Equation \eqref{eq:trace-mod-cat-lem} is proved as follows:
  \begin{align*}
    & (\id_{X^{**}} \otimes \iHomD_{M,M} \otimes \id_{X^*}) \circ (\trace_{\mathcal{C}}(X) \circ (\id_X \otimes \trace_{\mathcal{M}}(M) \otimes \id_{X^*}))^* \\
    & = (\id_{X^{**}} \otimes [\id_M, j'_M] \icoev_{\unitobj, M} \otimes \id_{X^*})
      \circ (j_X \otimes \id_{X^*}) \circ \coev_{X}
      \quad \text{(by~\eqref{eq:apdx-trace-M},~\eqref{eq:apdx-trace-C})}\\
    & = (j_X \otimes [\id_M, j'_M] \otimes \id_{X^*})
      \circ (\id_X \otimes \icoev_{\unitobj, M} \otimes \id_{X^*}) \circ \coev_{X} \\
    & = (j_X \otimes [\id_M, j'_M] \otimes \id_{X^*})
      \circ \iHomC_{X,M,M,X}^{-1} \circ \icoev_{X \otimes M}
      \quad \text{(by~\eqref{eq:apdx-coev-XM})} \\
    & = \iHomC_{X^{**},M,\Ser(M),X}^{-1} \circ [\id_{X \otimes M}, j_X \otimes j'_M] \circ \icoev_{X \otimes M}
      \quad \text{(by the naturality of $\iHomC$)} \\
    & = \iHomC_{X^{**},M,\Ser(M),X}^{-1} \circ [\id_{X \otimes M}, \zeta_{X,M}] \circ [\id_{X \otimes M}, j'_{X \otimes M}] \circ \icoev_{\unitobj, X \otimes M}
      \quad \text{(by~\eqref{eq:apdx-pivotal-def})} \\
    & = \iHomC_{X^{**},M,\Ser(M),X}^{-1} \circ [\id_{X \otimes M}, \zeta_{X,M}] \circ \iHomD_{X \otimes M, X \otimes M} \circ \trace_{\mathcal{M}}(X \otimes M)^*
      \quad \text{(by~\eqref{eq:apdx-trace-M})} \\
    & = (\id_{X^{**}} \otimes \iHomD_{M,M} \otimes \id_{X^*}) \circ \iHomC^{*}_{X,M,M,X} \circ \trace_{\mathcal{M}}(X \otimes M)^*
      \quad \text{(by~\eqref{eq:iHom-appendix-eq-1})}. \qedhere
  \end{align*}
\end{proof}

For a morphism $f: M \to M$ in $\mathcal{M}$, we have defined $\mathrm{ptr}_{\mathcal{M}}(f) \in k$ of $f$ by
\begin{equation}
  \tag{\ref{eq:def-ptrace}}
  \trace(M) \circ [\id_M, f] \circ \icoev_{\unitobj, M} = \mathrm{ptr}(f) \cdot \id_{\unitobj}.
\end{equation}

\begin{proposition}
  \label{prop:apdx-piv-trace}
  For morphisms $f: M \to N$ and $g: N \to M$ in $\mathcal{M}$, we have
  \begin{equation}
    \label{eq:piv-trace-cyclicity}
    \mathrm{ptr}(f g) = \mathrm{ptr}(g f).
  \end{equation}
  For morphisms $f: M \to M$ in $\mathcal{M}$ and $a: X \to X$ in $\mathcal{C}$, we have
  \begin{equation}
    \label{eq:piv-trace-mult}
    \mathrm{ptr}(a \otimes f) = \mathrm{ptr}(a) \cdot \mathrm{ptr}(f)
  \end{equation}
\end{proposition}
\begin{proof}
  Equation~\eqref{eq:piv-trace-cyclicity} follows from the dinaturality of $\trace_{\mathcal{M}}$ and $\icoev_{\unitobj, (-)}$.
  Equation~\eqref{eq:piv-trace-mult} follows from \eqref{eq:trace-mod-cat-lem}.
\end{proof}

For $M \in \mathcal{M}$, we have defined the internal character $\ich_{\mathcal{M}}(M) \in \CF(\mathcal{M})$ by
\begin{equation}
  \tag{\ref{eq:def-internal-char}}
  \ich_{\mathcal{M}}(M) = \trace_{\mathcal{M}}(M) \circ \pi_{\mathcal{M}}(M),
\end{equation}
where $\pi_{\mathcal{M}}: A_{\mathcal{M}} \to [M, M]$ is the universal dinatural transformation.

\begin{proposition}[$=$ Lemma~\ref{lem:character-multiplicative}]
  For all $X \in \mathcal{C}$ and $M \in \mathcal{M}$, we have
  \begin{equation*}
    \ich_{\mathcal{M}}(X \otimes M) = \ich_{\mathcal{C}}(X) \star \ich_{\mathcal{M}}(M),
  \end{equation*}
  where $\star$ is the action~\eqref{eq:class-ft-action} of $\CF(\mathcal{C})$ on $\CF(\mathcal{M})$.
\end{proposition}
\begin{proof}
  We recall that $A_{\mathcal{M}}$ has the $\mathsf{Z}$-coaction $\delta_{\mathcal{M}}: A_{\mathcal{M}} \to \mathsf{Z}(A_{\mathcal{M}})$ induced from the half-braiding of $A_{\mathcal{M}}$. By definition,
  \begin{equation*}
    (\id_X \otimes \pi_{\mathcal{M}}(M) \otimes \id_{X^*}) \circ \pi^{\mathsf{Z}}(A_{\mathcal{M}}) \circ \delta_{\mathcal{M}}
    = \iHomC_{X,M,M,X}^{-1} \circ \pi_{\mathcal{M}}(X \otimes M)
  \end{equation*}
  for all objects $X \in \mathcal{C}$ and $M \in \mathcal{M}$. Thus, by Lemma~\ref{lem:pivotal-trace-mod-cat-dinat},  we have
  \begin{align*}
    & \ich_{\mathcal{C}}(X) \star \ich_{\mathcal{M}}(M)
    = \ich_{\mathcal{C}}(X) \circ \mathsf{Z}(\ich_{\mathcal{M}}(M)) \circ \delta_{\mathcal{M}} \\
    & = \trace_{\mathcal{C}}(X) \circ (\id_X \otimes (\trace_{\mathcal{M}}(M) \circ \pi_{\mathcal{M}}(M)) \otimes \id_{X^*})
      \circ \pi^{\mathsf{Z}}(A_{\mathcal{M}}) \circ \delta_{\mathcal{M}} \\
    & = \trace_{\mathcal{C}}(X) \circ (\id_X \otimes \trace_{\mathcal{M}}(M) \otimes \id_{X^*})
      \circ \iHomC_{X,M,M,X}^{-1} \circ \pi_{\mathcal{M}}(X \otimes M) \\
    & = \trace_{\mathcal{C}}(X \otimes M) \circ  \pi_{\mathcal{M}}(X \otimes M)
      = \ich_{\mathcal{M}}(X \otimes M). \qedhere
  \end{align*}
\end{proof}

Finally, we give a proof of Lemma~\ref{lem:character-additive} in the following general form: Let $\mathcal{A}$ and $\mathcal{B}$ be abelian categories, and let $\langle -, - \rangle: \mathcal{A}^{\op} \times \mathcal{A} \to \mathcal{B}$ be a functor that is additive and exact in each variable. Let $X$ and $Y$ be objects of $\mathcal{B}$. Suppose that there are two dinatural transformations $d(M): X \to \langle M, M \rangle$ and $e(M): \langle M, M \rangle \to Y$ ($M \in \mathcal{A}$). For a morphism $f: M \to M$ in $\mathcal{A}$, we define
\begin{equation*}
  t(f) = e(M) \circ \langle \id_M, f \rangle \circ d(M) \in \Hom_{\mathcal{B}}(X, Y).
\end{equation*}

\begin{proposition}
  \label{apdx:prop-additivity-piv-trace}
  Suppose that
  \begin{equation*}
    \xymatrix@R=16pt{
      0 \ar[r]
      & M_1 \ar[r]^{r} \ar[d]_{f_1}
      & M_2 \ar[r]^{s} \ar[d]_{f_2}
      & M_3 \ar[r] \ar[d]_{f_3} & 0 \\
      0 \ar[r]
      & M_1 \ar[r]^{r}
      & M_2 \ar[r]^{s}
      & M_3 \ar[r] & 0
    }
  \end{equation*}
  is a commutative diagram in $\mathcal{A}$ with exact rows. Then we have
  \begin{equation*}
    t(f_2) = t(f_1) + t(f_3).
  \end{equation*}
\end{proposition}

We note that the internal Hom functor of an exact module category is exact in each variable. Lemma~\ref{lem:character-additive} is the case where $\mathcal{A} = \mathcal{M}$, $\mathcal{B} = \mathcal{C}$, $\langle -,- \rangle = [-, -]$, $d = \pi_{\mathcal{M}}$, $e = \trace_{\mathcal{M}}$ and $f_i = \id_{M_i}$ for $i = 1, 2, 3$. If we consider $d = \icoev_{\unitobj, (-)}$ instead of $d = \pi_{\mathcal{M}}$, then we obtain the additivity of the pivotal trace with respect to exact sequences.

\begin{proof}
  By the assumption on $\langle -, - \rangle$, we obtain the following commutative diagram with exact rows and exact columns:
  \begin{equation*}
    \xymatrix@R=16pt@C=32pt{
      & 0 \ar[d] & 0 \ar[d] & 0 \ar[d] \\ 
      0 \ar[r]
      & \langle M_3, M_1 \rangle \ar[r]^{\langle \id,r \rangle} \ar[d]_{\langle s,\id \rangle}
      & \langle M_3, M_2 \rangle \ar[r]^{\langle \id,s \rangle} \ar[d]_{\langle s,\id \rangle}
      & \langle M_3, M_3 \rangle \ar[r] \ar[d]_{\langle s,\id \rangle}
      & 0 \\
      0 \ar[r]
      & \langle M_2, M_1 \rangle \ar[r]^{\langle \id,r \rangle} \ar[d]_{\langle r,\id \rangle}
      & \langle M_2, M_2 \rangle \ar[r]^{\langle \id,s \rangle} \ar[d]_{\langle r,\id \rangle}
      & \langle M_2, M_3 \rangle \ar[r] \ar[d]_{\langle r,\id \rangle}
      & 0 \\
      0 \ar[r]
      & \langle M_1, M_1 \rangle \ar[r]^{\langle \id,r \rangle} \ar[d]
      & \langle M_1, M_2 \rangle \ar[r]^{\langle \id,s \rangle} \ar[d]
      & \langle M_1, M_3 \rangle \ar[r] \ar[d]
      & 0 \\
     & 0 & 0 & 0
    }
  \end{equation*}
  We set $K = \Ker(\langle r, s \rangle: \langle M_2, M_2 \rangle \to \langle M_1, M_3 \rangle)$. Then we have $K = I_1 + I_2$, where
  \begin{align*}
    I_1 & = \Img(\langle \id_{M_2}, r \rangle: \langle M_2, M_1 \rangle \to \langle M_2, M_2 \rangle), \\
    I_2 & = \Img(\langle s, \id_{M_2} \rangle: \langle M_3, M_2 \rangle \to \langle M_2, M_2 \rangle).
  \end{align*}
  Moreover, there are morphisms $p_i: K \to \langle M_i, M_i \rangle$ ($i = 1, 3$) such that
  \begin{equation}
    \label{eq:lem-additivity-p}
    \langle \id_{M_1}, r \rangle \circ p_1 = \langle r, \id_{M_2} \rangle
    \quad \text{and} \quad
    \langle s, \id_{M_3} \rangle \circ p_3 = \langle \id_{M_2}, s \rangle.
  \end{equation}
  These claims are checked by chasing the diagram. See \cite[Lemma 2.5.1]{MR2803849} for the detail of the verification, since the proof up to here is completely same.

  For simplicity of notation, we set $d_i = \langle \id_{M_i}, f_i \rangle \circ d(M_i)$. By $f_2 s = s f_3$ and the dinaturality of $d$, we have
  \begin{align}
    \label{eq:lem-additivity-1}
    \langle \id_{M_2}, s \rangle \circ d_2 & = \langle s, \id_{M_3} \rangle \circ d_3, \\
    \label{eq:lem-additivity-2}
    \langle r, \id_{M_2} \rangle \circ d_2 & = \langle \id_{M_1}, r \rangle \circ d_3.
  \end{align}
  Thus $\langle r, s \rangle \circ d_2 = \langle r, \id \rangle \circ \langle \id, s \rangle \circ d_2 = \langle r, \id \rangle \circ \langle s, \id \rangle \circ d_3 = 0$, that is, $\Img(d_2) \subset K$. Hence the following morphism is defined:
  \begin{equation}
    \label{eq:lem-additivity-trace-Gamma}
    \Gamma
    := e(M_1) \circ p_1 \circ d_2
    + e(M_3) \circ p_3 \circ d_2.
  \end{equation}
  We first show that $\Gamma = t(f_1) + t(f_3)$. By~\eqref{eq:lem-additivity-p} and~\eqref{eq:lem-additivity-2}, we have
  \begin{equation*}
    \langle \id_{M_1}, r \rangle \circ p_1 \circ d_2
    = \langle r, \id_{M_2} \rangle \circ d_2
    = \langle \id_{M_1}, r \rangle \circ d_1.
  \end{equation*}
  Since $r$ is monic, so is $\langle \id_{M_1}, r \rangle$. Thus we have $p_1 \circ d_2 = d_1$. Therefore the first term of \eqref{eq:lem-additivity-trace-Gamma} is $t(f_1)$. Similarly, we have
  \begin{equation*}
    \langle s, \id_{M_3} \rangle \circ p_3 \circ d_2
    = \langle \id_{M_2}, s \rangle \circ d_2
    = \langle s, \id_{M_1} \rangle \circ d_3
  \end{equation*}
  by~\eqref{eq:lem-additivity-p} and~\eqref{eq:lem-additivity-1}, and thus $p_3 \circ d_2 = d_3$. From this, we see that the second term of \eqref{eq:lem-additivity-trace-Gamma} is $t(f_3)$. Thus the claim follows.

  To complete the proof, we show $\Gamma = t(f_2)$. To see this, we remark
  \begin{gather*}
    \langle \id_{M_1}, r \rangle \circ p_1 \circ \langle s, \id_{M_2} \rangle = \langle r, \id_{M_2} \rangle \circ \langle s, \id_{M_2} \rangle = 0, \\
    \langle s, \id_{M_3} \rangle \circ p_3 \circ \langle \id_{M_2}, r \rangle = \langle \id_{M_2}, s \rangle \circ \langle \id_{M_2}, r \rangle = 0
  \end{gather*}
  by~\eqref{eq:lem-additivity-p}. Since both $\langle \id_{M_1}, r \rangle$ and $\langle s, \id_{M_3} \rangle$ are monic, $p_1 \circ \langle s, \id_{M_2} \rangle$ and $p_3 \circ \langle \id_{M_2}, r \rangle$ are zero morphisms.
  Set $\Gamma' = e(M_1) \circ p_1 + e(M_3) \circ p_3$. We have
  \begin{align*}
    \Gamma' \circ \langle s, \id_{M_2} \rangle
    = e(M_2) \circ \langle s, \id_{M_2} \rangle
    \quad \text{and} \quad
    \Gamma' \circ \langle \id_{M_2}, r \rangle
    = e(M_2) \circ \langle \id_{M_2}, r \rangle
  \end{align*}
  by the dinaturality of $e$. These equation imply that $\Gamma' = e(M_2)$ on $K = I_1 + I_2$. Since $\Img(d_3) \subset K$, we conclude that $\Gamma = t(f_2)$. The proof is done.
\end{proof}


\begin{thebibliography}{GKPM11}

\bibitem[AM07]{MR2331768}
Nicol{\'a}s Andruskiewitsch and Juan~Mart{\'{\i}}n Mombelli.
\newblock On module categories over finite-dimensional {H}opf algebras.
\newblock {\em J. Algebra}, 314(1):383--418, 2007.

\bibitem[AN13]{MR3045342}
Yusuke Arike and Kiyokazu Nagatomo.
\newblock Some remarks on pseudo-trace functions for orbifold models associated
  with symplectic fermions.
\newblock {\em Internat. J. Math.}, 24(2):1350008, 29, 2013.

\bibitem[Ari10a]{MR2722373}
Yusuke Arike.
\newblock A construction of symmetric linear functions on the restricted
  quantum group {$\overline U_q({\rm sl}_2)$}.
\newblock {\em Osaka J. Math.}, 47(2):535--557, 2010.

\bibitem[Ari10b]{MR2663653}
Yusuke Arike.
\newblock Some remarks on symmetric linear functions and pseudotrace maps.
\newblock {\em Proc. Japan Acad. Ser. A Math. Sci.}, 86(7):119--124, 2010.

\bibitem[BK01]{MR1797619}
Bojko Bakalov and Alexander Kirillov, Jr.
\newblock {\em Lectures on tensor categories and modular functors}, volume~21
  of {\em University Lecture Series}.
\newblock American Mathematical Society, Providence, RI, 2001.

\bibitem[BKLT00]{MR1759389}
Yuri Bespalov, Thomas Kerler, Volodymyr Lyubashenko, and Vladimir Turaev.
\newblock Integrals for braided {H}opf algebras.
\newblock {\em J. Pure Appl. Algebra}, 148(2):113--164, 2000.

\bibitem[BLV11]{MR2793022}
Alain Brugui{\`e}res, Steve Lack, and Alexis Virelizier.
\newblock Hopf monads on monoidal categories.
\newblock {\em Adv. Math.}, 227(2):745--800, 2011.

\bibitem[BV07]{MR2355605}
Alain Brugui{\`e}res and Alexis Virelizier.
\newblock Hopf monads.
\newblock {\em Adv. Math.}, 215(2):679--733, 2007.

\bibitem[BV12]{MR2869176}
Alain Brugui{\`e}res and Alexis Virelizier.
\newblock Quantum double of {H}opf monads and categorical centers.
\newblock {\em Trans. Amer. Math. Soc.}, 364(3):1225--1279, 2012.

\bibitem[CW08]{MR2464107}
Miriam Cohen and Sara Westreich.
\newblock Characters and a {V}erlinde-type formula for symmetric {H}opf
  algebras.
\newblock {\em J. Algebra}, 320(12):4300--4316, 2008.

\bibitem[DSS13]{2013arXiv1312.7188D}
C.~L. {Douglas}, C.~{Schommer-Pries}, and N.~{Snyder}.
\newblock {Dualizable tensor categories}.
\newblock {\tt arXiv:1312.7188}, 2013.

\bibitem[DSS14]{2014arXiv1406.4204D}
C.~L. {Douglas}, C.~{Schommer-Pries}, and N.~{Snyder}.
\newblock {The balanced tensor product of module categories}.
\newblock {\tt arXiv:1406.4204}, 2014.

\bibitem[EG17]{MR3600085}
Pavel Etingof and Shlomo Gelaki.
\newblock Exact sequences of tensor categories with respect to a module
  category.
\newblock {\em Adv. Math.}, 308:1187--1208, 2017.

\bibitem[EGNO15]{MR3242743}
Pavel Etingof, Shlomo Gelaki, Dmitri Nikshych, and Victor Ostrik.
\newblock {\em Tensor categories}, volume 205 of {\em Mathematical Surveys and
  Monographs}.
\newblock American Mathematical Society, Providence, RI, 2015.

\bibitem[ENO04]{MR2097289}
Pavel Etingof, Dmitri Nikshych, and Viktor Ostrik.
\newblock An analogue of {R}adford's {$S^4$} formula for finite tensor
  categories.
\newblock {\em Int. Math. Res. Not.}, (54):2915--2933, 2004.

\bibitem[ENO05]{MR2183279}
Pavel Etingof, Dmitri Nikshych, and Viktor Ostrik.
\newblock On fusion categories.
\newblock {\em Ann. of Math. (2)}, 162(2):581--642, 2005.

\bibitem[EO04]{MR2119143}
Pavel Etingof and Viktor Ostrik.
\newblock Finite tensor categories.
\newblock {\em Mosc. Math. J.}, 4(3):627--654, 782--783, 2004.

\bibitem[FGR17]{2017arXiv170201086F}
V.~{Farsad}, A.~M. {Gainutdinov}, and I.~{Runkel}.
\newblock {SL(2,Z)-action for ribbon quasi-Hopf algebras}.
\newblock {\tt arXiv:1702.01086}, 2017.

\bibitem[FSS16]{2016arXiv161204561F}
J.~{Fuchs}, G.~{Schaumann}, and C.~{Schweigert}.
\newblock {Eilenberg-Watts calculus for finite categories and a bimodule
  Radford $S^4$ theorem}.
\newblock {\tt arXiv:1612.04561}, 2016.

\bibitem[GKPM11]{MR2803849}
Nathan Geer, Jonathan Kujawa, and Bertrand Patureau-Mirand.
\newblock Generalized trace and modified dimension functions on ribbon
  categories.
\newblock {\em Selecta Math. (N.S.)}, 17(2):453--504, 2011.

\bibitem[GR16]{2016arXiv160504448G}
A.~M. {Gainutdinov} and I.~{Runkel}.
\newblock {The non-semisimple Verlinde formula and pseudo-trace functions}.
\newblock {\tt arXiv:1605.04448}, 2016.

\bibitem[GR17]{2017arXiv170300150G}
A.~M. {Gainutdinov} and I.~{Runkel}.
\newblock {Projective objects and the modified trace in factorisable finite
  tensor categories}.
\newblock {\tt arXiv:1703.00150}, 2017.

\bibitem[KL01]{MR1862634}
Thomas Kerler and Volodymyr~V. Lyubashenko.
\newblock {\em Non-semisimple topological quantum field theories for
  3-manifolds with corners}, volume 1765 of {\em Lecture Notes in Mathematics}.
\newblock Springer-Verlag, Berlin, 2001.

\bibitem[Koshi16]{Koshitani2016}
  S. Koshitani, Endo-trivial modules for finite groups with dihedral Sylow $2$-groups,
  RIMS K\^oky\^uroku 2003 (2016) 128--132.

\bibitem[LM94]{MR1280590}
Volodimir Lyubashenko and Shahn Majid.
\newblock Braided groups and quantum {F}ourier transform.
\newblock {\em J. Algebra}, 166(3):506--528, 1994.

\bibitem[LMSS17]{2017arXiv170704032L}
S.~{Lentner}, S.~N. {Mierach}, C.~{Schweigert}, and Y.~{Sommerhaeuser}.
\newblock {Hochschild Cohomology and the Modular Group}.
\newblock {\em arXiv:1707.04032}, 2017.

\bibitem[Lyu95a]{MR1324034}
V.~Lyubashenko.
\newblock Modular transformations for tensor categories.
\newblock {\em J. Pure Appl. Algebra}, 98(3):279--327, 1995.

\bibitem[Lyu95b]{MR1352517}
Volodimir Lyubashenko.
\newblock Modular properties of ribbon abelian categories.
\newblock In {\em Proceedings of the 2nd {G}auss {S}ymposium. {C}onference {A}:
  {M}athematics and {T}heoretical {P}hysics ({M}unich, 1993)}, Sympos.
  Gaussiana, pages 529--579, Berlin, 1995. de Gruyter.

\bibitem[Lyu95c]{MR1354257}
Volodymyr~V. Lyubashenko.
\newblock Invariants of {$3$}-manifolds and projective representations of
  mapping class groups via quantum groups at roots of unity.
\newblock {\em Comm. Math. Phys.}, 172(3):467--516, 1995.

\bibitem[Miy04]{MR2046807}
Masahiko Miyamoto.
\newblock Modular invariance of vertex operator algebras satisfying
  {$C_2$}-cofiniteness.
\newblock {\em Duke Math. J.}, 122(1):51--91, 2004.

\bibitem[ML98]{MR1712872}
Saunders Mac~Lane.
\newblock {\em Categories for the working mathematician}, volume~5 of {\em
  Graduate Texts in Mathematics}.
\newblock Springer-Verlag, New York, second edition, 1998.

\bibitem[Mon93]{MR1243637}
Susan Montgomery.
\newblock {\em Hopf algebras and their actions on rings}, volume~82 of {\em
  CBMS Regional Conference Series in Mathematics}.
\newblock Published for the Conference Board of the Mathematical Sciences,
  Washington, DC, 1993.

\bibitem[NS43]{MR0009024}
C.~Nesbitt and W.~M. Scott.
\newblock Some remarks on algebras over an algebraically closed field.
\newblock {\em Ann. of Math. (2)}, 44:534--553, 1943.

\bibitem[Oku81]{Okuyama81}
  T.~Okuyama, $\mathrm{Ext}^1(S, S)$ for a simple $k G$-module $S$,
  in: Proceedings of the Symposium ``Representations of Groups and Rings and Its applications'' (in Japanese), December 16--19 1981,
  Edited by S. Endo. pp.238--249.

\bibitem[{Ost}13]{2013arXiv1309.4822O}
V.~{Ostrik}.
\newblock {Pivotal fusion categories of rank 3 (with an Appendix written
  jointly with Dmitri Nikshych)}.
\newblock {\tt arXiv:1309.4822}, 2013.

\bibitem[Rad94]{MR1265853}
David~E. Radford.
\newblock The trace function and {H}opf algebras.
\newblock {\em J. Algebra}, 163(3):583--622, 1994.

\bibitem[Ros95]{MR1347919}
Alexander~L. Rosenberg.
\newblock {\em Noncommutative algebraic geometry and representations of
  quantized algebras}, volume 330 of {\em Mathematics and its Applications}.
\newblock Kluwer Academic Publishers Group, Dordrecht, 1995.

\bibitem[{Sak}17a]{2017arXiv170103799S}
T.~{Sakurai}.
\newblock {Central elements of the Jennings basis and certain Morita
  invariants}.
\newblock {\tt arXiv:1701.03799}, 2017.

\bibitem[Sak17b]{MR3656722}
Taro Sakurai.
\newblock A generalization of dual symmetry and reciprocity for symmetric
  algebras.
\newblock {\em J. Algebra}, 484:265--274, 2017.

\bibitem[Sch01]{MR1822847}
Peter Schauenburg.
\newblock The monoidal center construction and bimodules.
\newblock {\em J. Pure Appl. Algebra}, 158(2-3):325--346, 2001.

\bibitem[Sch16]{MR3441221}
Peter Schauenburg.
\newblock Computing higher {F}robenius-{S}chur indicators in fusion categories
  constructed from inclusions of finite groups.
\newblock {\em Pacific J. Math.}, 280(1):177--201, 2016.

\bibitem[{Shi}16]{2016arXiv160206534S}
K.~{Shimizu}.
\newblock {Non-degeneracy conditions for braided finite tensor categories}.
\newblock {\tt arXiv:1602.06534}, 2016.

\bibitem[{Shi}17a]{2017arXiv170709691S}
K.~{Shimizu}.
\newblock {Ribbon structures of the Drinfeld center}.
\newblock {\tt arXiv:1707.09691}, 2017.

\bibitem[Shi17b]{MR3631720}
Kenichi Shimizu.
\newblock The monoidal center and the character algebra.
\newblock {\em J. Pure Appl. Algebra}, 221(9):2338--2371, 2017.

\bibitem[Shi17c]{MR3632104}
Kenichi Shimizu.
\newblock On unimodular finite tensor categories.
\newblock {\em Int. Math. Res. Not. IMRN}, (1):277--322, 2017.

\bibitem[Shi17d]{MR3569179}
Kenichi Shimizu.
\newblock The relative modular object and {F}robenius extensions of finite
  {H}opf algebras.
\newblock {\em J. Algebra}, 471:75--112, 2017.

\bibitem[{Shi}17e]{2016arXiv170202425S}
K.~{Shimizu}.
\newblock {Integrals for finite tensor categories}
\newblock {\tt arXiv:1702.02425}, 2017.

\bibitem[Sut94]{MR1284788}
Ruedi Suter.
\newblock Modules over {$\mathfrak{U}_q(\mathfrak{sl}_2)$}.
\newblock {\em Comm. Math. Phys.}, 163(2):359--393, 1994.

\bibitem[SZ12]{MR2985696}
Yorck Sommerh{\"a}user and Yongchang Zhu.
\newblock Hopf algebras and congruence subgroups.
\newblock {\em Mem. Amer. Math. Soc.}, 219(1028):vi+134, 2012.

\end{thebibliography}
\def\cprime{$'$}

\end{document}